\documentclass{article}

\usepackage[utf8]{inputenc} 
\usepackage[T1]{fontenc}    
\usepackage{booktabs}       
\usepackage{nicefrac}       
\usepackage{microtype}      

\usepackage[round]{natbib}
\usepackage{amssymb,amsmath,amsthm,bbm}
\usepackage[margin=1in]{geometry}
\usepackage{verbatim,float,url,dsfont}
\usepackage{graphicx,subfigure,psfrag}
\usepackage{algorithm,algorithmic}
\usepackage{mathtools,enumitem}
\usepackage{multirow}

\newtheorem{theorem}{Theorem}
\newtheorem{lemma}{Lemma}
\newtheorem{corollary}{Corollary}
\newtheorem{proposition}{Proposition}
\theoremstyle{definition}
\newtheorem{remark}{Remark}
\newtheorem{definition}{Definition}
\newtheorem{assumption}{Assumption}

\def\R{\mathbb{R}}

\def\E{\mathbb{E}}
\def\P{\mathbb{P}}
\def\Cov{\mathrm{Cov}}
\def\Var{\mathrm{Var}}

\def\rank{\mathrm{rank}}

\def\hf{\hat{f}}

\def\hbeta{\hat{\beta}}

\def\cF{\mathcal{F}}
\def\cG{\mathcal{G}}

\makeatletter
\newcommand*\rel@kern[1]{\kern#1\dimexpr\macc@kerna}
\newcommand*\widebar[1]{%
  \begingroup
  \def\mathaccent##1##2{%
    \rel@kern{0.8}%
    \overline{\rel@kern{-0.8}\macc@nucleus\rel@kern{0.2}}%
    \rel@kern{-0.2}%
  }%
  \macc@depth\@ne
  \let\math@bgroup\@empty \let\math@egroup\macc@set@skewchar
  \mathsurround\z@ \frozen@everymath{\mathgroup\macc@group\relax}%
  \macc@set@skewchar\relax
  \let\mathaccentV\macc@nested@a
  \macc@nested@a\relax111{#1}%
  \endgroup
}
\makeatother

\newcommand{\<}{\langle}
\renewcommand{\>}{\rangle}

\def\om{\overline{m}}
\def\tbSigma{\tilde{\boldsymbol \Sigma}}

\def\sabs{\mbox{\tiny \rm abs}}

\def\sequi{\mbox{\tiny \rm equi}}
\def\slat{\mbox{\tiny \rm lat}}

\def\sop{\mbox{\tiny \rm op}}

\def\Proj{{\sf P}}

\def\oF{{\overline F}}
\def\eps{{\varepsilon}}

\def\id{{{I}}}
\def\bh{{h}}

\def\m1{{m^{(1)}}}
\def\m2{{m^{(2)}}}
\def\bm{{\boldsymbol m}}

\def\hf{\hat{f}}

\def\bE{{\boldsymbol E}}

\def\tB{\widetilde{B}}
\def\tQ{\widetilde{Q}}
\def\tX{\widetilde{X}}
\def\tbQ{\widetilde{\boldsymbol Q}}
\def\tbX{\widetilde{\boldsymbol X}}

\def\bDelta{{\boldsymbol{\Delta}}}

\def\tlambda{\tilde{\lambda}}

\def\bbeta{{{\beta}}}

\def\btheta{{\boldsymbol{\theta}}}

\def\bfe{{\boldsymbol{e}}}

\def\bSigma{{{\Sigma}}}

\def\tbw{\tilde{{\boldsymbol w}}}

\def\tbW{\tilde{{W}}}

\def\bC{{\boldsymbol{C}}}
\def\bQ{{{Q}}}

\def\bV{{\boldsymbol{V}}}
\def\bS{{\boldsymbol{S}}}

\def\bzero{{\mathbf 0}}

\def\cF{{\mathcal F}}
\def\cG{{\mathcal G}}

\def\tZ{\tilde{Z}}
\def\tz{\tilde{z}}
\def\tbz{\tilde{\boldsymbol z}}

\def\op{\mbox{\tiny\rm op}}

\def\naturals{{\mathbb N}}
\def\reals{{\mathbb R}}
\def\disk{{\mathbb D}}
\def\complex{{\mathbb C}}

\def\normal{{ N}}

\def\sT{{T}}

\def\bv{{\boldsymbol{v}}}
\def\bz{{{z}}}
\def\bx{{\boldsymbol{x}}}

\def\bA{{A}}
\def\bB{{B}}

\def\bU{{U}}

\def\etab{{\eta}}

\def\de{{\rm d}}
\def\tbA{\tilde{{A}}}
\def\tbB{\tilde{\boldsymbol{B}}}
\def\tbX{\tilde{{X}}}
\def\bX{{X}}

\def\bW{{W}}
\def\prob{{\mathbb P}}
\def\E{{\mathbb E}}

\def\<{\langle}
\def\>{\rangle}
\def\Tr{{\sf Tr}}

\def\rank{{\rm rank}}

\def\cv{{*}}

\def\bw{{{w}}}

\def\bu{{\boldsymbol{u}}}
\def\b0{{\boldsymbol{0}}}

\def\hG{\widehat{G}}

\def\Var{{\rm Var}}

\def\bfone{{\boldsymbol 1}}

\def\bF{{F}}
\def\bG{{G}}
\def\bR{{\boldsymbol R}}

\def\br{{\boldsymbol r}}

\def\bfzero{\boldsymbol{0}}

\DeclareSymbolFont{rsfs}{U}{rsfs}{m}{n}
\DeclareSymbolFontAlphabet{\mathscrsfs}{rsfs}
\def\cuB{\mathscrsfs{B}}
\def\cuV{\mathscrsfs{V}}
\def\cuE{\mathscrsfs{E}}

\def\bbD{{\mathbb{D}}}
\def\hH{\widehat{H}}
\def\om{\overline{m}}
\def\oX{\overline{X}}

\def\tx{\tilde{x}}
\def\tu{\tilde{u}}

\renewcommand\Re{{\operatorname{Re}}}
\renewcommand\Im{{\operatorname{Im}}}

\usepackage{floatpag}
\floatpagestyle{empty}
\def\hSigma{\hat\Sigma}
\def\snr{\mathrm{SNR}}
\def\one{\mathds{1}}
\def\cv{\mathrm{CV}}
\def\gcv{\mathrm{GCV}}

\title{Surprises in High-Dimensional Ridgeless Least Squares Interpolation}  
\author{Trevor Hastie \and Andrea Montanari$^*$ \and Saharon Rosset \and 
  Ryan J. Tibshirani\thanks{Corresponding authors.}}  
\date{}

\begin{document}
\maketitle

\begin{abstract}
Interpolators---estimators that achieve zero training error---have attracted
growing attention in machine learning, mainly because state-of-the art neural
networks appear to be models of this type. In this paper, we study minimum
$\ell_2$ norm (``ridgeless'') interpolation in high-dimensional least squares 
regression. We consider two different models for the feature distribution: a
linear model, where the feature vectors $x_i \in \R^p$ are obtained by
applying a linear transform to a vector of i.i.d.\ entries, $x_i = \Sigma^{1/2}
z_i$ (with $z_i \in \R^p$); and a nonlinear model, where the feature vectors
are obtained by passing the input through a random one-layer neural network,
$x_i =  \varphi(W z_i)$ (with $z_i \in \R^d$, $W \in \R^{p \times d}$ a
matrix of i.i.d.\ entries, and $\varphi$ an activation function acting
componentwise on $W z_i$). We recover---in a precise quantitative way---several
phenomena that have been observed in large-scale neural networks and kernel
machines, including the ``double descent'' behavior of the prediction risk, and
the potential benefits of overparametrization.  
\end{abstract}

\section{Introduction}

Modern deep learning models involve a huge number of parameters.  In many
applications, current practice suggests that we should design
the network to be sufficiently complex so that the model (as trained, typically,
by gradient descent) interpolates the data, i.e., achieves zero training
error. Indeed, in a thought-provoking experiment, \citet{zhang2016understanding}
showed that state-of-the-art deep neural network architectures are complex
enough that they can be trained to interpolate the data even when the actual
labels are replaced by entirely random ones.  

Despite their enormous complexity, deep neural networks are frequently observed
to generalize well in practice. At first sight, this seems to defy conventional
statistical wisdom: interpolation (vanishing training error) is commonly taken to
be a proxy for overfitting, poor generalization (large gap between training
and test error), and hence large test error.  In an insightful series of papers,
\citet{belkin2018understand,belkin2018does,belkin2018reconciling} pointed out
that these concepts are in general distinct, and interpolation does not
contradict generalization. For example, recent work has investigated interpolation ---via
kernel ridge regression--- in reproducing kernel Hilbert spaces \citep{liang2018just,ghorbani2019linearized}.
While in low dimension a positive regularization is needed to achieve good interpolation, in certain high
dimensional settings interpolation can be nearly optimal. 

In this paper, we investigate these phenomena in the context of simple linear models.
We assume to be given i.i.d. data $(y_i,x_i)$, $i\le n$, with $x_i\in\reals^p$ a feature vector and
$y_i\in \reals$ a response variable. These are distributed according to the model (see Section \ref{sec:prelim}
for further definitions)
\begin{gather}
\label{eq:data_x}
(x_i, \epsilon_i) \sim P_x \times P_\epsilon, \quad i=1,\ldots,n, \\ 
\label{eq:data_y}
y_i = x_i^T \beta + \epsilon_i, \quad i=1,\ldots,n,
\end{gather}
where, $P_x$ is a 
distribution on $\R^p$ such that $\E(x_i)=0$, $\Cov(x_i)=\Sigma$, and
$P_\epsilon$ is a distribution on $\R$ such that $\E(\epsilon_i)=0$,
$\Var(\epsilon_i)=\sigma^2$.

We estimate $\beta$ by linear regression.
Since our focus is on the overparametrized regime $p>n$,
the usual least square objective does not have a unique minimizer, and needs to be regularized.
We consider two approaches:  min-norm regression, which estimates $\beta$ by the
least squares solution with minimum $\ell_2$ norm; and ridge regression, which penalizes a coefficients
vector $\beta$ by its $\ell_2$ norm square $\|\beta\|_2^2$. We denote these estimates by $\hbeta$ and $\hbeta_{\lambda}$
($\lambda$ being the regularization parameter), and note that $\lim_{\lambda\to 0}\hbeta_{\lambda}=\hbeta$.
If the design matrix has full row rank, which is generically the case for $p>n$, the min-norm estimator is an interpolator,
namely $x_i^T\hbeta=y_i$ for all $i\le n$.
In order to evaluate these methods, we will study the prediction risk at a new (unseen) test point $(y_0,x_0)$.

We study  the model \eqref{eq:data_y}  in the proportional 
regime  $p\asymp n$,  with a special focus on the overparametrized
case $p>n$.
Our main contribution is to show that, by considering different choices of the features distribution $P_x$, we
 can reproduce a  number of statistically interesting phenomena that have emerged in the context of deep learning.

From a technical perspective, our main results are: Theorems \ref{thm:risk_gen} and \ref{thm:risk_ridge},
which assume the linear model
$x_i=\Sigma^{1/2} z_i$ with $z_i$ a vector with independent  coordinates;
and Theorem \ref{cor:purely_nonlinear}, which assumes a  nonlinear model
$x_i = \varphi(Wz_i)$ with $z_i\sim N(0,I_d)$. While the linear model
has already a attracted significant
amount of work (see Section \ref{sec:Related} for an overview), Theorems \ref{thm:risk_gen} and  \ref{thm:risk_ridge}
provide a more accurate non-asymptotic approximation of the prediction risk, as
compared to available results in the literature.

The prediction risk depends on the  geometry of the pair $(\Sigma,\beta)$. We consider
a few different choices for this geometry, which are broadly motivated by our objective to understand
overparametrized models, and specialize our formulas to these special cases:
\begin{enumerate}
\item \emph{Isotropic features.} This is the simplest case, in which  $\Sigma=I_p$ and therefore
  ---as we will see--- the asymptotic risk depends on $\beta$ only through its norm $\|\beta\|_2$. This simple model
  captures some interesting features of overparametrization, but misses others.
  
  We first consider a well specified case in which $x_i \in\reals^p$ and we regress against $x_i$. We then pass to a  misspecified case,
 in which  the model \eqref{eq:data_y} holds for covariates  $x_i\in\reals^{p+q}$, but we regress only against the first $p$ covariates.  
\item \emph{Latent space features.} In the overparametrized regime, it is natural to
  assume that both the covariates $x_i$, and the coefficients vector $\beta$ lie close to a low-dimensional
  subspace.   In order to model this property, we assume $\Sigma = WW^T+I$, with $W\in\reals^{p\times d}$, $d\ll p$,
  and $\beta$ lies in the span of the columns of $W$. Interestingly, this
  model reproduces many phenomena observed in more complex nonlinear models, and has a more direct
  connection to neural networks.
\item \emph{Nonlinear model.} In all of the previous cases, the distribution of $x_i$  is of the form $x_i=\Sigma^{1/2}z_i$
  where $z_i$ is a vector with independent coordinates. In order to test the generality of our results, we consider a
  model in which $x_i$ is obtained by passing  $z_i\sim\normal(0,I_d)$ through a one-layer neural net with random first layer weights.
\end{enumerate}
We will summarize our results for these four examples in the next subsection.

A skeptical reader might ask what linear models have to do with neural networks.
We emphasize that linear models provide more than a simple analogy, and a recent line of work outlines
a concrete connection between the two settings
\citep{jacot2018neural,du2018gradient,du2018gradient2,allen2018convergence,chizat2018note}.
We will discuss this connection in Section \ref{sec:Connection}.

\subsection{Summary of results}

As mentioned above, we analyze the out-of-sample prediction risk of the minimum  
$\ell_2$ norm (or min-norm, for short) least squares estimator, and of ridge-regularized
least squares.

We denote by $\gamma:=p/n \in (0,\infty)$ the overparametrization ratio.
When $\gamma < 1$, we call the problem \emph{underparametrized}, and 
when $\gamma > 1$, we call it \emph{overparametrized}.
Our most general results for the linear model (Theorem \ref{thm:risk_gen} and \ref{thm:risk_ridge})
apply to a non-asymptotic setting in which $n,p$ are finite, and provide a deterministic approximation of
the risk with error bounds that are uniform in the distribution of the data.
We will occasionally consider a simpler asymptotic scenario in which both $p$ and $n$ diverge
with $p/n\to\gamma$.

We assume the model \eqref{eq:data_y} and denote by $\snr= \|\beta\|_2^2/\sigma^2$ the
signal-to-noise ratio. We refer to Figure \ref{fig:summary} for supporting plots of the
asymptotic risk curves for different cases of interest.  

\begin{figure}[htbp]
\centering
\includegraphics[width=0.725\textwidth]{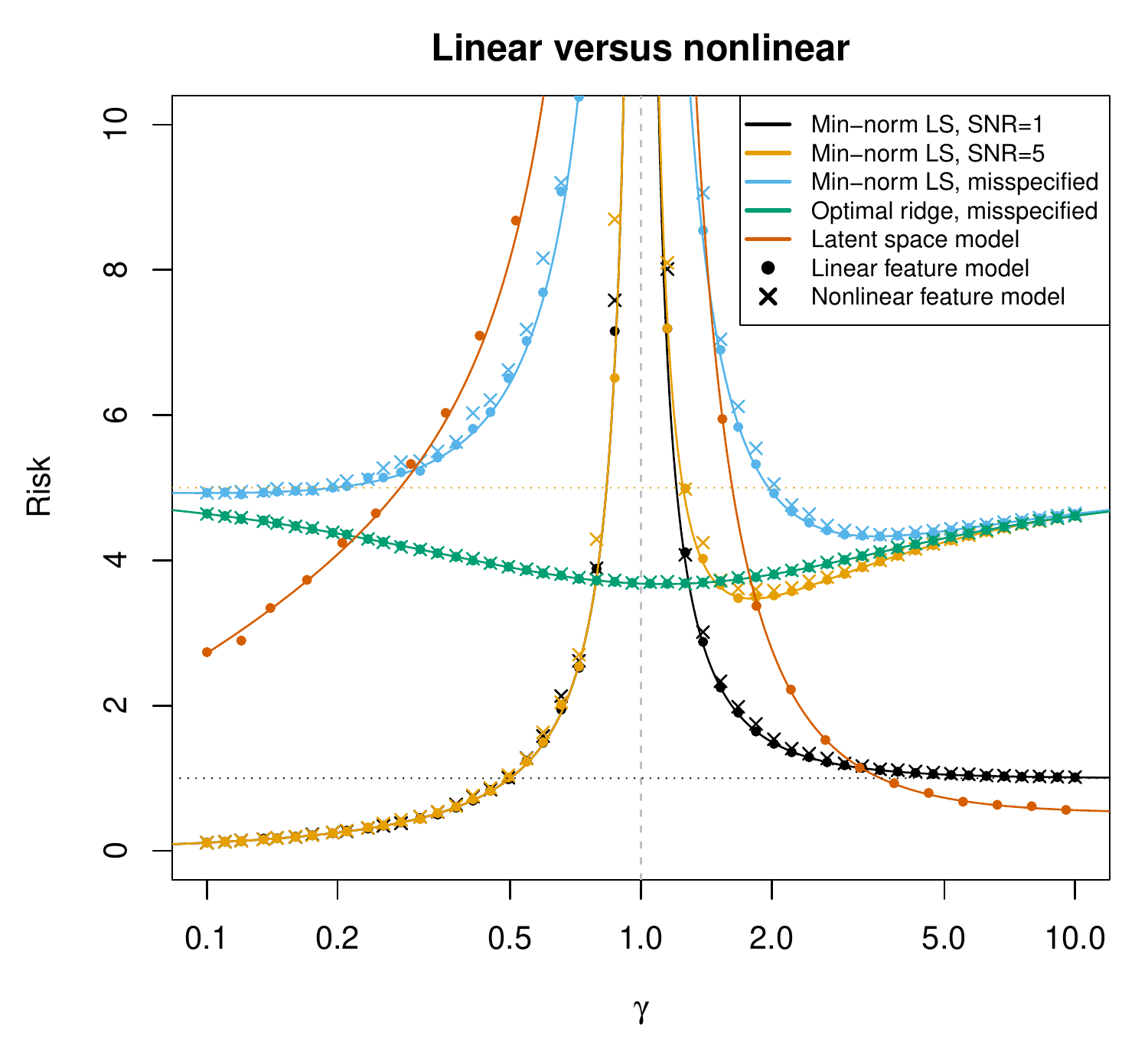} 
\caption{\footnotesize Asymptotic risk curves for the linear feature model, as a 
  function of the limiting aspect ratio $\gamma$.  Black and yellow: risks for min-norm least
  squares in the isotropic well specified model, for $\snr=1$ and $\snr=5$,
  respectively.  These two match for $\gamma<1$ but differ for $\gamma>1$.
  The null risks for $\snr=1$ and $\snr=5$ are marked by the dotted black and 
  yellow lines, respectively. Light blue: risk for  a misspecified model with
  significant approximation bias ($a=1.5$ in \eqref{eq:poly_decay}), when
  $\snr=5$. Green: optimally-tuned (equivalently, CV-tuned) ridge
  regression, in the same misspecified setup as for the light blue. Red: latent space model
  of Section \ref{sec:LatentSpace}, with $r=7$, $\sigma=0$. The
  points denote finite-sample risks, with $n=200$, $p=[\gamma n]$, 
  across various values of $\gamma$.  Meanwhile, the  `$\times$' points mark finite-sample risks for a 
  nonlinear feature model, with $n=200$, $p=[\gamma n]$, $d=100$, and 
  $X=\varphi(ZW^T)$, where $Z$ has i.i.d.\ $N(0,1)$ entries, $W$ has i.i.d.\
  $N(0,1/d)$ entries, and $\varphi(t) = a(|t|-b)$ is a ``purely nonlinear''
  activation function, for constants $a,b$. Theorem \ref{cor:purely_nonlinear} predicts that this
  nonlinear risk should converge to the linear risk with $p$ features
  (regardless of $d$).}  
\label{fig:summary} 
\end{figure}

Our main results are twofold: $(i)$~We show that by taking $\Sigma\neq I_p$, we can easily construct scenarios
in which the minimum of the risk is achieved in the overparamertized regime $p>n$; $(ii)$~We show that these findings
are robust to the details of the distribution of $(y_i,x_i)$.

As a preliminary remark, note that in the underparametrized regime  ($\gamma<1$), the min-norm estimator
coincides with the standard least squares estimator. Its  risk is purely
  variance (there is no bias), and does not depend on $\beta,\Sigma$ 
  (see Proposition \ref{thm:risk_lo}).  Interestingly,  the asymptotic risk diverges as we approach
  the interpolation boundary (as $\gamma \to 1$).   

   In contrast, in the overparametrized regime ($\gamma>1$), the risk is composed of both
  bias and variance\footnote{Note that in the overparametrized regime the bias is non-vanishing even in the interpolation
    limit $\lambda\to 0$. The reason is that the set of interpolators is an affine space of dimension $p-n$, and  the min-norm criterion selects one specific  interpolator which has norm smaller than $\beta$.},
  and generally depends on $\beta,\Sigma$ (see Theorem
  \ref{thm:risk_gen}).
  
We next highlight some concrete results for the four models discussed in the previous section
(unless explicitly said, we refer to the min-norm estimator).
\begin{description}
\item[Isotropic features.] The asymptotic risk depends on the coefficients vector only through its norm
  $\|\beta\|_2^2$ or, up to a scaling, on $\snr= \|\beta\|_2^2/\sigma^2$. 
  \begin{enumerate}
  \item If the model is \emph{well-specified}, we observe two different behaviors.
    For $\snr \leq 1$, the risk is decreasing for $\gamma \in (1,\infty)$.
   For $\snr > 1$, the risk has a {\it local} minimum on $\gamma \in
 (1,\infty)$.

 In either case, the risk approaches the null risk  as $\gamma \to \infty$, and achieves its global minimum
 in the underparametrized regime (see Section  \ref{sec:risk_iso}).
 \item If the model is \emph{misspecified}, when $\snr > 1$, the risk can attain its {\it
    global} minimum in the overparametrized regime $\gamma \in (1,\infty)$ (when there is strong enough
  approximation bias, see Section \ref{sec:poly_bias}). However, the risk is again increasing for $\gamma$ 
  large enough.
\item Optimally-tuned ridge regression uses a non-vanishing regularization $\lambda>0$,
  and dominates the min-norm least squares
  estimator in risk, across all values of $\gamma$ and $\snr$,  both the 
  well-specified and misspecified settings.  For a misspecified model,
  optimally-tuned ridge regression attains its global minimum around $\gamma=1$
  (see Section \ref{sec:ridge}).
  \item Optimal tuning of the ridge penalty can be achieved by leave-one-out cross-validation
  (see Theorem \ref{thm:risk_cv}).  
\end{enumerate}
\item[Anisotropic features.] In this case $\Sigma\neq \id$ and the risk depends on the geometry of $(\Sigma,\beta)$,
  and in particular on how $\beta$ aligns with the eigenvectors of $\Sigma$.
  \begin{enumerate}
  \item If the coefficients vector is equidistributed along the eigenvectors of $\Sigma$, the behavior is
    qualitatively similar to the isotropic case. This situation arises, for instance, if $\beta$ is itself random with a spherical prior.
  \item If $\beta$ is aligned with the top eigenvectors of $\Sigma$, the situation is qualitatively different.
    As an example we obtain an explicit formula for the asymptotic risk in the latent space model discussed above, see Figure
    \ref{fig:risk_latent_space_n} for an illustration.
    We find that, for natural choices of the model parameters, the risk is monotone decreasing in
  the overparametrized regime, and reaches its global minimum as $\gamma\to \infty$. This
  qualitative behavior matches the one observed for neural networks (see Section \ref{sec:LatentSpace}).
\item For the latent space model, we observe that, at large overparametrization, the minimum error is
  achieved as $\lambda\to 0$, i.e. by min-norm interpolators (see Section \ref{sec:LatentRidge},
  and Section \ref{sec:Related} for related work).
\end{enumerate}
\item[Nonlinear model.] Finally, we consider a nonlinear model in which $x_i = \varphi(Wz_i)$
  where $\varphi$ is a non-linear activation function applied componentwise, $W\in\reals^{p\times d}$,
  and $z_i\sim\normal(0,I_d)$. We compute the limiting risk of min-norm
  regression for a purely non-linear activation, and  show that this  matches the one for Gaussian $x_i$.
  This confirms that the results found for the case $x_i=\Sigma^{1/2}z_i$ with $z_i$ an i.i.d. vector
  are likely to hold in substantially greater generality
  (see Theorem \ref{cor:purely_nonlinear}). 
\end{description}

From a technical viewpoint, analysis of the isotropic covariates model is straightforward and relies on
standard random matrix theory results. However, we believe it provides useful insights.

In contrast, the results for general covariance and coefficients structure $(\Sigma,\beta)$ is
technically novel. We discuss related work in Section \ref{sec:Related}.
Our results for the  nonlinear
model are also novel. In this setting, we derive a new asymptotic result on resolvents of
certain block matrices, which may be of independent interest (see Lemma
\ref{thm:ResolventKernel}).

We next discuss some intuition behind and implications of our results. 

\paragraph{Bias and variance}  The shape of the asymptotic risk curve for
min-norm least squares is, of course, controlled by its components: bias and 
variance.  For fully specified models,  the bias increases with $\gamma$ in  the overparametrized regime,
which is intuitive.  When $p>n$, the min-norm least squares estimate of 
$\beta$ is constrained to lie the row space of $X$, the training feature
matrix.  This is a subspace of dimension $n$ lying in a feature space of
dimension $p$.  Thus as $p$ increases, so does the bias, since this row space
accounts for less and less of the ambient $p$-dimensional feature space.   

Meanwhile, we find that, in the overparametrized regime, the variance {\it decreases} with
$\gamma$.  This may seem counterintuitive at first, because it says, in a 
sense, that the min-norm least squares estimator becomes {\it more}
regularized as $p$ grows.  However, this too can be explained intuitively, as
follows.  As $p$ grows, the minimum $\ell_2$ norm least squares solution---i.e.,
the minimum $\ell_2$ norm solution to the linear system $Xb=y$, for a 
training feature matrix $X$ and response vector $y$---will generally have
decreasing $\ell_2$ norm.  Why?  Compare two such linear systems: in each, we 
are asking for the min-norm solution to a linear system with the same $y$, but
in one instance we are given more columns in $X$, so we can generally
decrease the components of $b$ (by distributing them over more columns), 
and achieve a smaller $\ell_2$ norm.  This can in fact be formalized
asymptotically, see Corollaries \ref{cor:norm_iso} and \ref{cor:norm_gen}.

\paragraph{Double descent} Recently, \citet{belkin2018reconciling} pointed out
a fascinating empirical trend where, for popular methods like neural networks
and random forests, we can see a {\it second} bias-variance tradeoff in the
out-of-sample prediction risk beyond the interpolation limit.  The risk curve
here resembles a traditional U-shape curve before the interpolation limit,
and then descends again beyond the interpolation limit, which these authors call
``double descent''.  A closely related phenomenon was found earlier by
\citet{spigler2018jamming}, who studied the ``jamming transition'' from
underparametrized to overparametrized neural networks. Our results formally
verify that this double descent phenomenon occurs even in the simple and
fundamental case of least squares regression.  The appearance of the second
descent in the risk, past the interpolation boundary ($\gamma=1$), is explained
by the fact that the variance decreases as $\gamma$ grows, as discussed above. 

In the misspecified case, the variance still decreases with $\gamma$ (for the
same reasons), but interestingly, the bias can now also decrease with $\gamma$,   
provided $\gamma$ is not too large (not too far past the interpolation
boundary). The intuition here is that in a misspecified model, some part of the 
true regression function is always unaccounted for, and adding features
generally improves our approximation capacity. As a consequence, the double 
descent phenomenon can be even more pronounced in the misspecified case
(depending on the strength of the approximation bias), and that the risk can
attain its global minimum past the interpolation limit.

Finally in the latent space model, we observe that the overall risk can be monotone decreasing
in the over parametrized regime, and attain its global minimum for large
overparametrization $\gamma\to \infty$ (after $p,n\to\infty$). In this case
we can write the design matrix as $X=ZW^T+U$, where $U$ is noise, and $Z$ is
the $n\times d$ matrix of latent covariates. Equivalently, the $i$-th column of $X$ (the $i$-th feature) takes the form $\tx_i = Zw_i+\tu_i$,
where $w_i$ is the $i$-th column of $W^T$ and $\tu_i$ is the $i$-th column of $U$. 
Therefore, each  new feature provides new information about the underlying low-dimensional latent variables $Z$.
As $p$ gets large, ridge regression with respect to the feature matrix $X$ approximates increasingly well a
ridge regression with respect to the latent variables $Z$.

\paragraph{Interpolation versus regularization} The min-norm least squares 
estimator can be seen as the limit of ridge regression as the tuning parameter 
tends to zero. It is also the convergence point of gradient descent run on the
least squares loss.  We would not in general expect the best-predicting ridge
solution to be at the end of its regularization path. Our results,
comparing min-norm least squares to optimally-tuned ridge regression, 
show that (asymptotically) this is never the case, when $\beta$ is
incoherent with respect to the eigenvectors of $\Sigma$. This is for instance the case
when $\Sigma=I_p$, or $\beta$ is distributed according to a spherically symmetric prior.
In contrast, \cite{kobak2020optimal} recently pointed out that ---when $\beta$ is aligned with the leading eigenvectors of
$\Sigma$--- min-norm regression can have optimal  risk (i.e. the optimal regularization vanishes). 
We show that this is indeed the case in the latent space model mentioned above:
this provides indeed an extremely simple example of a phenomenon that has been observed in the past
for kernel methods \cite{liang2018just}.

As mentioned above, early-stopped gradient descent is known to
be closely connected to ridge regularization, see, e.g., \citet{ali2019continuous} which proves
a tight coupling between the two (see Section \ref{sec:Related} for further related work).

In practice, of course, we would not have access to the optimal tuning
parameter for ridge (optimal stopping for gradient descent), and we would rely
on, e.g., cross-validation (CV).  Our theory shows that for ridge regression, CV
tuning is asymptotically equivalent to optimal tuning (and we would expect
the same results to carry over to gradient descent, but have not pursued this
formally).    

Historically, the debate between interpolation and regularization has been alive
for the last 30 or so years. Support vector machines find maximum-margin
decision boundaries, which often perform very well for problems where the Bayes
error is close to zero. But for less-separated classification tasks, one needs
to tune the cost parameter \citep{hastie2004entire}. Relatedly, in
classification, it is common to run boosting until the training error is
zero, and similar to the connection between gradient descent and $\ell_2$ 
regularization, the boosting path is tied to $\ell_1$ regularization
\citep{rosset2004boosting,tibshirani2015general}. Again, we now know that
boosting can overfit, and the number of boosting iterations should be treated 
as a tuning parameter.

\subsection{Connection to neural networks}
\label{sec:Connection}

As mentioned above, recent literature has pointed out a direct connection between
linear models and more complex models such as neural networks
\citep{jacot2018neural,du2018gradient,du2018gradient2,allen2018convergence,chizat2018note}.
Our work contributes to this line of work and in particular points at the important role played by \emph{universality},
a core concept in random matrix theory \cite{tao2012topics}.

Assume to be given data i.i.d. $(y_i,z_i)$, $i\le n$, $y_i\in\reals$, $z_i\in \reals^d$. 
Consider learning a neural network with parameters (weights) $\btheta\in\reals^p$,
$f(\,\cdot\,;\btheta):\reals^d\to\reals$, $z\mapsto f(z;\btheta)$. The specific form or
architecture of the network is not important for our discussion. However it is important to
emphasize that the form of $f$ need not be related to the actual distribution of $(y_i,\bz_i)$. 

In some settings, the number of parameters $p$ is so  
large that training effectively changes $\theta$ only by a small amount
with respect to a random initialization $\theta_0 \in \R^p$. It thus
makes sense to linearize the model around $\theta_0$. Further, supposing that 
the initialization is such that $f(z;\theta_0) \approx 0$, and letting $\theta =
\theta_0 + \beta$, we can approximate the statistical model $z\mapsto f(z;\btheta)$ by
\begin{equation}
\label{eq:linearized}
z\mapsto \nabla_\theta f(z;\theta_0)^T\beta\, .
\end{equation}
This model is still nonlinear in the input $z$, but is linear in the parameters $\beta$.
We are therefore led to consider a linear regression problem, with random
features $x_i=\nabla_\theta f(z_i;\theta_0)$, $i=1,\ldots,n$, of
high-dimensionality ($p$ much greater than $n$).  Notice that the features are
random because of the initialization $\theta_0$.  Further, since $p>n$, many vectors $\beta$
give rise to a model that interpolates the data.

The above scenario was made rigorous in a number of papers
\cite{jacot2018neural,du2018gradient,du2018gradient2,allen2018convergence,chizat2018note}. In
particular, \cite{chizat2018note} shows ---under some technical conditions--- that the linearization \eqref{eq:linearized}
can be accurate if the model is overparametrized
($p>n$), and closed under scalings (if $f(\,\cdot\, )$ is encoded by a neural network,
then $sf(\,\cdot\,)$ is also  a neural network for any $s\in\reals$). Under these conditions, there exists a
scaling of the network's parameters such that gradient-based training converges to a model that can be
approximated arbitrarily well by \eqref{eq:linearized}. Further,
under the linearization \eqref{eq:linearized},
gradient descent converges to the interpolator that minimizes \footnote{Understanding the bias induced by
  gradient-based algorithms on fully nonlinear models is a broadly open problem, which has attracted considerable attention recently, see e.g.  \cite{gunasekar2018characterizing,gunasekar2018implicit}.} the $\ell_2$ norm $\|\beta\|_2$
(see Proposition \ref{prop:grad_ls} below).

What are the \emph{statistical consequences} of these linearization results? In principle, one could proceed
as follows: assume $\{(y_i,z_i)\}_{i\le n}$ to be i.i.d. samples with a certain population distribution $P_{y,z}$,
and then study the behavior of minimum $\ell_2$ norm interpolator of the form \eqref{eq:linearized}
under this data model. Of course, carrying out such a study in any generality is an outstanding mathematical challenge:
we refer to Section \ref{sec:Related} for some pointers to recent progress in this direction
after a first appearance of this manuscript.

A different approach relies on a  \emph{universality hypothesis}.
Recall that $\{(y_i,z_i)\}_{i\le n}$ have joint distribution $P_{y,z}$, and $x_i= \nabla_\theta f(z_i;\theta_0)$ and assume these
  quantities to have zero mean\footnote{We further note that the assumption of zero mean is mainly for convenience, as a
non-zero mean of the covariates can be effectively compensated by an intercept.}. Universality posits that  the asymptotic  risk of ridge regression (or min-norm interpolation) does not change   if we replace the joint distribution $(x_i,z_i)$ by a Gaussian with the same covariance.   

While  universality is not expected to hold for any distribution of the data $(y_i,z_i)$,
and for any function $f$,  it has been confirmed in a number of examples of interest (see Section \ref{sec:Related}).
In fact, our analysis of the nonlinear model in Section \ref{sec:nonlinear} provides a concrete example in which universality
can be confirmed rigorously.

Under the universality hypothesis, we can assume $(z_i,x_i)$ to be jointly Gaussian.
If the true joint distribution of the data $(y_i,z_i)$ is also Gaussian, it is sufficient to study the data distribution
\begin{align*}
  &y_i = \beta^Tx_i+\epsilon_i,\;\;\;\;\;\; x_i\sim N(0,\Sigma)\, ,  \epsilon_i \sim N(0,\sigma^2) \,,
\end{align*}
where $\Sigma, \sigma$ are chosen to match the second order statistics of the original model.
This is a special case of our `linear model' \eqref{eq:data_y}.
Summarizing, replacing the original nonlinear model $y_i = f(z_i;\theta)$, with the
linear model  rests on two key steps. The first step is the linearization in
Eq.~\eqref{eq:linearized}: this has been proved to be a good approximation under
certain `lazy training' schemes \cite{chizat2018note}. The second step is the universality hypothesis,
which is quite common (albeit for simpler models) in random matrix theory \cite{tao2012topics}. 

Finally, although we believe our our results are relevant to 
understanding overparametrized neural networks, the concrete correspondence outlined
above only holds in a certain `lazy training' regime, in which network weights do not change much during training,
More generally, in a neural network, the feature representation and 
the regression function or classifier are learned {\it simultaneously}.  In both 
our linear and nonlinear model settings, the features $X$ are not learned, but
observed. Learning $X$ could significantly change some aspects of the behavior
of an interpolator. (See for instance Chapter 9 of \citet{goodfellow2016deep},
and also \citet{chizat2018note,zhang2019layers}, which emphasize the importance 
of learning the representation.)  

\subsection{Related work}
\label{sec:Related}
  
The present work connects to and is motivated by the recent interest in
interpolators in machine learning
\citep{belkin2018understand,belkin2018reconciling,liang2018just,
belkin2018does,geiger2019scaling}.
Several authors have argued that minimum $\ell_2$ norm least squares regression    
captures the basic behavior of deep neural networks, at least in early (lazy)
training 
\citep{jacot2018neural,du2018gradient,du2018gradient2,allen2018convergence,
zou2018stochastic,chizat2018note,lee2019wide}.
The connection between neural networks and kernel ridge regression arises when 
the number of hidden units diverges. The same limit was also studied (beyond the
linearized regime) in
\cite{mei2018mean,rotskoff2018neural,sirignano2018mean,chizat2018global}.

Ridge regression with random features has been studied in the past. 
\citet{dicker2016ridge}  considers a model in which the covariates
are isotropic Gaussian $x_i\sim N(0,I_p)$ and  computes the asymptotic risk of ridge
regression in the proportional asymptotics $p,n\to\infty$, with $p/n\to\gamma\in (0,\infty)$. 
\citet{dobriban2018high} generalize these results to $x_i = \Sigma^{1/2}z_i$,
where $z_i$ has independent entries with bounded $12$-th moment.

Recently, \citet{advani2017high} study the effect of early stopping and ridge regularization
in a model with isotropic Gaussian covariates $x_i\sim N(0,I_p)$, again focusing on the
proportional asymptotics $p,n\to\infty$, with $p/n\to\gamma\in (0,\infty)$. They
show that this simple model reproduces several phenomena observed in neural networks training.
The same model is reconsidered in concurrent work by  \citet{belkin2019two}, who obtain
exact results for the expected risk of min-norm regression, relying on the
jointly Gaussian distribution of $(y_i,x_i)$.
We contribute to this line of work by extending the analysis to general covariance structures
and to misspecified models. As we will see, these generalizations allow to produce examples
for which the global minimum of the risk is achieved in the overparametrized regime $\gamma>1$.

The importance of the relation between the coefficient vector $\beta$ and the eigenvectors
of $\Sigma$ was emphasized by \citet{kobak2020optimal} and \citet{bartlett2020benign}.
These papers point out ---under different asymptotic settings--- that $\lambda= 0+$ (i.e., min-norm
regression) can be optimal or nearly optimal. After a preprint of this paper
appeared, \citet{wu2020optimal} and \citet{richards2020asymptotics} generalized our earlier
results to cover the case in which $\beta$ is potentially aligned with $\Sigma$.
We review in further detail these important generalizations in Section \ref{sec:correlated}.
We contribute to this line of work by obtaining non-asymptotic approximations
for the risk, with explicit and nearly optimal error bounds. These hold under weaker assumptions
on the geometry of $(\Sigma,\beta)$ than the results of \cite{wu2020optimal,richards2020asymptotics}.

For the nonlinear model, the random matrix theory literature is much sparser,
and focuses on the related model of kernel random matrices, namely, symmetric
matrices of the form $K_{ij}=\varphi(z_i^Tz_j)$.  \citet{el2010spectrum} 
studied the spectrum of such matrices in a regime in which $\varphi$ can be
approximated by a linear function (for $i\neq j$) and hence the spectrum
converges to a rescaled Marchenko-Pastur law. This approximation does not hold 
for the regime of interest here, which was studied instead by
\citet{cheng2013spectrum} (who determined the limiting spectral distribution)
and \citet{fan2015spectral} (who characterized the extreme eigenvalues). The 
resulting eigenvalue distribution is the free convolution of a semicircle law
and a Marchenko-Pastur law. In the current paper, we must consider asymmetric 
(rectangular) matrices $x_{ij} =\varphi(w_j^Tz_i)$, whose singular value
distribution was recently computed by \citet{pennington2017nonlinear}, using the 
moment method. Unfortunately, the prediction variance depends on both the
singular values and vectors of this matrix. In order to address this issue, we
apply the leave-one out method of \citet{cheng2013spectrum} to compute the
asymptotics of the resolvent of a suitably extended matrix. We then extract the
information of interest from this matrix. 
After appearance of a preprint of this paper, \citet{mei2019generalization} extended the
results presented here, to obtain a complete characterization of the risk
for the non-linear random features model. 

Let us finally mention that the universality (or `invariance') phenomenon is quite common in random matrix
theory \cite{tao2012topics}. In the context of kernel inner product random matrices, it appears (somewhat implicitly)
in \cite{cheng2013spectrum} and (more explicitly) in \cite{fan2015spectral}.
After a first appearance of this manuscript, universality has been investigated in the context of neural networks
in several papers \cite{mei2019generalization,montanari2019generalization,gerace2020generalisation,goldt2020gaussian,hu2020universality,adlam2020neural}.

\subsection{Outline}

Section \ref{sec:prelim} provides important background. Sections
\ref{sec:isotropic}--\ref{sec:cv} consider the linear model case, focusing 
on isotropic features, correlated features, misspecified models, ridge
regularization, and cross-validation, respectively.  Section \ref{sec:nonlinear} 
covers the nonlinear model case.  Nearly all proofs are deferred until the
appendix. 

\section{Preliminaries}
\label{sec:prelim}

We describe our setup and gather a number of important preliminary results. 

\subsection{Data model and risk}

Assume we observe training data $(x_i,y_i) \in \R^p \times \R$, $i=1,\ldots,n$
from the model of Eqs.~\eqref{eq:data_x}, \eqref{eq:data_y}. We collect the responses in a vector $y \in
\R^n$, and the features in a matrix $X \in \R^{n\times p}$ (with rows $x_i \in
\R^p$, $i=1,\ldots,n$).  

Consider a test point $x_0 \sim P_x$, independent of the training data. For an
estimator \smash{$\hbeta$} (a function of the training data $X,y$), we define
its out-of-sample prediction risk (or simply, risk) as  
$$
R_X(\hbeta;\beta) = \E \big[ ( x_0^T \hbeta - x_0^T \beta )^2 \,|\, X \big]
= \E \big[ \| \hbeta - \beta \|_\Sigma^2 \,| \, X \big],
$$
where $\|x\|_\Sigma^2 = x^T \Sigma x$.  Note that our definition of risk is
conditional on $X$ (as emphasized by our notation $R_X$).  Note also that we
have the bias-variance decomposition   
$$
R_X(\hbeta;\beta) = 
\underbrace{\|\E(\hbeta|X) - \beta\|_\Sigma^2}_{B_X(\hbeta; \beta)} {}+{}     
\underbrace{\Tr[\Cov(\hbeta|X) \Sigma]}_{V_X(\hbeta; \beta)}.    
$$

\subsection{Ridgeless least squares}

We consider the minimum $\ell_2$ norm (min-norm) least squares regression
estimator, of $y$ on $X$, defined by  
\begin{equation}
\label{eq:min_ls}
\hbeta = \arg\min \Big\{ \|b\|_2 : b \; \text{minimizes} \;
\|y-Xb\|_2^2 \Big\}\, .
\end{equation}
This can be equivalently written as $\hbeta = (X^T X)^+ X^T y$, where   $(X^T X)^+$ is the pseudoinverse of $X^T X$.
An alternative name for \eqref{eq:min_ls} is the 
``ridgeless'' least squares estimator, motivated by the fact that  
\smash{$\hbeta = \lim_{\lambda \to 0^+} \hbeta_\lambda$}, where
\smash{$\hbeta_\lambda$} denotes the ridge regression estimator:
\begin{equation}
\label{eq:ridge}
\hbeta_\lambda = \arg\min_{b \in \R^p} \bigg\{ \frac{1}{n}\|y - Xb\|_2^2 +
\lambda \|b\|_2^2 \bigg\}  \, ,
\end{equation}
or, equivalently, 
$\hbeta_\lambda = (X^T X + n\lambda I)^{-1} X^T y$.

When $X$ has full column rank the  
min-norm least squares estimator reduces to \smash{$\hbeta = (X^T
  X)^{-1} X^T y$}, the usual least squares estimator.  When $X$ has rank $n$,
importantly, this estimator interpolates the training data: \smash{$y_i = x_i^T
  \hbeta$}, for $i=1,\ldots,n$.  

Lastly, the following is a well-known fact that connects the min-norm least
squares solution to gradient descent (as referenced in the introduction).  
\begin{proposition}
\label{prop:grad_ls}
Initialize $\beta^{(0)} = 0$, and consider running gradient descent on the least 
squares loss, yielding iterates
$$
\beta^{(k)} = \beta^{(k-1)} + t X^T (y-X\beta^{(k-1)}), 
\quad k=1,2,3,\ldots,
$$
where we take $0 < t \leq 1/\lambda_{\max}(X^T X)$ (and $\lambda_{\max}(X^T  
X)$ is the largest eigenvalue of $X^T X$).  Then \smash{$\lim_{k \to \infty}
  \beta^{(k)} = \hbeta$}, the min-norm least squares solution in
\eqref{eq:min_ls}.  
\end{proposition}

\begin{proof}
The choice of step size guarantees that \smash{$\beta^{(k)}$} converges to a 
least squares solution as $k \to \infty$, call it \smash{$\tilde\beta$}.  Note
that \smash{$\beta^{(k)}$}, $k=1,2,3,\ldots$ all lie in the row space of $X$;
therefore \smash{$\tilde\beta$} must also lie in the row space of $X$; and the
min-norm least squares solution \smash{$\hbeta$} is the unique least squares
solution with this property.  
\end{proof}

\subsection{Bias and variance}

We recall expressions for the bias and variance of the min-norm least squares
estimator, which are standard.  

\begin{lemma}
\label{lem:bias_var}
Under the model \eqref{eq:data_x}, \eqref{eq:data_y}, the min-norm least squares
estimator \eqref{eq:min_ls} has bias and variance 
$$
B_X(\hbeta; \beta) = \beta^T \Pi \Sigma \Pi \beta \quad \text{and} \quad   
V_X(\hbeta; \beta) = \frac{\sigma^2}{n} \Tr (\hSigma^+ \Sigma),  
$$
where \smash{$\hSigma=X^T X/n$} is the (uncentered) sample covariance of $X$,
and \smash{$\Pi = I-\hSigma^+\hSigma$} is the projection onto the null space of
$X$.  
\end{lemma}

\begin{proof}
As \smash{$\E(\hbeta|X) = (X^T X)^+ X^T X \beta = \hSigma^+ \hSigma
  \beta$} and \smash{$\Cov(\hbeta|X) = \sigma^2 (X^T X)^+ X^T X (X^T X)^+ =
  \sigma^2 \hSigma^+/n$}, the bias and variance expressions follow from plugging 
these into their respective definitions. 
\end{proof}

\subsection{Underparametrized asymptotics}

We consider an asymptotic setup where $n,p \to \infty$, in such a way that 
$p/n \to \gamma \in (0,\infty)$.  Recall that when $\gamma < 1$, we call the
problem underparametrized; when $\gamma > 1$, we call it overparametrized. 
Here, we recall the risk of the min-norm least squares estimator in the
underparametrized case.  The rest of this paper focuses on the overparametrized
case.   

The following is a known result in random matrix theory, and can be found in
Chapter 6 of \citet{serdobolskii2007multiparametric}.
It can also be found in the wireless communications literature, see Chapter 4 of
\citet{tulino2004random}.
%

\begin{proposition}
\label{thm:risk_lo}
Assume the model \eqref{eq:data_x}, \eqref{eq:data_y}, and assume $x \sim
P_x$ is of the form $x=\Sigma^{1/2} z$, where $z$ is a random vector with
i.i.d.\ entries that have zero mean, unit variance, and a finite 4th moment, and
$\Sigma$ is a (sequence of) deterministic positive definite matrix, such that 
$\lambda_{\min}(\Sigma) \geq c > 0$, for all $n,p$ and a constant $c$
(here $\lambda_{\min}(\Sigma)$ is the smallest eigenvalue of $\Sigma$).
Then as $n,p \to \infty$, such that $p/n \to \gamma < 1$, the
risk of the least squares estimator \eqref{eq:min_ls} satisfies, almost surely,    
$$
\lim_{n\to\infty}R_X(\hbeta; \beta) = \sigma^2 \frac{\gamma}{1-\gamma}.
$$
\end{proposition}

As it can be seen from the last proposition, in the underparametrized case the risk is just variance.
In contrast, in the overparametrized case,
the bias \smash{$B_X(\hbeta; \beta) = \beta^T \Pi
  \Sigma \Pi \beta$} is nonzero, because $\Pi$ is. This will be the focus of the next sections.

\section{Isotropic features}
\label{sec:isotropic}

We begin by considering the simpler case in which $\Sigma = I$. In this case
the limiting  bias is relatively straightforward to compute and depends on $\beta$ only through
$\|\beta\|_2^2$.  In Section \ref{sec:correlated}, we generalize our analysis and study the dependence
of the bias on the geometry of $\Sigma$ and $\beta$.

\subsection{Limiting bias}

As mentioned above, in the isotropic case the risk depends $\beta$ only on through
$r^2=\|\beta\|_2^2$.  To give some
intuition as to why this is true, consider the special case where $X$ has
i.i.d.\ entries from $N(0,1)$.  By rotational invariance, for any orthogonal $U
\in \R^{p \times p}$, the distribution of $X$ and $XU$ is the same.    
Thus
\begin{align*}
B_X(\hbeta; \beta) &= \beta^T \big(I - (X^T X)^+ X^T X \big) \beta \\
&\overset{d}{=} \beta^T \big(I - U^T (X^T X)^+ U U^T X^T X U\big) \beta \\
&= r^2 - (U \beta)^T (X^T X)^+ X^T X (U \beta).
\end{align*}
Choosing $U$ so that $U\beta=r e_i$, the $i$th standard basis vector, then
averaging over $i=1,\ldots,p$, yields  
$$
\E B_X(\hbeta; \beta)  = r^2 \E\big[ 1- \Tr\big((X^T X)^+X^T X\big)/p 
\big] = r^2(1-n/p).
$$
It is possible to show that,  $B_X(\hbeta; \beta)$ concentrates around its expectation ant therefore
$B_X(\hbeta; \beta) \to r^2(1-1/\gamma)$, almost surely. This is stated formally in the next section.

\subsection{Limiting risk}
\label{sec:risk_iso}

As the next result shows, the independence of the risk on $\beta$ is still
true outside of the Gaussian case, provided the features are isotropic.
The next result can be proved as a corollary of the more general Theorem \ref{thm:risk_gen_asymp}
below. We give a simpler self-contained proof using a theorem of 
\citet{rubio2011spectral} in Appendix \ref{sec:ProofRiskIso}.
\begin{theorem}
\label{thm:risk_iso}
Assume the model \eqref{eq:data_x}, \eqref{eq:data_y}, where $x \sim P_x$ has
i.i.d.\ entries with zero mean, unit variance, and a finite moment of order
$4+\eta$, for some $\eta>0$. Also assume that $\|\beta\|_2^2=r^2$ for all
$n,p$. Then for the min-norm least squares estimator \smash{$\hbeta$} in
\eqref{eq:min_ls}, as $n,p \to \infty$, such that $p/n \to \gamma \in (0,\infty)$, it holds
almost surely that
\begin{align}
  B_X(\hbeta;\beta)&\to r^2\big(1-\frac{1}{\gamma}\big) \, ,\\
  V_X(\hbeta;\beta)&\to \sigma^2 \frac{1}{\gamma-1}\, .
\end{align}
Hence,  summarizing with Proposition \ref{thm:risk_lo}, we have
\begin{equation}
\label{eq:risk_iso}
R_X(\hbeta; \beta) \to \begin{cases}
\sigma^2 \frac{\gamma}{1-\gamma} & \text{for $\gamma < 1$}, \\
r^2\big(1-\frac{1}{\gamma}\big) + \sigma^2 \frac{1}{\gamma-1} & 
\text{for $\gamma > 1$}.
\end{cases}
\end{equation}
\end{theorem}

On $(0,1)$, there is no bias, and the variance increases with $\gamma$; on
$(1,\infty)$, the bias increases with $\gamma$, and the variance decreases with
$\gamma$. Let  
$\snr=r^2/\sigma^2$.  Observe that the risk of the null estimator 
\smash{$\tilde\beta = 0$} is $r^2$, which we hence call the null risk.  The
following facts are immediate from the form of the risk curve in
\eqref{eq:risk_iso}.  See Figure \ref{fig:risk_iso} for an accompanying
plot when $\snr$ varies from 1 to 5.

\begin{figure}[ht]
\centering
\includegraphics[width=0.725\textwidth]{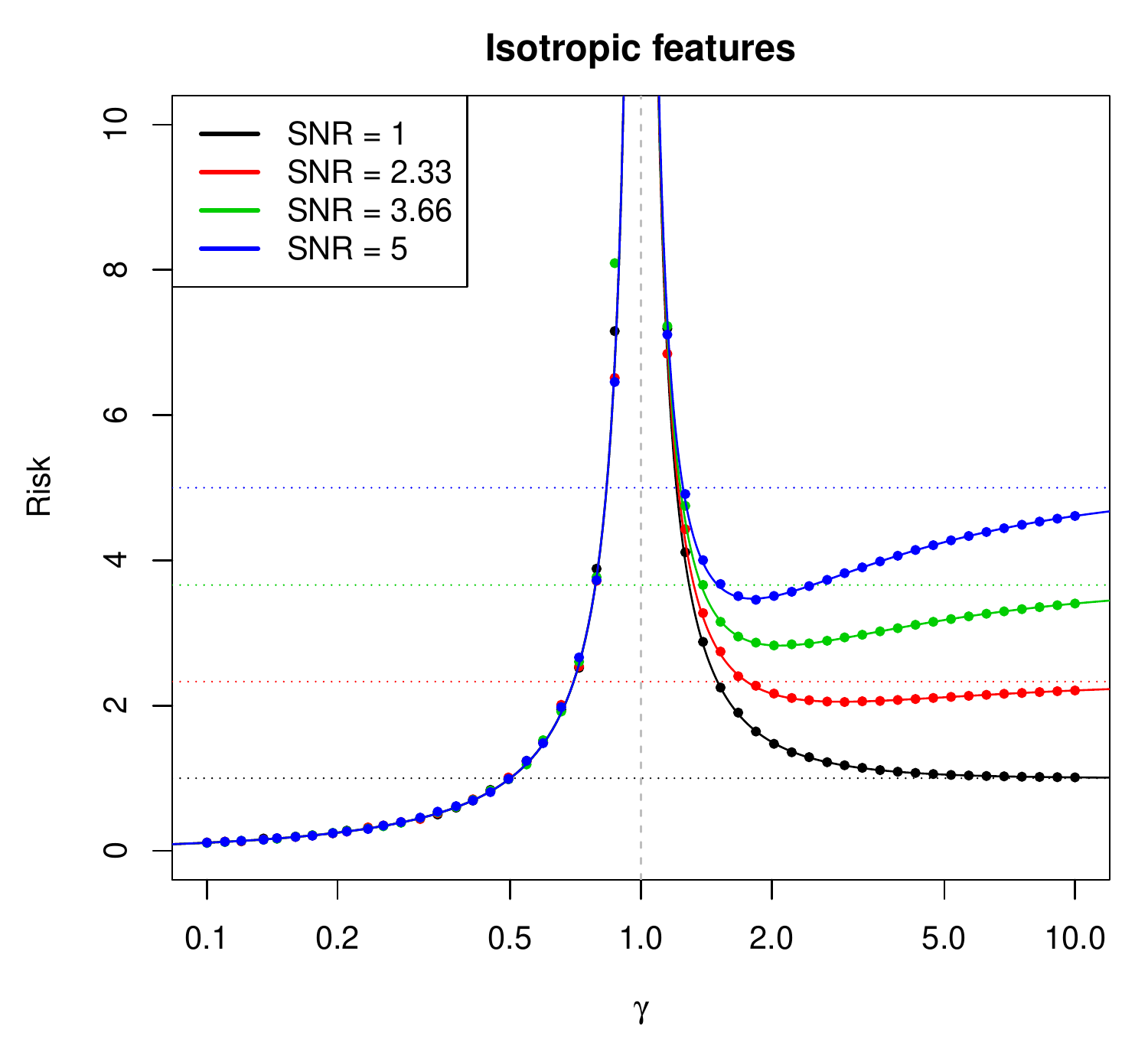} 
\caption{\footnotesize Asymptotic risk curves in \eqref{eq:risk_iso} for the min-norm
  least squares estimator, when $r^2$ varies from 1 to 5, and $\sigma^2=1$. For
  each value of $r^2$, the null risk is marked as a dotted line, and the points
  denote finite-sample risks, with $n=200$, $p=[\gamma n]$, across various
  values of $\gamma$, computed from features $X$ having i.i.d.\ $N(0,1)$
  entries.}     
\label{fig:risk_iso} 
\end{figure}

\begin{enumerate}
\item On $(0,1)$, the least squares risk $R(\gamma)$ is better than the null
  risk if and only if \smash{$\gamma < \frac{\snr}{\snr+1}$}. 

\item On $(1,\infty)$, when $\snr \leq 1$, the min-norm least squares
  risk $R(\gamma)$ is always worse than the null risk.  Moreover, it is
  monotonically decreasing, and approaches the null risk (from above) as $\gamma
  \to \infty$.  

\item On $(1,\infty)$, when $\snr > 1$, the min-norm least squares risk
  $R(\gamma)$ beats the null risk if and only if \smash{$\gamma >
    \frac{\snr}{\snr-1}$}.  Further, it has a local minimum at \smash{$\gamma= 
    \frac{\sqrt{\snr}}{\sqrt{\snr}-1}$}, and approaches the null risk (from
  below) as $\gamma \to \infty$.  
\end{enumerate}

\subsection{Limiting $\ell_2$ norm}

Calculation of the limiting $\ell_2$ norm of the min-norm least squares
estimator is quite similar to the study of the limiting risk in Theorem
\ref{thm:risk_iso}, and therefore we state the next result without proof. 

\begin{corollary}
\label{cor:norm_iso}
Assume the conditions of Theorem \ref{thm:risk_iso}.  Then as $n,p \to \infty$,
such that $p/n \to \gamma$, the squared $\ell_2$ norm of the min-norm least
squares estimator \eqref{eq:min_ls} satisfies, almost surely,  
$$
\E[\|\hbeta\|_2^2 \,|\, X] \to 
\begin{cases}
r^2 + \sigma^2 \frac{\gamma}{1-\gamma} & \text{for $\gamma < 1$}, \\  
r^2 \frac{1}{\gamma} + \sigma^2 \frac{1}{\gamma-1} & \text{for $\gamma > 
  1$}. 
\end{cases}
$$
\end{corollary}


We can see that the limiting norm, as a function of $\gamma$, has a somewhat 
similar profile to the limiting risk in \eqref{eq:risk_iso}: it is monotonically
increasing on $(0,1)$, diverges at the interpolation boundary, and is
monotonically decreasing on $(1,\infty)$.   

\section{Correlated features}
\label{sec:correlated}

We broaden the scope of our analysis from the last section, where we
examined isotropic features.  In this section, we take $x \sim P_x$ to be of
the form  $x = \Sigma^{1/2} z$, where $z$ is a random vector with independent
entries that have zero mean and unit variance, and $\Sigma$ is arbitrary (but
still deterministic and positive definite).

The risk of min-norm regression depends on the geometry of $\Sigma$ and $\beta$.
Denote by $\Sigma = \sum_{i=1}^ps_i v_iv_i^T$ the eigenvalue
decomposition of $\Sigma$ with $s_1\ge s_2\ge \cdots\ge s_p\ge 0$.
The geometry of the problem is captured by the sequence of eigenvalues $(s_1,\dots,s_p)$,
and  by the coefficients of $\beta$ in the basis of eigenvectors $((v_1^T\beta),\dots, (v_p^T\beta))$. We encode
these via two probability distributions on $\reals_{\ge0}$:
\begin{align}
  \hH_n(s) := \frac{1}{p}\sum_{i=1}^p 1_{\{s\ge s_i\}}\, ,\;\;\;\; \hG_n(s)  =\frac{1}{\|\beta\|_2^2}\sum_{i=1}^p \<\beta,v_i\>^2 1_{\{s\ge s_i\}}\, .
  \label{eq:HGdef}
\end{align}

We next state our assumptions about the data distribution: our results will be uniform with respect to the (large) constants $M$,
$\{C_k\}_{k\ge 2}$
appearing in this assumption. 
\begin{assumption}\label{ass:CovariatesAssumption}
  The covariates vector $x\sim P_x$ is of the form $x = \Sigma^{1/2} z$, where, defining $\hH_n$ as per Eq.~\eqref{eq:HGdef}, we have
  \begin{enumerate}
  \item[$(a)$] The vector  $z=(z_1,\dots, z_p)$ has independent (not necessarily identically distributed) entries with $\E\{z_i\}=0$,
    $\E\{z_i^2\}=1$, and $\E\{|z_i|^k\} \le C_k<\infty$ for all $k\ge 2$.
  \item[$(b)$] $s_1= \|\Sigma\|_{op}\le M$, $\int s^{-1} d \hH_n(s)<M$.
    \item[$(c)$] $|1-(p/n)|\ge 1/M$, $1/M\le p/n\le M$.
      \end{enumerate}
\end{assumption}
Condition $(a)$ bounds the tail probabilities on the covariates. Requiring finite moment of all order is
useful to get strong bounds on the deviations of the risk from its predicted value. As discussed below,
bounds on the first few moments are sufficient if we are satisfied in weaker probability bounds.

Conditions $(b)$ requires the eigenvalues of $\Sigma$ to be bounded, and not to
accumulate\footnote{The latter assumption could have been further relaxed, by requiring only $p_{+}$ of the
  $p$ eigenvalues to be non-vanishing and to satisfy the other conditions.  This would require to redefine $\gamma$
as $p_{+}/n$.} near $0$. 
For the analysis of min-norm interpolation, we will add the additional assumption that the
minimum eigenvalue of $\Sigma$ is bounded away from zero. However 
condition $(b)$ is sufficient for the analysis of ridge regression  in Section
\ref{sec:ridge}.

Finally, as our statements are non-asymptotic, we do not assume $p/n$  to converge to to a value. However 
condition $(c)$ requires $p/n$ to be bounded and bounded away from the interpolation threshold $p/n=1$.

\subsection{Prediction risk}

\begin{definition}[Predicted bias and variance: min-norm regression]\label{def:LimitBiasVar}
    Let $\hH_n$ be the empirical distribution of eigenvalues of $\Sigma$,
    and $\hG_n$ the reweighted distribution as per Eq.~\eqref{eq:HGdef}.
    For $\gamma\in\reals_{>0}$,
  define $c_0 = c_0(\gamma,\hH_n)\in\reals_{> 0}$ to be the
  unique non-negative solution of
  \begin{align}
    1-\frac{1}{\gamma} = \int \frac{1}{1+c_0\gamma s}\, d\hH_n(s)\, , \label{eq:c0def}
  \end{align}
  We then define the predicted bias and variance by
  \begin{align}
    \cuB(\hH_n,\hG_n,\gamma) &:=\|\beta\|_2^2\left\{1+\gamma  c_0\frac{\int \frac{s^2}{(1+c_0\gamma s)^2}\, d\hH_n(s)}
                               {\int \frac{s}{(1+c_0\gamma s)^2}\, d\hH_n(s)}\right\}\cdot  \int \frac{s}{(1+c_0\gamma s)^2}\, d\hG_n(s)\, ,
    \label{eq:BiasPredMinNorm}\\
    \cuV(\hH_n,\gamma) &:= \sigma^2\gamma  \frac{\int \frac{s^2}{(1+c_0\gamma s)^2}\, d\hH_n(s)}
{\int \frac{s}{(1+c_0\gamma s)^2}\, d\hH_n(s)}\, .
  \end{align}
\end{definition}

Note that evaluating $\cuB(H,G,\gamma)$,  $\cuV(H,\gamma)$ numerically is relatively straightforward,
with the most complex part being the solution of Eq.~\eqref{eq:c0def}.
The next theorem establishes that ---under suitable technical conditions---
the functions $\cuB$, $\cuV$ characterize the test error. Similar theorems were proved in
\cite{wu2020optimal,richards2020asymptotics}, which generalized an earlier version of this manuscript
to account for the geometry of $(\Sigma,\beta)$.
\begin{theorem}
\label{thm:risk_gen}
Assume the  data model \eqref{eq:data_x}, \eqref{eq:data_y} and that the covariates distribution satisfies
Assumption \ref{ass:CovariatesAssumption}. Further assume $s_p = \lambda_{\min}(\Sigma)>1/M$.
Define $\gamma=p/n$  and let
$\hbeta$ be the min-norm least squares
estimator in Eq.~\eqref{eq:min_ls}.

Then for any constants $D>0$ (arbitrarily large)
  there exist $C=C(M,D)$  such that, with probability at least $1-Cn^{-D}$ the following hold
\begin{align}
  &R_X(\hbeta;\beta)  =B_X(\hbeta;\beta)  +V_X(\hbeta;\beta) \, ,\\
  &\big|B_X(\hbeta;\beta)-\cuB(\hH_n,\hG_n,\gamma) \big|\le \frac{C\|\beta\|_2^2}{n^{1/7}}\, ,\label{eq:MinNormBias}\\
  &\big|V_X(\hbeta;\beta)-\cuV(\hH_n,\gamma) \big|\le \frac{C}{n^{1/7}}\, ,\label{eq:MinNormVariance}
\end{align}
where $\cuB$ and $\cuV$ are given in Definition \ref{def:LimitBiasVar-lambda}.
\end{theorem}
 The proof of this theorem is deferred to Section \ref{sec:ProofGeneralFormulaMinNorm}.

The order of the error bound in Eqs.~\eqref{eq:MinNormBias}, \eqref{eq:MinNormVariance} is not optimal:
a central limit theorem heuristics suggests the deterministic approximation to be accurate up to  an
error of order $n^{-1/2}$. Indeed, we are able to establish the correct order in the case of ridge regression, see
Theorem \ref{thm:risk_ridge}.

 \begin{remark}\label{rmk:Comparison}
Note that Theorem \ref{thm:risk_gen} establishes deterministic approximations for the bias and variance,
that are valid at finite $n,p$. The overparametrization ratio $\gamma=p/n$ is a non-asymptotic
quantity, and the error bounds are uniform, i.e. depend uniquely on the constant $M$.
This is to be contrasted with the asymptotic setting of \cite{richards2020asymptotics,wu2020optimal}.
Both of these papers assume a sequence of regression problems with $n,p\to\infty$, $p/n\to \gamma$,
and obtain an asymptotically exact expression for the risk.

In order for the asymptotics to make sense, additional assumptions are required by \cite{richards2020asymptotics,wu2020optimal}.
In \cite{richards2020asymptotics}, this is achieved by assuming $\beta$ to be random
with $\E[\beta\beta^T]=r^2\Phi(\Sigma)/d$ for a certain (deterministic) function $\Phi:\reals\to\reals$
(promoted to a function on matrices in the usual way). In addition,
the empirical spectral distribution of $\Sigma$ is assumed to converge.
To state the assumptions in \cite{wu2020optimal}, recall that $(s_i)_{i\le p}$ are the eigenvalues
of $\Sigma$, and denote by $b_i = p\cdot (v_i^T\beta)^2$ the projection of $\beta$ onto the eigenvectors of $\Sigma$. Then \cite{wu2020optimal} assumes that the joint empirical distribution
$p^{-1}\sum_{i=1}^{p}\delta_{s_i,b_i}$ converges weakly as $n,p\to\infty$.

Technically, \cite{richards2020asymptotics,wu2020optimal} apply asymptotic random matrix theory results, such as
 \cite{ledoit2011eigenvectors}, while we have to take a longer detour to exploit 
non-asymptotic results established in\cite{knowles2017anisotropic}. We believe that the non-asymptotic approach
provides more concrete and accurate statements.
\end{remark}

 Theorem \ref{thm:risk_gen} also implies asymptotic predictions under minimal assumptions.
  In particular,  if the two  probability measures $\hH_n$, $\hG_n$ converge weakly to probability measures
  $H$, $G$ on $[0,\infty)$, then we obtain $B_X(\hbeta;\beta)/\|\beta\|^2\to \cuB(H,G,\gamma)$,
  $V_X(\hbeta;\beta)\to \cuV(H,\gamma)$ almost surely.
\begin{theorem}\label{thm:risk_gen_asymp}
  Consider the setting of Theorem \ref{thm:risk_gen} but, instead of Assumption \ref{ass:CovariatesAssumption} $(a)$,
  assume that $(z_i)_{i\le p}$ are identically distributed and satisfy the conditions
  $\E z_i= 0$, $\E(z_i^2) = 1$, $\E(|z_i|^{4+\delta})\le C<\infty$.
  Further assume $p/n\to \gamma\in (0,\infty)$, $\hH_n\Rightarrow H$, $\hG_n\Rightarrow G$. Then,
  almost surely $B_X(\hbeta;\beta)/\|\beta\|^2_2\to \cuB(H,G,\gamma)$,
  $V_X(\hbeta;\beta)\to \cuV(H,\gamma)$.
 \end{theorem}
As pointed out above the condition $\hH_n\Rightarrow H$, $\hG_n\Rightarrow G$ (here $\Rightarrow$ denotes weak convergence)
is strictly weaker than the condition assumed in \cite{wu2020optimal} to establish asymptotic results.
Further, we require weaker moment conditions.

  In the next sections we illustrate the role of the geometry of $\beta,\Sigma$
  by considering two models for which $\hH_n\Rightarrow H$, $\hG_n\Rightarrow G$ as $n,p\to\infty$. First we consider the case $G=H$,
  which we refer to as `equidistributed:' the components of $\beta$ are roughly equally distributed along the
  eigenvectors of $\Sigma$. In this case there is no special relation between $\beta$ and $\Sigma$.
  
  As a further application, we consider a latent space model in which $\beta$ is aligned with the top eigenvectors of $\Sigma$. This can be regarded as a misspecified model, and is therefore presented in
  Section \ref{sec:LatentSpace} below.

  \subsection{Equidistributed coefficients}

  In this section we assume $G=H$, $\|\beta\|_2\to r$,
  and $p/n\to \gamma$.
  One way to generate
  $\beta$ satisfying this condition is to draw it uniformly at random on the $p$-dimensional sphere
  of radius $\|\beta\|_2 = r$. In this case the conditions of Theorem \ref{thm:risk_gen} (or Theorem \ref{thm:risk_gen_asymp}), hold
  with $\hG_n\to G=H$. 
  \begin{corollary}\label{coro:equidistributed}
    Under the assumptions of Theorem \ref{thm:risk_gen_asymp}, 
    further assume $G=H$,  $\|\beta\|^2\to r^2$. Then
    for $n,p \to \infty$, with $p/n \to \gamma > 1$, almost surely
    \begin{align}
      B_X(\hbeta;\beta) & \to \cuB_{\sequi}(H,\gamma) := \frac{r^2}{c_0(H,\gamma)\gamma^2}\, ,\label{eq:Bequi}\\
       V_X(\hbeta;\beta) & \to \cuV_{\sequi}(H,\gamma) := \cuV(H,\gamma) \, .
    \end{align}
    \end{corollary}
    As a special case, we can revisit the isotropic case $\Sigma=I$, which results in $dH = \delta_1$. In this case
    $c_0(H,\gamma) = \gamma(\gamma-1)$ yielding immediately $\cuB_{\sequi}(H,\gamma) =1-\gamma^{-1}$ and
    $\cuV_{\sequi}(H,\gamma) =1/(\gamma-1)$.

\subsection{Limiting $\ell_2$ norm}

Again, as in the isotropic case, analysis of the limiting $\ell_2$ norm is
similar to analysis of the risk in Theorem \ref{thm:risk_gen}. We give the next result without proof, as it is an immediate generalization of previous results.
\begin{corollary}
  \label{cor:norm_gen}
    Under the assumptions of Theorem \ref{thm:risk_gen_asymp}, 
    further assume  $\|\beta\|^2\to r^2$, and let $c_0=c_0(H,\gamma)$ be defined as there.
Then as $n,p \to \infty$,
such that $p/n \to \gamma$,  the min-norm least
squares estimator \eqref{eq:min_ls} satisfies, almost surely,
\begin{align}                                                                                                                        
\|\hbeta\|_2^2 \to 
\begin{dcases}
r^2 + \sigma^2 \frac{\gamma}{1-\gamma} \int \frac{1}{s} \, dH(s)& \text{for $\gamma < 1$}, \\  
\int\frac{c_0\gamma s}{1+c_0\gamma s}\,  dG(s) + c_0\gamma\sigma^2  & \text{for $\gamma >  
  1$},
\end{dcases}
 \end{align}
\end{corollary}

\section{Misspecified model}
\label{sec:misspecified}

In this section, we consider a misspecified model, in which the regression
function is still linear, but we observe only a subset of the features. Such a
setting provides another potential motivation for interpolation: in
many problems, we do not know the form of the regression function, and we
generate features in order to improve our approximation capacity.  Increasing
the number of features past the point of interpolation (increasing $\gamma$ past
1) can now decrease \emph{both} bias and variance.

As such, the misspecified model setting also yields further interesting
asymptotic comparisons between the $\gamma<1$ and $\gamma>1$ regimes.
Recall the isotropic features model of Section \ref{sec:risk_iso}: the risk function in \eqref{eq:risk_iso} is
globally minimized at $\gamma = 0$. This is a consequence of the fact that, in a well-specified linear model
at $\gamma=0$ there is no bias and no variance, and hence no risk.  In a misspecified model, we will see that
the story can be quite different, and the asymptotic risk can actually attain
its \emph{global} minimum on $(1,\infty)$.

\subsection{Data model and risk}

Consider, instead of \eqref{eq:data_x}, \eqref{eq:data_y}, a data model 
\begin{gather}
\label{eq:data_x_mis}
\big((x_i, w_i), \epsilon_i\big) \sim P_{x,w} \times P_\epsilon, \quad
i=1,\ldots,n, \\   
\label{eq:data_y_mis}
y_i = x_i^T \beta + w_i^T \theta + \epsilon_i, \quad i=1,\ldots,n,
\end{gather}
where as before the random draws across $i=1,\ldots,n$ are independent. Here, 
we partition the features according to $(x_i,w_i) \in \R^{p+q}$, $i=1,\ldots,n$,
where the joint distribution $P_{x,w}$ is such that $\E((x_i,w_i))=0$ and
$$
\Cov\big((x_i,w_i)\big) = \Sigma = 
\begin{bmatrix}
\Sigma_x & \Sigma_{xw} \\ 
\Sigma_{xw}^T & \Sigma_w
\end{bmatrix}. 
$$
We collect the features in a block matrix $[\,X \; W\,] \in \R^{n \times (p+q)}$
(which has rows $(x_i,w_i) \in \R^{p+q}$, $i=1,\ldots,n$).  We presume that $X$
is observed but $W$ is unobserved, and focus on the min-norm least squares
estimator exactly as before in \eqref{eq:min_ls}, from the regression of $y$ on
$X$ (not the full feature matrix $[\,X \; W\,]$).

Given a test point $(x_0,w_0) \sim P_{x,w}$, and an estimator \smash{$\hbeta$}
(fit using $X,y$ only, and not $W$), we define its out-of-sample
prediction risk as 
$$
R_X(\hbeta;\beta,\theta) = 
\E \big[ \big( x_0^T \hbeta - \E(y_0|x_0,w_0) \big)^2 \,|\, X \big] =
\E \big[ \big( x_0^T \hbeta -x_0^T \beta - w_0^T \theta \big)^2 \,|\, X \big].
$$
Note that this definition is conditional on $X$, and we are integrating over the 
randomness not only in $\epsilon$ (the training errors), but in the unobserved
features $W$, as well.  The next lemma decomposes this notion of risk in
a useful way.

\begin{lemma}
\label{lem:risk_mis_decomp}
Under the misspecified model \eqref{eq:data_x_mis}, \eqref{eq:data_y_mis}, for
any estimator \smash{$\hbeta$}, we have 
$$
R_X(\hbeta; \beta, \theta) = 
\underbrace{\E \big[ \big( x_0^T \hbeta - \E(y_0|x_0) \big)^2 \,|\, X
  \big]}_{R^*_X(\hbeta; \beta, \theta)} +
\underbrace{\E \big[ \big( \E(y_0|x_0) - \E(y_0|x_0,w_0) \big)^2
  \big]}_{M(\beta,\theta)}. 
$$
\end{lemma}

\begin{proof}
Simply add an subtract $\E(y_0|x_0)$ inside the square in the definition of
\smash{$R_X(\hbeta; \beta, \theta)$}, then expand, and note that the cross term
vanishes because
$\E \big[\big( \E(y_0|x_0) - \E(y_0|x_0,w_0) \big) \,|\, x_0 \big] = 0$.
\end{proof}

The first term \smash{$R^*_X(\hbeta; \beta, \theta)$} in the decomposition in
Lemma \ref{lem:risk_mis_decomp} is precisely the risk that we studied 
previously in the well-specified case, except that the response distribution 
has changed (due to the presence of the middle term in \eqref{eq:data_y_mis}).  
We call the second term \smash{$M(\beta, \theta)$} in Lemma
\ref{lem:risk_mis_decomp} the {\it misspecification bias}.

\begin{remark}
If $(x_i,w_i)$ are jointly Gaussian, then the above expressions simplify and  Theorem \ref{thm:risk_gen}  
can be used to characterize the risk \smash{$R_X(\hbeta; \beta, \theta)$}.
In particular, the conditional distribution of $w$ given $x$ is $P_{w|x} = N(\Sigma_{wx}\Sigma_{x}^{-1}x,\Sigma_{w|x})$
where $\Sigma_{wx}=\Sigma_{xw}^T$, and $\Sigma_{w|x} =\Sigma_w-\Sigma_{wx}\Sigma_x^{-1}\Sigma_{wx}^T$.
Further, $y= \tilde{\beta}^Tx +\tilde{\epsilon}$, where $\tilde\beta = \beta+\Sigma_x^{-1}\Sigma_{xw}\theta$,
and $\tilde{\eps}\sim N(0,\tilde\sigma^2)$, $\tilde\sigma^2=\sigma^2+\theta^T\Sigma_{w|x}\theta$.
It is then easy to show that the misspecification bias is $M(\beta,\theta) = \theta^T\Sigma_{w|x}\theta$
and the term $R^*_X(\hbeta; \beta, \theta)$ can be approximated using Theorem \ref{thm:risk_gen}.
\end{remark}

In order to discuss some qualitative features, we focus on the simplest possible model by assuming independent
covariates.

\subsection{Isotropic features}
\label{sec:risk_mis_iso}

Here, we make the additional   
simplifying assumption that $(x,w) \sim P_{x,w}$ has i.i.d.\ entries with unit
variance, which implies that $\Sigma=I$.  (The case of independent features but
general covariances $\Sigma_x,\Sigma_w$ is similar, and we omit the details.)
Therefore, we may write the response distribution in \eqref{eq:data_y_mis} as  
$$
y_i = x_i^T \beta + \delta_i, \quad i=1,\ldots,n,
$$
where $\delta_i$ is independent of $x_i$, having mean zero and variance
$\sigma^2+\|\theta\|_2^2$, for $i=1,\ldots,n$.  Denote the 
total signal by $r^2=\|\beta\|_2^2+\|\theta\|_2^2$, and the fraction of the
signal captured by the observed features by $\kappa=\|\beta\|_2^2/r^2$. Then
\smash{$R_X^*(\hbeta;\beta,\theta)$} behaves exactly as we computed previously, 
for isotropic features in the well-specified setting (Theorem \ref{thm:risk_lo}
for $\gamma < 1$, and Theorem \ref{thm:risk_iso} for $\gamma>1$), after we make
the substitutions:   
\begin{equation}
\label{eq:subs}
r^2 \mapsto r^2 \kappa \quad \text{and} \quad
\sigma^2 \mapsto \sigma^2 + r^2 (1-\kappa).
\end{equation}
Furthermore, we can easily calculate the misspecification bias: 
$$
M(\beta,\theta) = \E(w_0^T \theta)^2 = r^2 (1-\kappa).
$$
Putting these results together leads to the next conclusion. 

\begin{theorem}
\label{thm:risk_mis_iso}
Assume the misspecified model \eqref{eq:data_x_mis}, \eqref{eq:data_y_mis}, and
assume $(x,w) \sim P_{x,w}$ has i.i.d.\ entries with zero mean, unit variance,
and a finite moment of order $8+\eta$, for some $\eta>0$.  Also assume that
$\|\beta\|_2^2+\|\theta\|_2^2=r^2$ and $\|\beta\|_2^2/r^2 = \kappa$ for all
$n,p$. Then for the min-norm least squares estimator \smash{$\hbeta$} in
\eqref{eq:min_ls}, as $n,p \to \infty$, with $p/n \to \gamma$, it holds almost
surely that
$$
R_X(\hbeta;\beta,\theta) \to
\begin{cases}
r^2(1-\kappa) + \big(r^2(1-\kappa) + \sigma^2\big) 
\frac{\gamma}{1-\gamma} & \text{for $\gamma < 1$}, \\
r^2(1-\kappa) + r^2 \kappa \big(1-\frac{1}{\gamma}\big) 
+ \big(r^2(1-\kappa) + \sigma^2\big) \frac{1}{\gamma-1} &  
\text{for $\gamma > 1$}.
\end{cases}
$$
\end{theorem}

We remark that, in the independence setting considered in Theorem 
\ref{thm:risk_mis_iso}, the dimension $q$ of the unobserved feature space does 
not play any role: we may equally well take $q=\infty$ for all $n,p$
(i.e., infinitely many unobserved features).  

The components of the limiting risk from Theorem \ref{thm:risk_mis_iso} are
intuitive and can be interpreted as follows.  The first term $r^2(1-\kappa)$ is
the misspecification bias (irreducible).  The second term, which we deem as 0
for $\gamma<1$ and $r^2 \kappa (1-1/\gamma)$ for $\gamma>1$, is the bias.  The
third term, $r^2(1-\kappa) \gamma/(1-\gamma)$ for $\gamma<1$ and $r^2(1-\kappa)
/(\gamma-1)$ for $\gamma>1$, is what we call the {\it misspecification variance}:
the inflation in variance due to unobserved features, when we take
$\E(y_0|x_0)$ to be the target of estimation.  The last term, $\sigma^2
\gamma/(1-\gamma)$ for $\gamma<1$ and $\sigma^2/(\gamma-1)$ for  
$\gamma>1$, is the variance itself.

\subsection{Polynomial approximation bias}
\label{sec:poly_bias}

Since adding features should generally improve our approximation capacity, it is
reasonable to model $\kappa=\kappa(\gamma)$ as an increasing function of
$\gamma$. To get an idea of the possible shapes taken by the asymptotic risk
curve from Theorem \ref{thm:risk_mis_iso}, we consider the example of a  \emph{polynomial decay} for
the approximation bias, 
\begin{equation}
\label{eq:poly_decay}
1-\kappa(\gamma) = (1+\gamma)^{-a},
\end{equation}
for some $a > 0$.  In this case, the limiting risk in the isotropic setting,
from Theorem \ref{thm:risk_mis_iso}, becomes   
\begin{equation}
\label{eq:risk_mis_iso}
R_a(\gamma) = \begin{cases}
r^2(1+\gamma)^{-a} + (r^2(1+\gamma)^{-a} + \sigma^2)
\frac{\gamma}{1-\gamma} & \text{for $\gamma < 1$}, \\
r^2(1+\gamma)^{-a}  + r^2 \big(1-(1+\gamma)^{-a}\big)
\big(1-\frac{1}{\gamma}\big) + (r^2(1+\gamma)^{-a}  + \sigma^2)
\frac{1}{\gamma-1} & \text{for $\gamma > 1$}.  
\end{cases}
\end{equation}
We next summarize some interesting features of these risk curves, and Figures \ref{fig:risk_mis_lo} and
\ref{fig:risk_mis_hi} give accompanying plots for $\snr=1$ and $5$,
respectively.  Recall that the null risk is $r^2$, which comes from predicting
with the null estimator \smash{$\tilde\beta=0$}. 

\begin{figure}[p]
\vspace{-40pt}
\centering
\includegraphics[width=0.725\textwidth]{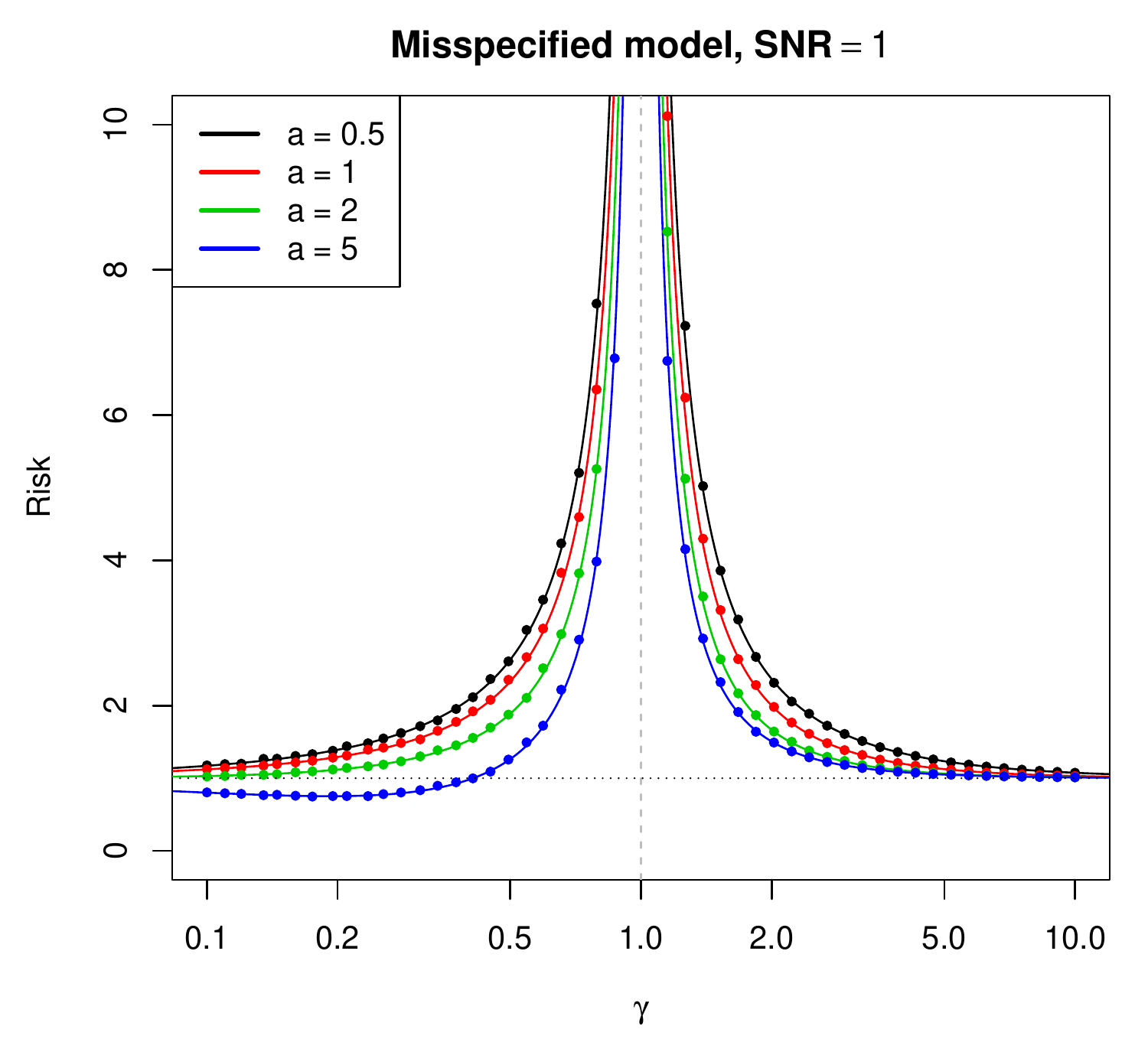} 
\caption{\footnotesize Asymptotic risk curves in \eqref{eq:risk_mis_iso} for the 
  min-norm least squares estimator in the misspecified case, when the
  approximation bias has polynomial decay as in \eqref{eq:poly_decay}, as $a$
  varies from 0.5 to 5.  Here $r^2=1$ and $\sigma^2=1$, so $\snr=1$. The null
  risk $r^2=5$ is marked as a dotted black line. The points denote finite-sample
  risks, with $n=200$, $p=[\gamma n]$, across various values of $\gamma$,
  computed from features $X$ having i.i.d.\ $N(0,1)$ entries.}    
\label{fig:risk_mis_lo} 

\bigskip
\includegraphics[width=0.725\textwidth]{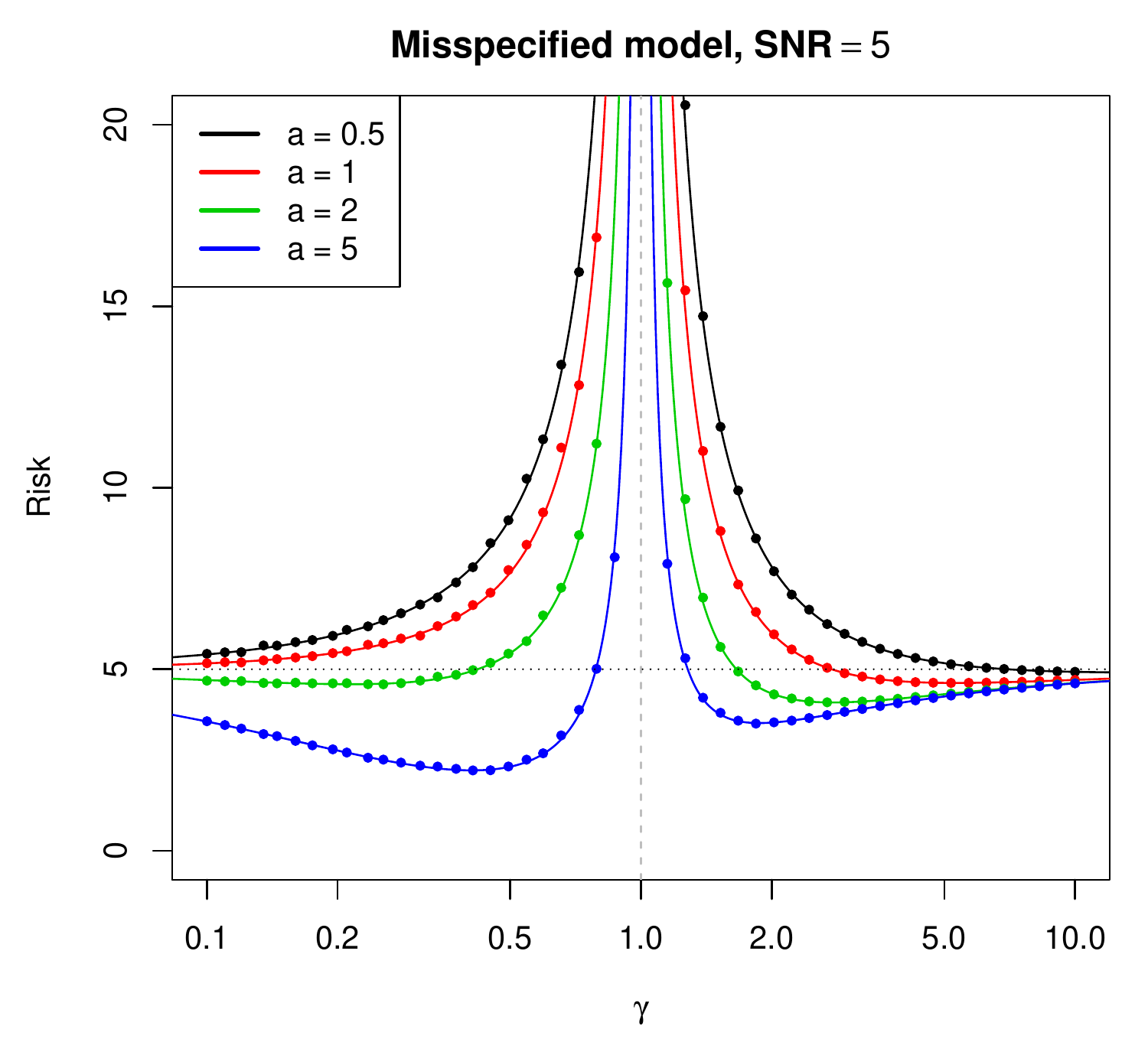}
\caption{\footnotesize Asymptotic risk curves in \eqref{eq:risk_mis_iso} for the 
  min-norm least squares estimator in the misspecified case, when the
  approximation bias has polynomial decay as in \eqref{eq:poly_decay}, as $a$ 
  varies from 0.5 to 5.  Here $r^2=5$ and $\sigma^2=1$, so $\snr=5$. The null
  risk $r^2=5$ is marked as a  dotted black line. The points are again
  finite-sample risks, with $n=200$, $p=[\gamma n]$, across various values of
  $\gamma$.}  
\label{fig:risk_mis_hi}
\end{figure}

\begin{enumerate}
\item On $(0,1)$, the least squares risk $R_a(\gamma)$ can only be better than
  the null risk if \smash{$a > 1+\frac{1}{\snr}$}. Further, in this case, we
  have $R_a(\gamma) < r^2$ if and only if $\gamma < \gamma_0$, where $\gamma_0$
  is the unique zero of the function
  $$
  (1+x)^{-a} + \Big(1+\frac{1}{\snr}\Big) x - 1
  $$
  that lies in \smash{$(0,\frac{\snr}{\snr+1})$}.  Finally, on
  \smash{$(\frac{\snr}{\snr+1},1)$}, the least squares risk $R_a(\gamma)$ is
  always worse than the null risk, regardless of $a>0$, and it is monotonically
  increasing.  

\item On $(1,\infty)$, when $\snr \leq 1$, the min-norm least squares risk
  $R_a(\gamma)$ is always worse than the null risk.  Moreover, it is
  monotonically decreasing, and approaches the null risk (from above) as $\gamma
  \to \infty$.      

\item On $(1,\infty)$, when $\snr > 1$, the min-norm least squares risk
  $R_a(\gamma)$ can be better than the null risk for any $a>0$, and in
  particular we have $R_a(\gamma) < r^2$ if and only if $\gamma <  \gamma_0$,
  where $\gamma_0$ is the unique zero of the function
  $$
  (1+x)^{-a} (2x-1) + 1 - \Big(1-\frac{1}{\snr}\Big) x
  $$
  lying in \smash{$(\frac{\snr}{\snr-1},\infty)$}.  Indeed, on
  \smash{$(1,\frac{\snr}{\snr-1})$}, the min-norm least squares risk
  $R_a(\gamma)$ is always worse than the null risk (regardless of $a>0$), and it
  is monotonically decreasing.  

\item When $\snr > 1$, for small enough $a>0$, the global minimum of the
  min-norm least squares risk $R_a(\gamma)$ occurs after $\gamma=1$.  A
  sufficient but not necessary condition is \smash{$a \leq 1 +
    \frac{1}{\snr}$} (because, due to points 1 and 3 above, we see that in this
  case $R_a(\gamma)$ is always worse than null risk for $\gamma < 1$, but will
  be better than the null risk at some $\gamma > 1$).  
\end{enumerate}

\subsection{A latent space model}
\label{sec:LatentSpace}

\begin{figure}[ht]
\vspace{-40pt}
\centering
\includegraphics[width=0.725\textwidth]{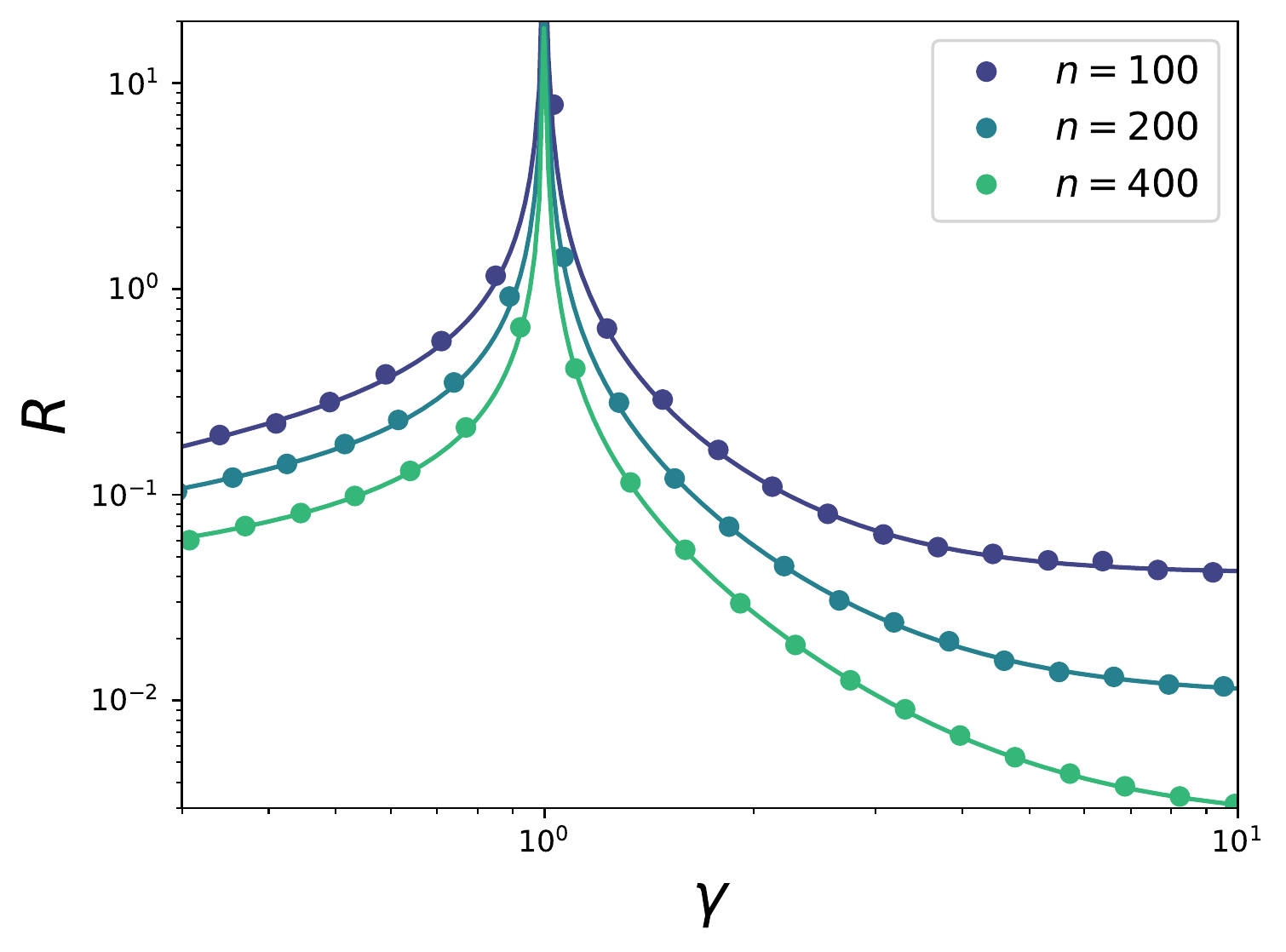} 
\caption{\footnotesize Latent space model of Section \ref{sec:LatentSpace}: test error $R_X(\hbeta;\beta)$ of minimum norm regression as a function of the overparametrization ratio $\gamma$. Here
  $d=20$, $r=1$, $\sigma_{\xi}=0$ and $n$ varies across various curves. Symbols are averages over $100$
 realizations; continuous lines report the analytical prediction of Corollary \ref{coro:Latent}.}    
\label{fig:risk_latent_space_n} 
\end{figure}

\begin{figure}[ht]
\vspace{-40pt}
\centering
\includegraphics[width=0.725\textwidth]{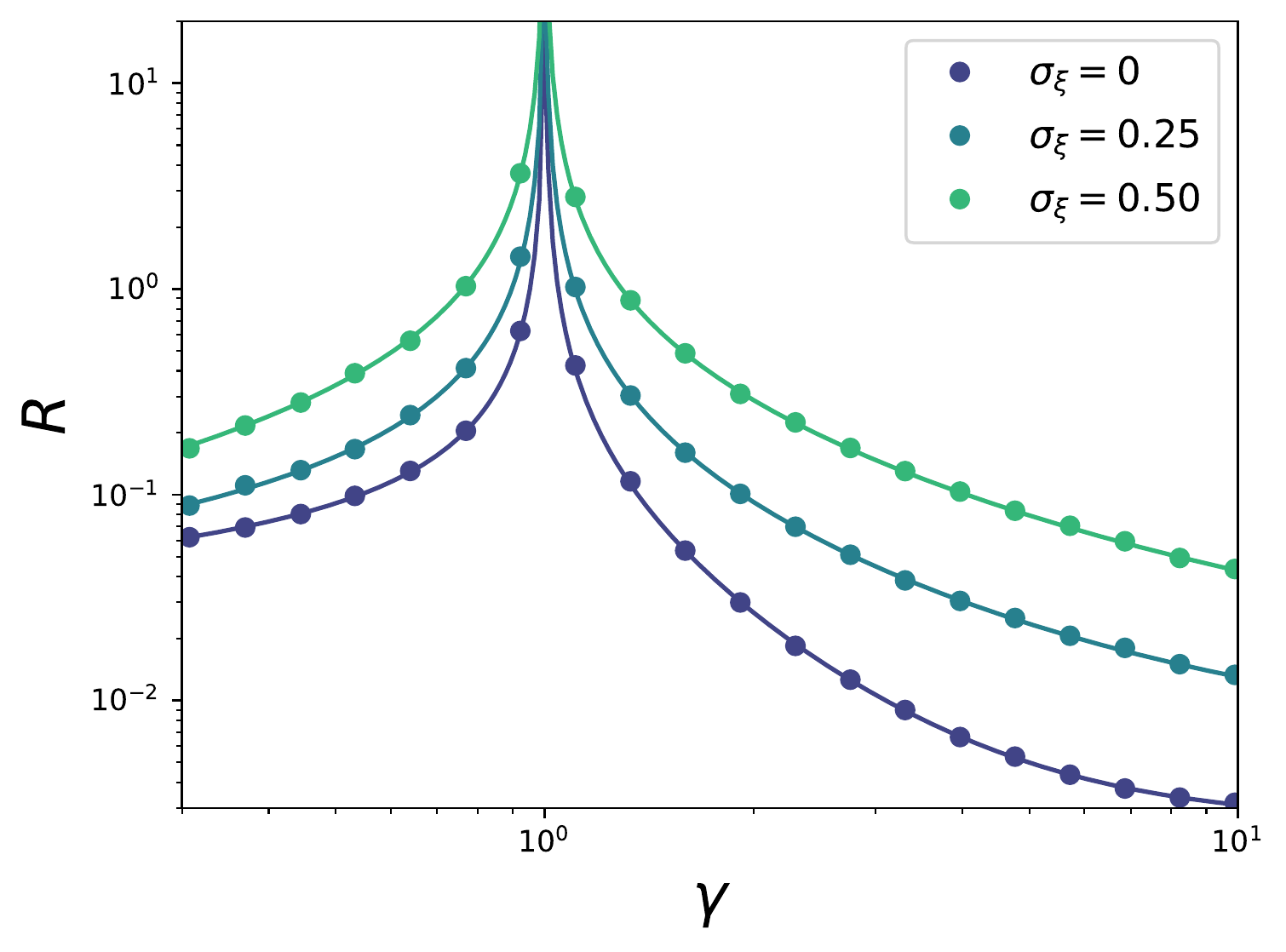}
\caption{\footnotesize Latent space model of Section \ref{sec:LatentSpace}: test error $R_X(\hbeta;\beta)$ of minimum norm regression as a function of the overparametrization ratio $\gamma$. Here
  $d=20$, $r=1$, $n=400$, and the noise variance $\sigma_{\xi}$
  varies across different curves. Symbols are averages over $100$
 realizations; continuous lines report the analytical prediction of Corollary \ref{coro:Latent}.}  
\label{fig: risk_latent_space_sigma}
\end{figure}

We next consider an example in which $\beta$ is aligned with the top eigenvectors of $\Sigma$.
To motivate this example, assume that the responses $y_i$ are linear in the latent features vectors
$z_i\in\reals^d$. We do not observe this latent vector, but rather observe $p\ge d$ covariates
$x_i:=(x_{i1},\dots,x_{ip})$ that also linear in the latent vector $z_i$:
\begin{align}
  y_i = \theta^Tz_i+\xi_i\, ,\;\;\;\;\;\; x_{ij} = w_j^Tz_i+ u_{ij}\, .\label{eq:LS_First}
\end{align}
Here $(\xi_i)_{i\le n}$, $(u_{ij})_{i\le n,j\le p}$ are noise variables that are mutually independent,
and independent of $z_i$, with $\xi_i\sim\normal(0,\sigma_{\xi}^2)$, $u_{ij}\sim\normal(0,1)$.
The features matrix takes the form $X = ZW^T+U$ and therefore, for $p>n$, min-norm regression amounts to
\begin{align}
  \hbeta = \arg\min \Big\{ \|b\|_2 :  ZW^Tb +Ub = 0\Big\}\, .\label{eq:MinNormLatent}
  \end{align}

Apart from its intrinsic interest, this latent-space model is directly connected to nonlinear random features models, as
the ones studied in Section \ref{sec:nonlinear}. Indeed, in  non-linear random features model we have
$x_{ij} = \varphi(w_j^Tz_i)$. We can decompose this as $x_{ij} = a_0+a_1w_j^Tz_i+ \tilde{\varphi}(w_j^Tz_i)$,
where $\tilde{\varphi}$ is such that $\tilde\varphi(w_j^Tz_i)$ has zero mean and is uncorrelated with $w_j^Tz_i$
when $z_j\sim\normal(0,\id_d)$. Equation \eqref{eq:LS_First} then corresponds to replacing the uncorrelated random variable $\tilde\varphi(w_j^Tz_i)$
by the independent Gaussians $u_{ij}$.

We assume the latent vector to be isotropic $z_i\sim\normal(0,I_d)$.
It is then easy to see that the model is equivalent
to (a special case) model studied in Section \ref{sec:correlated}, namely $x_i\sim\normal(0,\Sigma)$, $y_i = \beta^Tx_i+\eps_i$.
Indeed (considering the case of Gaussian $\eps_i$), in both cases $(y_i, x_i)\in\reals^{p+1}$ is a centered Gaussian vector.
By matching their covariances we obtain (defining $\psi = d/p$):
\begin{align}
  &\Sigma = I_p+WW^T\, ,\;\;\;\;\; \beta =  W(I+W^TW)^{-1}\theta\, ,\\
  &\epsilon_i  \sim N(0,\sigma^2)\, ,\;\;\;\;\;\; \sigma^2 = \sigma_{\xi}^2+\theta^T(I+W^TW)^{-1}\theta\, .
\end{align}
Here $W\in\reals^{p\times d}$ is the matrix with rows $(w_i)_{i\le p}$. In what follows $r_{\theta}^2:=\|\theta\|_2^2$.
As anticipated, the coefficients vector is aligned with the top eigenspace of $\Sigma$ (the span of the columns of
$W$).

Notice that the signal-to-noise ratio in the features $x_{ij}$ is equal to $\|w_i\|_2^2$. To be definite, we'll fix
$\|w_i\|_2=1$, which implies in particular $\|W\|_F^2=p$. In order to simplify our calculations,
we assume all the non-zero singular values of $W$ to be equal, whence $W^TW = (p/d)I_d$. It is not hard to check that this
model satisfies the assumptions of  Theorem \ref{thm:risk_gen}, with
\begin{align}
  H(s) & = (1-\psi)\bfone(s\ge 1)+\psi\bfone(s\ge 1+\psi^{-1})\, ,\\
  G(s) &=\bfone(s\ge 1+\psi^{-1})\, ,\;\;\;\; \|\beta\|^2_2 = \frac{\psi r^2_{\theta}}{(1+\psi)^2}
\end{align}
Using Theorem \ref{thm:risk_gen}, we get the following explicit expressions.
\begin{corollary}\label{coro:Latent}
  Consider the latent space model described above, and assume $d/p\to \psi\in (1,\infty)$,
  $p/n\to \gamma\in (0,\infty)$. Then, almost surely
    \begin{align}
      R_X(\hbeta;\beta) &\to  \cuB_{\slat}(\psi,\gamma) + \cuV_{\slat}(\psi,\gamma) \, ,\\
      \cuB_{\slat}(\psi,\gamma)  & :=\left\{1+\gamma c_0
                                   \frac{  \cuE_1(\psi,\gamma) }{  \cuE_2(\psi,\gamma) }\right\}\cdot\frac{r^2_{\theta}}{(1+\psi)
                                   (1+c_0\gamma(1+\psi^{-1}))^2}\, ,\\
     \cuV_{\slat}(\psi,\gamma)  & :=\sigma^2\gamma c_0
                                   \frac{  \cuE_1(\psi,\gamma) }{  \cuE_2(\psi,\gamma) }\, ,\\
     \cuE_1(\psi,\gamma) & := \frac{1-\psi}{(1+c_0\gamma)^2}+ \frac{\psi(1+\psi^{-1})^2}{(1+c_0(1+\psi^{-1})\gamma)^2}\, ,\\  \cuE_2(\psi,\gamma) &:= \frac{1-\psi}{(1+c_0\gamma)^2}+ \frac{1+\psi}{(1+c_0(1+\psi^{-1})\gamma)^2}\, .
    \end{align}
    where $\sigma^2=\sigma_{\xi}^2+\psi r_{\theta}^2/(1+\psi)$, and $c_0 =c_0(\psi,\gamma)\ge 0$ is the unique non-negative solution of the following second order equation
    \begin{align}
      1-\frac{1}{\gamma} = \frac{1-\psi}{1+c_0\gamma}+ \frac{\psi}{1+c_0(1+\psi^{-1})\gamma}\, .
      \end{align}
\end{corollary}

Figures \ref{fig:risk_latent_space_n} and \ref{fig: risk_latent_space_sigma} illustrate this corollary
by comparing analytical predictions to numerical simulations. We observe that the prediction risk is monotone decreasing in the overparametrization ratio for $\gamma>1$, and reaches its global minimum asymptotically as $\gamma\to\infty$ (after $p,n,d\to\infty$). To understand why this happens,
notice  that each feature vector $x_{i}$ can be viewed as a noisy measurement of the latent
covariates $z_i$. If the noise $u_{ij}$ was absent, then performing min-norm regression with respect to
$(x_i)_{i\le n}$ amounts to minimize would be equivalent to min-norm regression with respect to $(z_i)_{i\le n}$.
To see this, consider again Eq.~\eqref{eq:MinNormLatent}. If we drop
the noise $U$, we are minimizing $\| b\|_2$
subject to $Z(W^Tb) = 0$, and the regression function is $\hat{f}(z) = x^T\hbeta =  z^T(W^T\hbeta)$.
Since $W$ is orthogonal, this is equivalent to computing $\hat{\theta} = \arg\min\{\|t\|_2:$ subject to $Zt = 0\}$.

In presence of noise $u_{ij}$, the latent features cannot be estimated exactly. However as $p$ gets larger,
the noise is effectively `averaged out' and we approach the idealized situation in which the $z_i$'s are observed.

\section{Ridge regularization}
\label{sec:ridge}

We generalize the formulas of Section \ref{sec:correlated} to non-vanishing ridge regularization.
We work under the same assumptions of that section.
In particular, recall that $\hH_n(s) = p^{-1}\sum_{i=1}^p 1_{\{s\ge s_i\}}$ is the empirical distribution of eigenvalue of
$\Sigma$, and $\hG_n(s)  =\sum_{i=1}^p \<\beta,v_i\>^2 1_{\{s\ge s_i\}}/\|\beta\|^2$ the same empirical distribution, reweighted
by the projection of $\beta$ onto the eigenvector $v_i$ of the covariance $\Sigma$.
(Recall the eigenvalue decomposition $\Sigma=\sum_{i=1}^ps_iv_iv_i^T$.)
\begin{definition}[Predicted bias and variance: ridge regression]\label{def:LimitBiasVar-lambda}
  Let $\hH_n$ be the empirical distribution of eigenvalues of $\Sigma$,
  and $\hG_n$ the reweighted distribution as per Eq.~\eqref{eq:HGdef}. For $\gamma\in\reals_{>0}$, and $z\in \complex_+$
  (the set of complex numbers with $\Im(z)>0$), 
  define $m_n(z) = m(z;\hH_n,\gamma)$ as the unique solution of
  \begin{align}
   m_n(z) = \int \frac{1}{s[1-\gamma-\gamma zm_n(z)] -z}\, d\hH_n(s)\, .\label{eq:m0def}
  \end{align}
  Further define $m_{n,1}(z) = m_{n,1}(z;\hH_n,\gamma)$ via
  \begin{align}
    m_{n,1}(z) := \frac{\int \frac{s^2[1-\gamma-\gamma zm_n(z)]}{[s[1-\gamma-\gamma zm_n(z)] -z]^2}\, d\hH_n(s)}
{1-\gamma \int \frac{zs}{[s[1-\gamma-\gamma zm_n(z)] -z]^2}\, d\hH_n(s)}\, .\label{eq:m1def}
  \end{align}
  These definitions are extended analytically to $\Im(z) = 0$ whenever possible.
  We then define the predicted bias and variance by
  \begin{align}
    \cuB(\lambda;\hH_n,\hG_n,\gamma) &:=\lambda^2\|\beta\|^2 \big(1+\gamma m_{n,1}(-\lambda)\big) \int \frac{s }{[\lambda +(1-\gamma+\gamma\lambda m_n(-\lambda))s]^2} d\hG_n(s)\, ,\label{eq:BiasRidge}\\
    \cuV(\lambda;\hH_n,\gamma) &:= \sigma^2\gamma  \int \frac{s^2(1-\gamma+\gamma\lambda^2 m'_n(-\lambda))}{[\lambda+s(1-\gamma+\gamma \lambda m_n(-\lambda))]^2} d\hH_n(s)\, . \label{eq:VarianceRidge}
  \end{align}
\end{definition}

 We next state our deterministic approximation of the risk.
\begin{theorem}\label{thm:risk_ridge}
  Let $M^{-1} \le p/n \le M$,  and Assumption \ref{ass:CovariatesAssumption} hold.
  Further assume $\lambda\vee \lambda_{\min}(\Sigma)>1/M$ and $n^{-2/3+1/M}<\lambda<M$.
  Let $\hbeta_{\lambda}$ be the ridge estimator of Eq.~\eqref{eq:ridge}.

  Then for any constants $D>0$ (arbitrarily large) and $\eps>0$ (arbitrarily small),
  there exist $C=C(M,D)$  such that, with probability at least $1-Cn^{-D}$ the following hold
\begin{align}
  &R_X(\hbeta _{\lambda};\beta)  =B_X(\hbeta _{\lambda};\beta)  +V_X(\hbeta _{\lambda};\beta) \, ,\\
  &\big|B_X(\hbeta_{\lambda};\beta)-\cuB(\lambda;\hH_n,\hG_n,\gamma) \big|\le \frac{C\|\beta\|_2^2}{\lambda n^{(1-\eps)/2}}\, ,\\
  &\big|V_X(\hbeta_{\lambda};\beta)-\cuV(\lambda;\hH_n,\gamma) \big|\le \frac{C}{\lambda^2 n^{(1-\eps)/2}}\, ,
\end{align}
  where $\cuB$ and $\cuV$ are given in Definition \ref{def:LimitBiasVar-lambda}. 
\end{theorem}
The proof of this theorem is deferred to Appendix \ref{sec:ProofGeneralFormula}.
 As for theorem \ref{thm:risk_gen}, similar results were proved in
 \cite{wu2020optimal,richards2020asymptotics}, subsequently to a first version of this manuscript that only focused on
 random $\beta$. The same comparison of Remark \ref{rmk:Comparison} applies here.
 
 In particular,  Theorem \ref{thm:risk_ridge} establishes  non-asymptotic deterministic approximations for the bias
  $B_X(\hbeta_{\lambda};\beta)$ and variance $V_X(\hbeta;\beta)$. The error terms are uniform
  over the covariance matrix, and have nearly optimal dependence upon
  the sample size $n$. Indeed, a  central-limit theorem
  heuristics suggests indeed fluctuations of order $n^{-1/2}$.

  As for the case of min-norm regression, it is easy to extract asymptotic prediction under minimal assumptions.
  \begin{theorem}\label{thm:risk_ridge_asymp}
      Consider the setting of Theorem \ref{thm:risk_ridge} but, instead of Assumption \ref{ass:CovariatesAssumption} $(a)$,
  assume that $(z_i)_{i\le p}$ are identically distributed and satisfy the conditions
  $\E z_i= 0$, $\E(z_i^2) = 1$, $\E(|z_i|^{4+\eta})\le C<\infty$ for some $\eta>0$.
  Further assume $p/n\to \gamma\in (0,\infty)$, $\hH_n\Rightarrow H$, $\hG_n\Rightarrow G$.
  Then, for any $\lambda>0$,
  almost surely $B_X(\hbeta_{\lambda};\beta)/\|\beta\|^2_2\to \cuB(\lambda;H,G,\gamma)$,
  $V_X(\hbeta_{\lambda};\beta)\to \cuV(\lambda;H,\gamma)$.
\end{theorem}

\subsection{Isotropic features}

As a special case, we can consider the simple isotropic model that was already studied in Section \ref{sec:isotropic}.
Very similar (though not
identical) results can be found in \citet{dicker2016ridge,dobriban2018high}.
\begin{corollary}
\label{coro:risk_ridge_iso}
Assume the conditions of Theorem \ref{thm:risk_iso} (well-specified model, 
isotropic features). Then for ridge regression estimator in \eqref{eq:ridge} 
as $n,p \to \infty$, such that $p/n \to \gamma \in 
(0,\infty)$, it holds almost surely that 
\begin{align}
R_X(\hbeta_\lambda;\beta) \to r^2\lambda^2m'(-\lambda)+\sigma^2 \gamma \big( m(-\lambda) - \lambda m'(-\lambda) \big)\ . \label{eq:RiskIsoRidge}
\end{align}
Here $m(z)$ is given by Eq.~\eqref{eq:m0def} which in this case has the explicit solution
$m(z) = [1-\gamma-z-\sqrt{(1-\gamma-z)^2-4\gamma z}]/(2\gamma z)$.

Furthermore, the limiting ridge risk is
minimized at $\lambda^*=\sigma^2\gamma/r^2$,  in which case we have the simpler expression
$R_X(\hbeta_\lambda;\beta) \to \sigma^2\gamma m(-\lambda^*)$.

Under the conditions of Theorem \ref{thm:risk_mis_iso} (misspecified model,
isotropic features), the limiting risk of ridge regression is as in the first
two displays, and the optimal limiting risk is as in the third, after we make the  
substitutions in \eqref{eq:subs} and add $r^2(1-\kappa)$, to each expression.  
\end{corollary}
It is easy to recover the formulas in Theorem \ref{thm:risk_iso} as a limiting case of Eq.~\eqref{eq:RiskIsoRidge},
by using the $z\to 0$ asymptotics $m(z) = (1-\gamma)^{-1}+O(z)$ for $\gamma<1$ and $m(z) = (1-\gamma)^{-1}+O(z)$ for $\gamma<1$ and
$m(z) = -(\gamma-1)/(\gamma z)+[(\gamma-1)\gamma]^{-1} +O(z)$ for $\gamma>1$. 

\begin{figure}[p]
\vspace{-30pt}
\centering
\includegraphics[width=0.725\textwidth]{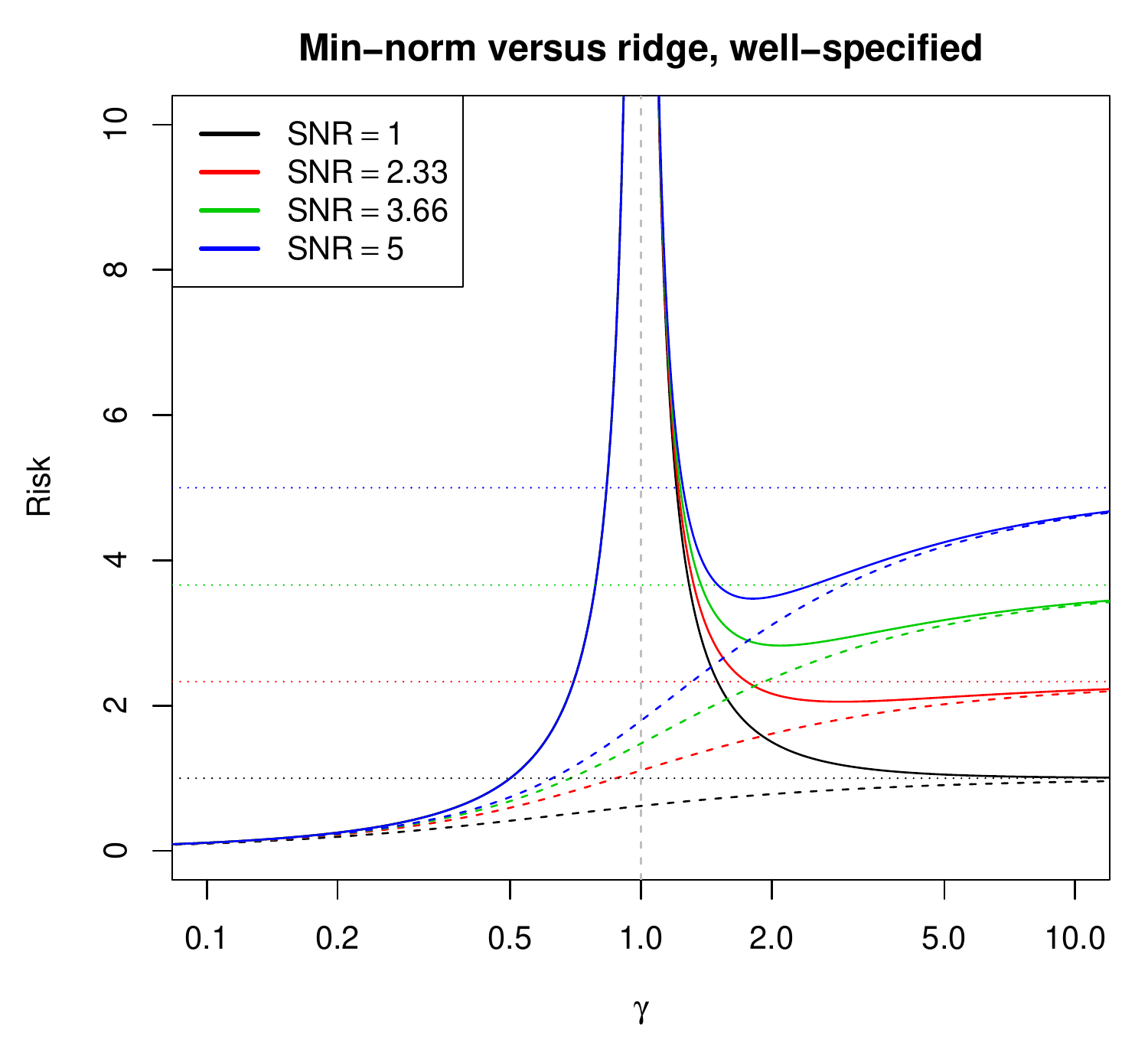} 
\caption{\footnotesize  Asymptotic risk curves for the min-norm least squares
  estimator in \eqref{eq:risk_iso} as solid lines, and optimally-tuned ridge
  regression (from Theorem \ref{thm:risk_ridge}) as dashed lines. Here $r^2$
  varies from 1 to 5, and $\sigma^2=1$.  The null risks are marked by the dotted
  lines.}  
\label{fig:risk_rg}

\bigskip
\includegraphics[width=0.725\textwidth]{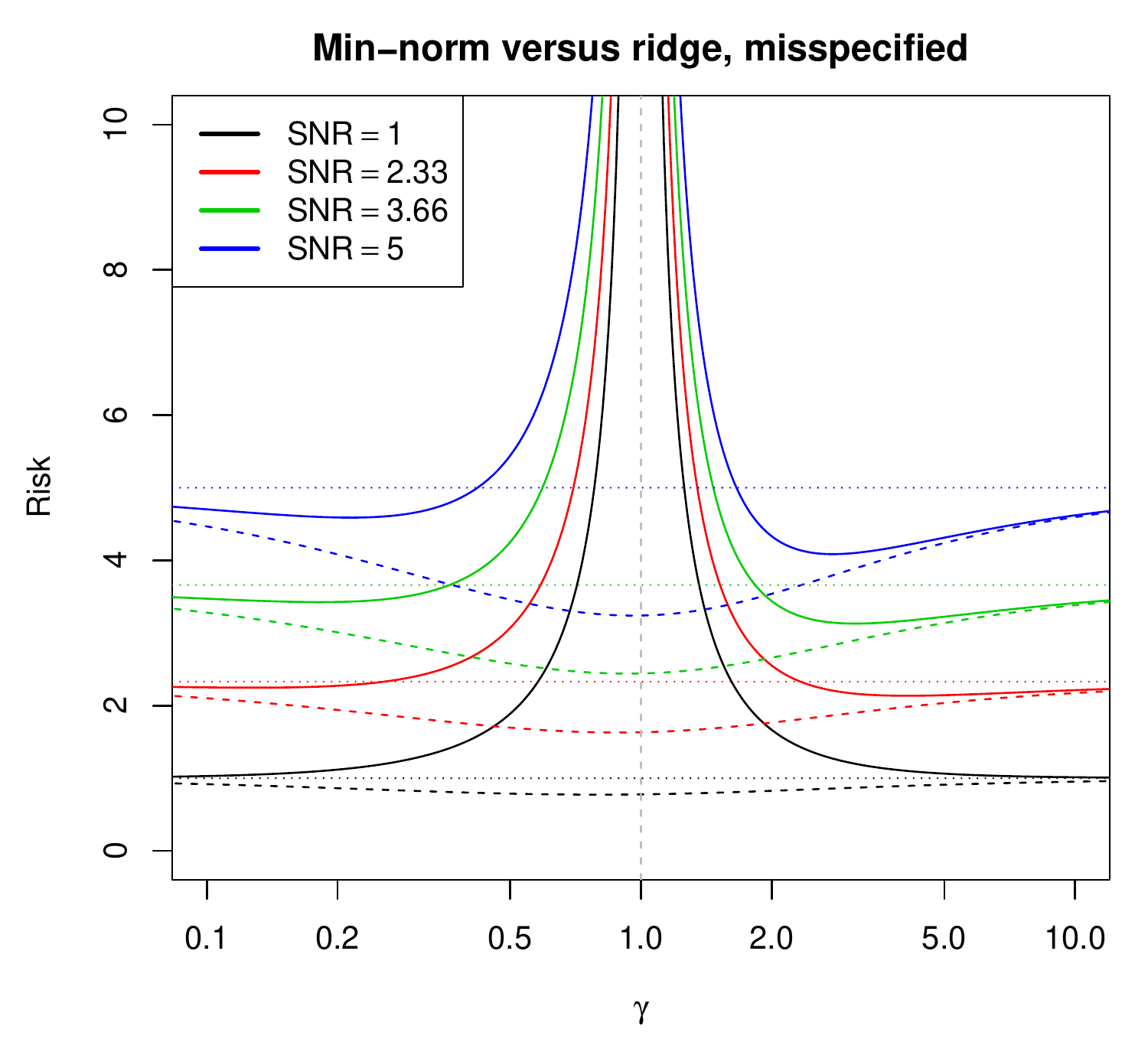}
\caption{\footnotesize Asymptotic risk curves for the min-norm least squares
  estimator in \eqref{eq:risk_mis_iso} as solid lines, and optimally-tuned ridge 
  regression (from Theorem \ref{thm:risk_ridge}) as dashed lines, in the
  misspecified case, when the approximation bias has polynomial decay as in
  \eqref{eq:poly_decay}, with $a=2$.  Here $r^2$ varies from 1 to 5, and
  $\sigma^2=1$.  The null risks are marked by the dotted lines.}      
\label{fig:risk_rg_ab}
\end{figure}

Figures \ref{fig:risk_rg} and \ref{fig:risk_rg_ab} compare the risk curves of
min-norm least squares to those from optimally-tuned ridge regression, in the
well-specified and misspecified settings, respectively.  There are two important
points to make.  The first is that optimally-tuned ridge regression is seen to
have strictly better asymptotic risk throughout, regardless of $r^2$, $\gamma$,
$\kappa$.  This should not be a surprise, as by definition optimal
tuning should yield better risk than min-norm least squares, which is the
special case given by $\lambda \to 0^+$.

The second point is that, in this example, the limiting risk of
optimally-tuned ridge regression appears to have a minimum around $\gamma=1$,
and this occurs closer and closer to $\gamma=1$ as $\snr$ grows.  This
behavior is interesting, especially because it is  antipodal to that
of the min-norm least squares risk, and leads us to very different suggestions
for practical usage for feature generators: in settings where we apply 
substantial $\ell_2$ regularization (say, using CV tuning to mimic
optimal tuning, which the next section shows to be asymptotically equivalent),
it seems we want the complexity of the feature space to put us as close to the
interpolation boundary ($\gamma=1$) as possible.

As we will see, the behavior is rather different in the latent space model.

\subsection{Latent space model}
\label{sec:LatentRidge}

\begin{figure}[p]
\vspace{-30pt}
\centering
\includegraphics[width=0.725\textwidth]{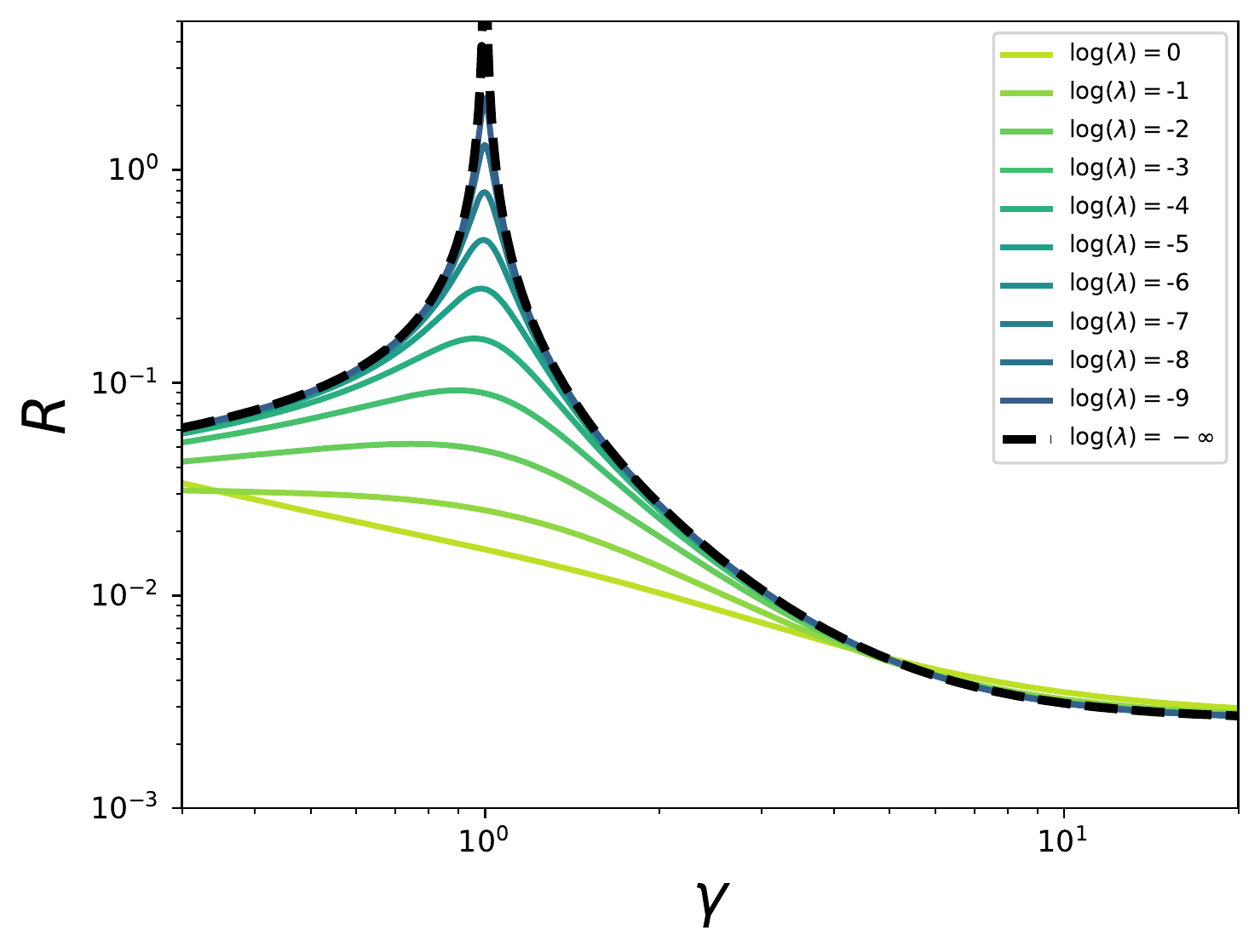} 
\caption{\footnotesize Asymptotic risk as a function of the overparametrization ratio $\gamma = p/n$,
  for ridge regression in the latent space model of Section \ref{sec:LatentRidge}. Here
  $\delta=n/d=20$, $r_{\theta}=1$, $\sigma_{\xi}=0$, and each curve corresponds to a different value of the
  regularization $\lambda$. The dashed curve correspond to the min-norm interpolator
  (which coincides with the $\lambda\to 0$ limit of ridge regression).}  
\label{fig:lowrank_ridge}

\bigskip
\includegraphics[width=0.725\textwidth]{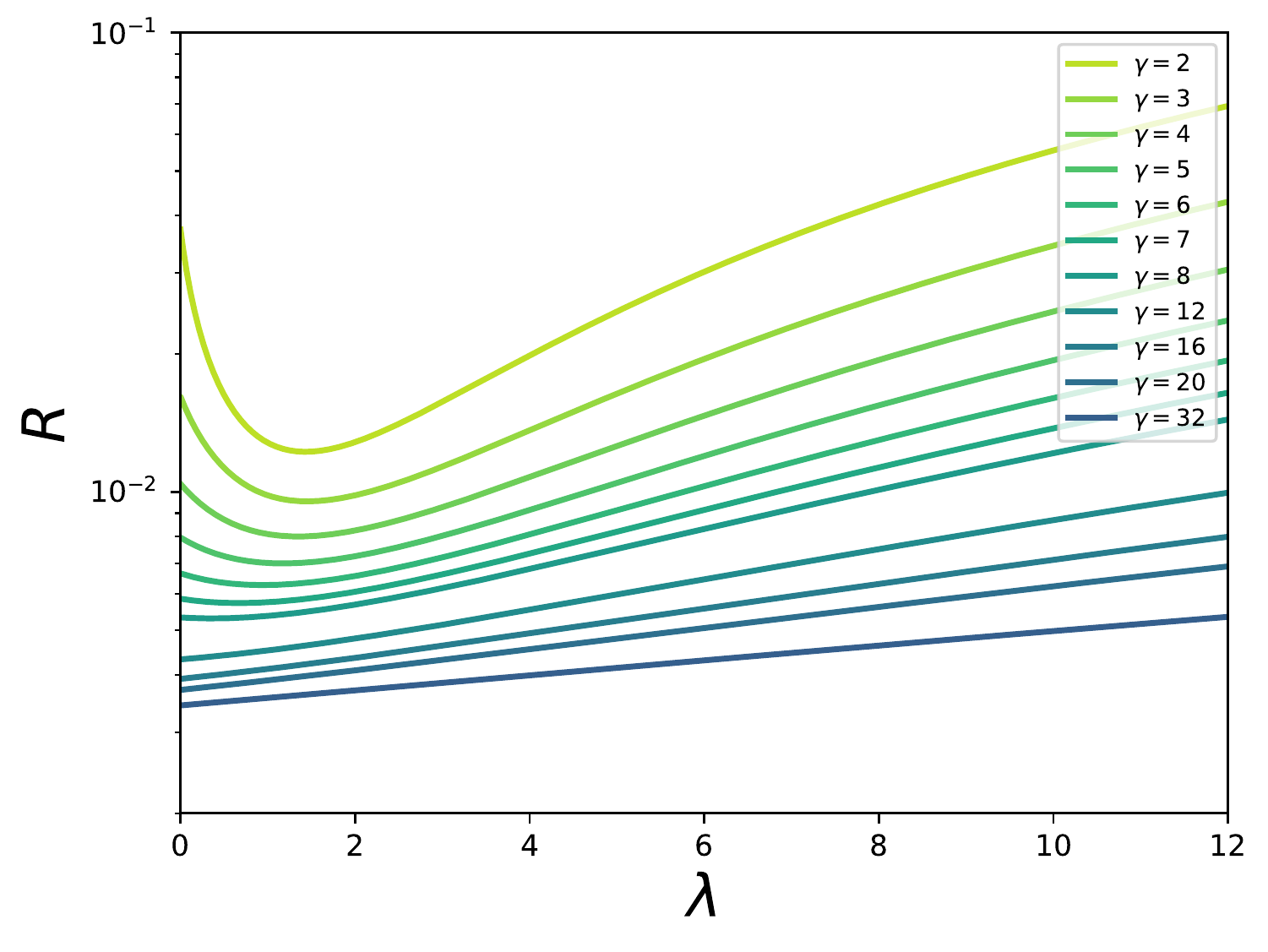}
\caption{\footnotesize Asymptotic risk as a function of the regularization parameter $\lambda$,
  for ridge regression in the latent space model of Section \ref{sec:LatentRidge}. Here
  $\delta=n/d=20$, $r_{\theta}=1$, $\sigma_{\xi}=0.1$, and each curve corresponds to a different value of the overparametrization ratio $\gamma=p/n$.}
\label{fig:lowrank_ridge_lamb}
\end{figure}

As a special application, we consider the latent space model of Section \ref{sec:LatentSpace}.
It is immediate to specialize Eqs.~\eqref{eq:BiasRidge} and \eqref{eq:VarianceRidge}
to this case. We omit giving giving explicit formulas for brevity, and
instead plot the resulting curves for the prediction risk (test error).

In Figure \ref{fig:lowrank_ridge} we plot the risk as a function of
the overparametrization ration $\gamma=p/n$ for several values of the regularization parameter $\lambda$
(included the ridgeless limit $\lambda\to 0$). The setting here is analogous to the one of Fig.~\ref{fig:risk_latent_space_n}.
We observe several interesting phenomena:
\begin{enumerate}
\item Independently of $\lambda$ in the probed range, the risk is minimized at large overparametrization $\gamma\gg 1$.
\item As expected, the divergence of the risk at the interpolation threshold $\gamma=1$ is smoothed out by regularization,
  and the risk becomes a monotone decreasing function of $\gamma$ when $\lambda$ is large enough. Crucially, the optimal amount of regularization
  (corresponding to the lower envelope of these curves) results in a monotonically decreasing risk.
\item At large overparametrization, the optimal value of the regularization parameter is $\lambda\to 0$.
\end{enumerate}
We confirm the last finding in Fig.~\ref{fig:lowrank_ridge_lamb} which plots the risk as a function of $\lambda$:
the optimal regularization is $\lambda\to 0$. Notice that this is the case despite the fact that the observations are noisy,
namely $\sigma_{\xi}>0$ strictly.

The optimality of $\lambda\to 0$ has a known intuitive explanation that is worth recalling here.
Recall that ridge predictor at point $x_0$ takes the form
\begin{align}
  \hf_{\lambda}(x_0) = \<x_0,\hbeta_{\lambda}\> = K(x_0,X)\big(K(X,X)+\lambda I\big)^{-1}y\, ,
\end{align}
where we introduced the kernel matrix $K(X,X) = XX^T/n$, and the vector $K(x_0,X) = x_0^TX/n$.
In the present case $X$ can ve viewed as a noisy version of $\oX=ZW^T$, namely 
Consider the case in which the covariates $X$ contain noise, as is our case,
$X=\oX+\eta U$ where $u_{ij}\sim N(0,1)$. (While we are considering $\eta=1$, it is instructive to
regard the noise standard deviation as a parameter.)
Then, we might expect
$K(X,X) \approx K(\oX,\oX)+\eta^2UU^T\approx K(\oX,\oX)+\eta^2 I_n$. If this approximations hold, the noise acts as an
extra ridge term, which can be sufficient to regularize the problem.

To the best of our knowledge, this argument was first presented by \citet{webb1994functional}
and, more explicitly, by \citet{bishop1995training}. Recently, \citet{kobak2020optimal} elucidated
its role in linear regression, establishing several of its consequences. In a parallel line of work,
a closely related idea has recently emerged  in the analysis of kernel methods in high dimension
\cite{el2010spectrum,liang2018just}.

\section{Cross-validation}
\label{sec:cv}

We analyze the effect of using cross-validation to choose the tuning parameter
in ridge regression.  In short, we find that choosing the ridge tuning parameter
to minimize the leave-one-out cross-validation error leads to the same
asymptotic risk as the optimally-tuned ridge estimator.  The next subsection
gives the details; the following subsection presents a new ``shortcut formula'' 
for leave-one-out cross-validation in the overparametrized regime, for min-norm 
least squares, akin to the well-known formula for underparametrized least
squares and ridge regression.    We refer to   \cite{Craven:79,golub1979gcv} for background on
CV and GCV, and to \cite{arlot2010survey} for a more recent review.

\subsection{Limiting behavior of CV tuning}

Given the ridge regression solution \smash{$\hbeta_\lambda$} in
\eqref{eq:ridge}, trained on $(x_i,y_i)$, $i=1,\ldots,n$, denote by
\smash{$\hf_\lambda$} the corresponding ridge predictor, defined as
\smash{$\hf_\lambda(x) = x^T \hbeta_\lambda$} for $x \in \R^p$.  Additionally,
for each $i=1,\ldots,n$, denote by \smash{$\hf^{-i}_\lambda$} the ridge
predictor trained on all but $i$th data point $(x_i,y_i)$.\footnote{To be
  precise, this is \smash{$\hf^{-i}(x) = x^T (X_{-i}^T X_{-i} + n\lambda I)^{-1}
    X_{-i}^T y_{-i}$}, where $X_{-i}$ denotes $X$ with the $i$th row removed,
  and $y_{-i}$ denotes $y$ with the $i$th component removed. Arguably, it may
  seem more natural to replace the factor of $n$ here by a factor of $n-1$; we 
  leave the factor of $n$ as is because it simplifies the presentation in what
  follows, but we remark that the same asymptotic results would hold with $n-1$
  in place of $n$.} Recall that the {\it leave-one-out cross-validation} 
(leave-one-out CV, or simply CV) error of the ridge solution at a tuning
parameter value $\lambda$ is   
\begin{equation}
\label{eq:cv}
\cv_n(\lambda) = 
\frac{1}{n} \sum_{i=1}^n \big( y_i - \hf^{-i}_\lambda(x_i) \big)^2.  
\end{equation}
We typically view this as an estimate of the out-of-sample prediction
error \smash{$\E(y_0 - x_0^T \hbeta_\lambda)^2$}, where the expectation is taken 
over everything that is random: the training data $(x_i,y_i)$, $i=1,\ldots,n$
used to fit \smash{$\hbeta_\lambda$}, as well as the independent test point 
$(x_0,y_0)$. Note also that, when we observe training data from the model
\eqref{eq:data_x}, \eqref{eq:data_y}, and when $(x_0,y_0)$ is drawn
independently according to the same process, we have the relationship  
$$
\E(y_0 - x_0^T \hbeta_\lambda)^2 
= \sigma^2 + \E(x_0^T \beta - x_0^T \hbeta_\lambda)^2 
= \sigma^2 + \E[R_X(\hbeta_\lambda; \beta)],
$$
where \smash{$R_X(\hbeta_\lambda; \beta) ] = \E[(x_0^T \beta - x_0^T
  \hbeta_\lambda)^2 \,|\, X]$} is the conditional prediction risk, which has
been our focus throughout.

Recomputing the leave-one-out predictors \smash{$\hf^{-i}_\lambda$},
$i=1,\ldots,n$ can be burdensome, especially for large $n$.  Importantly, there
is a well-known ``shortcut formula'' that allows us to express the leave-one-out
CV error \eqref{eq:cv} as a weighted average of the training errors, 
\begin{equation}
\label{eq:cv_shortcut}
\cv_n(\lambda) = 
\frac{1}{n} \sum_{i=1}^n \bigg( \frac{y_i - \hf_\lambda(x_i)}  
{1-(S_\lambda)_{ii}} \bigg)^2,
\end{equation}
where $S_\lambda=X(X^T X + n\lambda)^{-1} X^T$ is the ridge smoother matrix.
This identity is an immediate consequence of the Sherman-Morrison-Woodbury formula.
In the next subsection, we give an extension to 
the case $\lambda=0$ and $\rank(X)=n$, i.e., to min-norm least squares.  

The next result shows that, for isotropic features, the CV error of a ridge
estimator converges almost surely to its prediction error.  The focus on
isotropic features is only for simplicity: a more general analysis
is possible but is not pursued here. The proof, given in Appendix
\ref{app:risk_cv}, relies on the shortcut formula \eqref{eq:cv_shortcut}.  In 
the proof, we actually first analyze generalized cross-validation (GCV), which
turns out to be somewhat of an easier calculation (see the proof for details on
the precise form of GCV), and then relate leave-one-out CV to GCV. 

\begin{theorem}
\label{thm:risk_cv}
Assume the a isotropic prior, namely $\E(\beta)=0$, $Cov(\beta)=r^2\id_p/p$,
and the data model \eqref{eq:data_x},
\eqref{eq:data_y}.  Assume that $x \sim P_x$ has i.i.d.\ entries with zero mean, 
unit variance, and a finite moment of order $4+\eta$, for some $\eta>0$. Then 
for the CV error \eqref{eq:cv} of the ridge estimator in \eqref{eq:ridge} with
tuning parameter $\lambda > 0$, as $n,p \to \infty$, with $p/n \to \gamma \in
(0,\infty)$, it holds almost surely that    
$$
\cv_n(\lambda) - \sigma^2 \to \sigma^2 \gamma \big( m(-\lambda) -   
\lambda (1 - \alpha \lambda) m'(-\lambda) \big),
$$
where \smash{$m(z)$} denotes the Stieltjes transform of the
Marchenko-Pastur law $F_\gamma$ (as in Corollary \ref{coro:risk_ridge_iso}), and $\alpha = r^2/(\sigma^2 \gamma)$.  
Observe that the right-hand side is the asymptotic risk of ridge regression from 
Theorem \ref{thm:risk_ridge}.  Moreover, the above convergence is uniform
over compacts intervals excluding zero.  Thus if $\lambda_1,\lambda_2$ are  
constants with $0 < \lambda_1 \leq \lambda^* \leq \lambda_2 < \infty$, where
$\lambda^*=1/\alpha$ is the asymptotically optimal ridge tuning parameter value, 
and we define \smash{$\lambda_n = \arg\min_{\lambda \in [\lambda_1,\lambda_2]}  
\cv_n(\lambda)$}, then the expected risk of the CV-tuned ridge estimator $R_X(\hbeta):=\E_\beta R_X(\hbeta;\beta)$
\smash{$\hbeta_{\lambda_n}$} satisfies, almost surely, 
$$
R_X(\hbeta_{\lambda_n}) \to \sigma^2 \gamma m(-1/\alpha),
$$
with the right-hand side above being the asymptotic risk of optimally-tuned   
ridge regression. Further, the exact same set of results holds for GCV. 
\end{theorem}

Similar results were obtained  for various linear smoothers in 
\citet{li1986asymptotic,li1987asymptotic},
for the lasso in the high-dimensional (proportional) asymptotics in \citet{miolane2018distribution}, and for
general smooth penalized estimators in \citet{xu2019consistent}. The latter
paper covers ridge regression as a special case, and gives more precise results
(convergence rates), but assumes more restrictive conditions.  

The key implication of Theorem \ref{thm:risk_cv}, in the context of the current
paper and its central focus, is that the CV-tuned or GCV-tuned ridge
estimator has the same asymptotic performance as the optimally-tuned
ridge estimator. In other words, the ridge curves in Figures
\ref{fig:summary}, \ref{fig:risk_rg}, and \ref{fig:risk_rg_ab} can be
alternatively viewed as the asymptotic risk of ridge under CV tuning.

\subsection{Shortcut formula for ridgeless CV}

We extend the leave-one-out CV shortcut formula \eqref{eq:cv_shortcut} to work
when $p>n$ and $\lambda=0+$, i.e., for min-norm least squares.  In this case,
both the numerator and denominator are zero in each summand of
\eqref{eq:cv_shortcut}. To circumvent this, we can use the so-called ``kernel
trick'' to rewrite the ridge regression solution \eqref{eq:ridge} with
$\lambda>0$ as 
\begin{equation}
\label{eq:ridge_kernel}
\hbeta_\lambda = X^T (XX^T + n\lambda I)^{-1} y.
\end{equation}
Using this expression, the shortcut formula for leave-one-out CV in \eqref{eq:cv_shortcut} can be 
rewritten as
$$
\cv_n(\lambda) = \frac{1}{n} \sum_{i=1}^n 
\frac{[(XX^T + n\lambda I)^{-1} y]^2_i}
{[(XX^T + n\lambda I)^{-1}]^2_{ii}}\, .
$$
Taking $\lambda \to 0^+$ yields the shortcut formula for leave-one-out CV in
min-norm least squares (assuming without a loss of generality that
$\rank(X)=n$),  
\begin{equation}
\label{eq:cv_shortcut_mn}
\cv_n(0) = \frac{1}{n} \sum_{i=1}^n 
\frac{[(XX^T)^{-1} y]_i^2}{[(XX^T)^{-1}]_{ii}^2},
\end{equation}
In fact, the exact same arguments given here still apply when we replace $XX^T$
by a positive definite kernel matrix $K$ (i.e., $K_{ij}=k(x_i,x_j)$ for each
$i,j=1,\ldots,n$, where $k$ is a positive definite kernel function), in which
case \eqref{eq:cv_shortcut_mn} gives a shortcut formula for leave-one-out CV in 
kernel ridgeless regression (the limit in kernel ridge regression as $\lambda
\to 0^+$).  We also remark that, when we include an unpenalized intercept in the
model, in either the linear or kernelized setting, the shortcut formula
\eqref{eq:cv_shortcut_mn} still applies with $XX^T$ or $K$ replaced by their
doubly-centered (row- and column-centered) versions, and the matrix inverses
replaced by pseudoinverses.


\section{Nonlinear model}
\label{sec:nonlinear}

Our analysis in the previous sections assumed $x_i = \Sigma^{1/2}z_i$, with $z_i$ a vector with
i.i.d. entries. As discussed in the introduction, we expect our results to have implications on certain neural
networks in the lazy training regime, via the universality hypothesis. Roughly speaking, if
covariates are obtained by applying the featurization map $z_i\mapsto x_i=\nabla f(z_i;\theta_0)$ to
data $z_i$, the asymptotic behavior of ridge regression is expected to be close to the one of an equivalent
Gaussian model, whose second-order statistics match the ones of $\nabla f(z_i;\theta_0)$.

In this section, we test universality on one simple example.
We observe data as in \eqref{eq:data_x}, \eqref{eq:data_y},
but now $x_i = \varphi(W z_i) \in \R^p$, where $z_i \in \R^d$ has i.i.d.\
entries from $N(0,1)$, for $i=1,\ldots,n$.  Also, $W \in \R^{p \times d}$ has
i.i.d.\ entries from $N(0,1/d)$. Finally, $\varphi:\reals\to\reals$ is an activation function acting  
componentwise on vectors.

We focus on a simple case in which the second order statistics match the ones of the isotropic model.

\subsection{Limiting risk}

Notice that,  conditionally on $W$, the vectors  $x_i = \varphi(W z_i)$, $i\le n$ are independent. However, they do not have independent
coordinates: intuitively if $d$ is significantly smaller than $p$, $x_i$ contains much less randomness than a vector of $p$ independent
entries\footnote{For instance, if $\varphi(t) = at^2+b$,  we can reconstruct  $z_i$ from the first $2d$ coordinates of $x_i$, and
  therefore the remaining  $p-2d$ coordinates of $x_i$ are a function of the first $2d$.}.

Nevertheless, the next theorem shows that if $\varphi$ is purely nonlinear (in the sense that $\E\{\varphi(G)\} =\E\{G\varphi(G)\} = 0$),
then the feature matrix $X$ behaves ``as if'' it has
i.i.d.\ entries, in that the asymptotic bias and variance are exactly as in the
linear isotropic case, recall Eq.~\eqref{eq:risk_iso}.  In other words, this theorem provides a rigorous confirmation of the universality hypothesis stated in the introduction.
\begin{theorem}
  \label{cor:purely_nonlinear}
  Assume the model \eqref{eq:data_x}, \eqref{eq:data_y}, where each
$x_i = \varphi(W z_i) \in \R^p$, for $z_i \in \R^d$ having i.i.d.\
entries from $N(0,1)$, $W \in \R^{p \times d}$ having i.i.d.\ entries from  
$N(0,1/d)$ (with $W$ independent of $z_i$), and is $\varphi$ an activation
function that acts componentwise.  Assume that $|\varphi(x)|\le
c_0(1+|x|)^{c_0}$ for a constant $c_0>0$. Also, for $G \sim N(0,1)$, assume that
the standardization conditions hold: $\E[\varphi(G)]=0$ and
$\E[\varphi(G)^2]=1$, $\E[G\varphi(G)]=0$.

Then for $\gamma>1$, the variance  
satisfies, almost surely,   
$$
\lim_{\lambda \to 0^+} \, \lim_{n,p,d \to \infty} \, 
V_X(\hbeta_\lambda; \beta) = \frac{\sigma^2}{\gamma-1},
$$
which is precisely as in the case of linear isotropic features, recall Theorem 
\ref{thm:risk_iso}. Also, under a isotropic prior, namely $\E(\beta)=0$, $Cov(\beta)=r^2\id_p/p$), the Bayes bias
$B_X(\hbeta_{\lambda}) := \E_{\beta}B_X(\hbeta_{\lambda};\beta)$
satisfies, almost surely   
$$
\lim_{\lambda \to 0^+} \, \lim_{n,p,d \to \infty} \,
B_X(\hbeta_\lambda) = 
\begin{cases}
0 & \text{for $\gamma < 1$}, \\
r^2(1-1/\gamma) & \text{for $\gamma > 1$},
\end{cases}
$$
which is again as in the case of linear isotropic features, recall Theorem
 \ref{thm:risk_iso}. 
\end{theorem}
The proof of Theorem \ref{cor:purely_nonlinear} is lengthy and will be sketched
shortly.  

\begin{figure}[p]
\vspace{-40pt}
\centering
\includegraphics[width=0.725\textwidth]{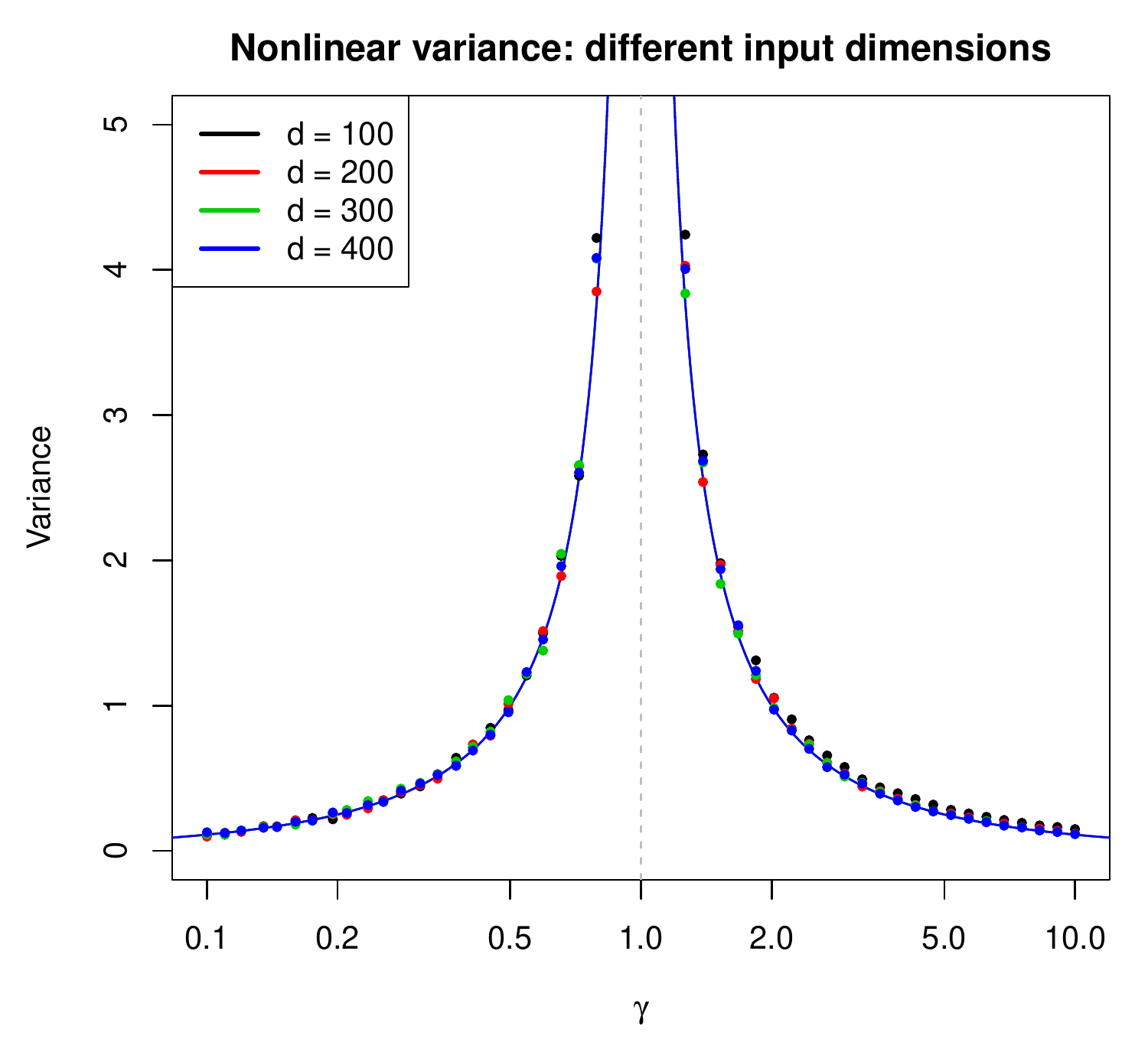} 
\caption{\footnotesize Asymptotic variance curves for the min-norm least squares 
  estimator in the nonlinear feature model (from Theorem
  \ref{cor:purely_nonlinear}), for the purely nonlinear activation
  $\varphi_{\sabs}$.  Here $\sigma^2=1$, and the points are finite-sample
  risks, with $n=200$, $p=[\gamma n]$, over various values of $\gamma$, and
  varying input dimensions: $d=100$ in black, $d=200$ in red, $d=300$ in green,
  and $d=400$ in black.  As before, the features used for finite-sample
  calculations are $X=\varphi(ZW^T)$, where $Z$ has i.i.d.\ $N(0,1)$ entries and
  $W$ has i.i.d.\ $N(0,1/d)$ entries.} 
\label{fig:purely_nonlinear}
\end{figure}

The origin of the conditions $\E\{\varphi(G)\}=\E\{G\varphi(G)\}=0$ can be easily explained
(throughout $G\sim N(0,1)$).
In summary, these conditions ensure that the first and second order statistics of $x_i=\varphi(Wz_i)$ approximately match
those of the isotropic model. To illustrate this point, let $i\neq j$, and
assume that the corresponding rows of $W$ (denoted by $w_i^T$ and $w_j^T$) have unit norm (this will only be approximately true,
but simplifies our explanation). We then have $\E\{\varphi(w_i^Tz_1)\} =\E\{\varphi(w_j^Tz_1)\} =\E\{\varphi(G)\}=0$. Further
\begin{align}
  \E\{x_{1,i}x_{1,j}|W\} = \E\{\varphi(w_i^Tz_1)\varphi(w_{j}^Tz_1)|W\} =  \E\{\varphi(G_1)\varphi(G_2)\}\, ,
\end{align}
where $G_1,G_2$ are jointly Gaussian with unit variance and covariance $w_i^Tw_j$. Denoting by $\varphi(x) = \sum_{k\ge 0}\lambda_k(\varphi)\,  h_k(x)$
the decomposition of $\varphi$ into orthonormal Hermite polynomials, we thus obtain
\begin{align}
  \E\{x_{1,i}x_{1,j}|W\} = \sum_{k=0}^{\infty}\lambda_k^2(\varphi)(w_i^Tw_j)^k\, ,
\end{align}
Since $\lambda_0(\varphi)=\E\{\varphi(G)\}=0$, $\lambda_1(\varphi) = \E\{G\varphi(G)\}=0$, we have
$\E\{x_{1,i}x_{1,j}|W\} = O((w_i^Tw_j)^2)=O(1/d)$. In other words,  the population covariance $\E\{x_{1}x_{1}^T|W\}$ has small entries out-of diagonal, and in fact
$\|\E\{x_{1}x_{1}^T|W\} -I_p\|_{\op}=o_P(1)$ \cite{el2010spectrum}.

Despite the nonlinear model $x_i = \varphi(Wz_i)$ matches the second-order
\emph{population} statistics of
the isotropic model, it is far from obvious that the resulting least norm regression risk is the same.
Indeed the coordinates of vector $x_i$ are highly dependent, as it can be seen by noting
that they are a function of only $d\ll p$ independent random variables.
Theorem \ref{cor:purely_nonlinear} confirms that ---despite dependence--- the risk is asymptotically the same, thus providing a concrete example of the general universality
phenomenon.

Figure \ref{fig:purely_nonlinear} compares the asymptotic risk curve from
Theorem  \ref{cor:purely_nonlinear} to that computed by simulation, using an
activation function $\varphi_{\sabs}(t) = a(|t|-b)$, where
\smash{$a=\sqrt{\pi/(\pi-2)}$} and \smash{$b=\sqrt{2/\pi}$} are chosen to meet    
the standardization conditions. This activation function is purely nonlinear,
i.e., it satisfies \smash{$\E[G\varphi_{\sabs}(G)]=0$} for $G \sim N(0,1)$, by 
symmetry. Again, the agreement between finite-sample and asymptotic risks is
excellent. Notice in particular that, as predicted by Theorem  \ref{cor:purely_nonlinear}, the risk
depends only on $p/n$ and not on $d/n$. 

\subsection{Proof outline for Theorem \ref{cor:purely_nonlinear}}
\label{sec:nonlinear_proof}

We define $\gamma_n= p/n$ and $\psi_n= d/p$.  Recall that as $n,p,d \to
\infty$, we have $\gamma_n \to \gamma$ and $\psi_n \to \psi$.  To reduce
notational overhead, we will generally drop the subscripts from
$\gamma_n,\psi_n$, writing these simply as $\gamma,\psi$, since their meanings 
should be clear from the context.
Let $N=p+n$ and define the symmetric matrix $\bA(s) \in \reals^{N\times N}$, 
for $s  \geq 0$, with the block structure:
\begin{equation}
\bA(s) = \begin{bmatrix}
s\id_p & \frac{1}{\sqrt{n}} \bX^{\sT} \\ 
\frac{1}{\sqrt{n}} \bX & 0_n
\end{bmatrix},\label{eq:aden}
\end{equation}
where $I_p \in \R^{p\times p}$ and $0_n \in \R^{n\times n}$ are the identity and zero matrix, respectively. As we will see this matrix
allows to construct the traces of interest by taking suitable derivatives of its resolvent. 

We introduce the following resolvents (as usual, these are defined for
$\Im(\xi)>0$ and by analytic continuation, whenever possible, for $\Im(\xi) =
0$):  
\begin{align*}
m_{1,n}(\xi,s) &= 
\E\Big\{\big(\bA(s) - \xi\id_{N}\big)^{-1}_{1,1}\Big\} = 
\E M_{1,n}(\xi,s),\\   
 M_{1,n}(\xi,s) &=
\frac{1}{p}\Tr_{[1,p]}\Big\{\big(\bA(s) - \xi\id_{N}\big)^{-1}\Big\},\\  
m_{2,n}(\xi,s) &= 
\E\Big\{\big(\bA(s) -\xi\id_{N}\big)^{-1}_{p+1,p+1}\Big\} =
\E M_{2,n}(\xi,s),\\   
 M_{2,n}(\xi,s) &= 
\frac{1}{n}\Tr_{[p+1,p+n]}\Big\{\big(\bA(s) - \xi\id_{N}\big)^{-1}\Big\}.  
\end{align*}
Here and henceforth, we write $[i,j]=\{i+1,\ldots,i+j\}$ for integers $i,j$.
We also write \smash{$M^{-1}_{ij}= (M^{-1})_{ij}$} for a matrix $M$, and  
\smash{$\Tr_S(M) = \sum_{i\in S} M_{ii}$} for a subset $S$.  The equalities in 
the first and third lines above follow by invariance of the distribution of
$\bA(s)$ under permutations of $[1,p]$ and $[p+1,p+n]$. Whenever clear from
the context, we will omit the arguments from block matrix and resolvents, and 
write $\bA=\bA(s)$, $m_{1,n} =m_{1,n}(\xi,s)$, and $m_{2,n}
=m_{2,n}(\xi,s)$.    

The next lemma characterizes the asymptotics of $m_{1,n}$, $m_{2,n}$.
\begin{lemma}\label{thm:ResolventKernel}
Assume the conditions of Theorem \ref{cor:purely_nonlinear}.  Consider $\Im(\xi)>0$
or $\Im(\xi)=0$, $\Re(\xi)<0$, with $s\ge t\ge 0$. Let $m_1$ and $m_2$ be the
unique solutions of the following quadratic equations: 
\begin{align}
m_2= \big(-\xi-\gamma m_1\big)^{-1},\;\;\;\;\;\;\;\;
m_1=  \big(-\xi-s-m_2 \big)^{-1},\label{eq:M1}   
\end{align}
subject to the condition of being analytic functions for $\Im(z)>0$, and
satisfying $|m_1(z,s)|, |m_2(z,s)| \le 1/\Im(z)$ for $\Im(z)>C$ (with $C$ a
sufficiently large constant).  Then, as $n,p,d \to \infty$, such that $p/n \to
\gamma$ and $d/p \to \psi$, we have almost surely (and in $L^1$),
\begin{align}
\lim_{n,p,d\to\infty} M_{1,n}(\xi,s) &= m_1(\xi,s),\label{eq:AS-1}\\ 
\lim_{n,p,d\to\infty} M_{2,n}(\xi,s) &= m_2(\xi,s).\label{eq:AS-2}
\end{align}
\end{lemma}

The proof of this lemma is given in Appendix \ref{app:ResolventKernel}.
As a corollary of the above, we obtain that the asymptotical empirical spectral distribution of the empirical covariance
\smash{$\hSigma=X^T X/n$} matches the one for the independent entries model, and is hence given by the Marchenko-Pastur law
(a result already obtained in \citet{pennington2017nonlinear}).
We state this formally using the Stieltjes transform
\begin{equation}
R_n(z) = \frac{1}{p}\Tr\big((\hSigma-z\id_p)^{-1}\big),\label{eq:Sn}  
\end{equation}

\begin{corollary}\label{coro:Stieltjes}
Assume the conditions of Theorem \ref{cor:purely_nonlinear}.  Consider
$\Im(z)>0$. As $n,p,d \to \infty$, with $p/n \to \gamma$ and $d/p \to \psi$,
we have (almost surely and in $L^1$) $R_{n}(\xi)\to r(\xi)$
where $r$ is  nonrandom and coincides with the Stieltjes transform of the Marchenko-Pastur law, namely 
\begin{align}
  r(z) = \frac{1-\gamma-z -\sqrt{(1-\gamma-z)^2-4\gamma z}}{2\gamma z}\, \label{eq:coroS1}.
\end{align}
\end{corollary}

We refer to Appendix \ref{app:CoroStieltjes} for a proof of this corollary.
The next lemma connects the above resolvents computed in Lemma \ref{thm:ResolventKernel} to the variance of min-norm
least squares, hence finishing our proof outline.
\begin{lemma}\label{lemma:ResolventToVar}
Assume the conditions of Theorem \ref{cor:purely_nonlinear}.
Let $m_1$, $m_2$ be the asymptotic resolvents given in Lemma
\ref{thm:ResolventKernel}. Define
$$
m(\xi,s) = \gamma m_1(\xi,s) + m_2(\xi,s)\, .
$$
Then for $\gamma\neq 1$, $\partial_x m(\xi,x)\big|_{x=0}$ as a simple pole at $\xi=0$,
and hence admits a Taylor-Laurent expansion around
$\xi=0$, whose coefficients will be denoted by $D_{-1}$, $D_0$
\begin{align}
-\partial_x m(\xi,x)\big|_{x=0} 
= \frac{D_{-1}}{\xi^2}+D_0+O(\xi^2).\label{eq:Laurent}
\end{align}
Here each coefficient is a function of $\gamma,\psi$:
$D_{-1} = D_{-1}(\gamma,\psi)$, $D_{0} = D_{0}(\gamma,\psi)$. Furthermore, for the ridge regression 
estimator \smash{$\hbeta_\lambda$} in \eqref{eq:ridge}, as $n,p,d \to \infty$,
such that $p/n \to \gamma \in (0,\infty)$, $d/p \to \psi \in (0,1)$, the
following ridgeless limit holds almost surely:
$$
\lim_{\lambda \to 0^+} \, \lim_{n,p,d\to\infty} \, 
V_X(\hbeta_\lambda;\beta) = D_0.
$$
\end{lemma}

The proof of this lemma can be found in Appendix
\ref{app:ResolventToVar}. Theorem \ref{cor:purely_nonlinear} follows by evaluating 
the formula in Lemma \ref{lemma:ResolventToVar}, by using the  
result of Lemma \ref{thm:ResolventKernel}. We refer to the appendices material for details.

\subsection*{Acknowledgements}

The authors are grateful to Brad Efron, Rob Tibshirani, and Larry Wasserman for
inspiring us to work on this in the first place.  RT sincerely thanks Edgar
Dobriban for many helpful conversations about the random matrix theory
literature, in particular, the literature trail leading up to Proposition
\ref{thm:risk_lo}, and the reference to \citet{rubio2011spectral}. We are also grateful to Dmitry Kobak
for bringing to our attention the \cite{kobak2020optimal} and clarifying the implications
of this work on our setting.

TH was partially supported by grants NSF DMS-1407548, NSF IIS-1837931, and NIH 
5R01-EB001988-21.  AM was partially supported by NSF DMS-1613091, NSF
CCF-1714305, NSF IIS-1741162, and ONR N00014-18-1-2729.  RT was partially 
supported by NSF DMS-1554123.   

\appendix

\section{Proofs for the linear model: Risk of ridge regression}
\label{sec:ProofGeneralFormula}

\subsection{Proof of Theorem \ref{thm:risk_ridge}: Preliminaries}

Before passing to the proof of Theorem \ref{thm:risk_ridge}, we state and prove a useful  lemma in analysis.
\begin{lemma}\label{lemma:ConvDerivative}
  For any $k\in\naturals$, there exist  absolute constants $C_{1,k},C_{2,k}<\infty$, such that the following holds.
  Given any two functions $f,g:[x-\Delta,x+\Delta]\to \reals$, that are $(2k+1)$ times differentiable with bounded
  derivatives, and any $\delta\in (0,\Delta/k)$, we have
  \begin{align*}
\big|f'(x)-g'(x)\big|\le \frac{C_{1,k}}{\delta} \|f-g\|_{\infty}+ C_{2,k}\|f^{(2k+1)}-g^{(2k+1)}\|_{\infty}\delta^{2k}\, .
  \end{align*}
  \end{lemma}
  \begin{proof}
    Without loss of generality we can assume $x=0$. Further  we set $h=f-g$. 
    We consider a finite difference approximation of order $2k$ of the first order derivative of $h$ at $0$.
    One such approximation is given by \cite[Chapter 3]{trefethen1996finite}
    \begin{align*}
      D_{2k}h(0) = \sum_{m=1}^k\alpha_m\Big[\frac{h(m\delta)-h(-m\delta)}{2m\delta}\Big]\, ,\;\;\;\;
      \alpha_m:= 2(-1)^{m+1}\binom{2k}{k}^{-1}\binom{2k}{k-m}\, .
    \end{align*}
    The operator $D_{2k}$ is obviously linear.
    The following two properties are standard and easy to check (see below): $(i)$~If $h(x) = x$, $D_{2k}h(0) = 1$;
    $(ii)$~If $h(x) = x^\ell$ with $\ell\le 2k$ an integer, $\ell\neq 1$, then $D_{2k}h(x) = 0$.
    Using Taylor's expansion with remainder, we obtain
    \begin{align*}
      h(x) &= \sum_{\ell=0}^{2k}\frac{h^{(\ell)}(0)}{\ell!} x^{\ell}+\frac{1}{(2k)!}\int_0^x(x-t)^{2k}h^{(2k+1)}(t)\, dt\\
           & = \sum_{\ell=0}^{2k}\frac{h^{(\ell)}(0)}{\ell!} x^{\ell}+\frac{h^{(2k+1)}(u_x)}{(2k+1)!}x^{2k+1}\, ,
    \end{align*}
    where $u_x$ is a point in the interval $[0,x]$.
    Using the properties of $D_{2k}$ listed above, we obtain
    \begin{align*}
      \big|    D_{2k}h(0) -h'(0)\big| &= \left|\sum_{m=1}^k\alpha_m\frac{h^{(2k+1)}(u_{m\delta}) - h^{(2k+1)}(u_{-m\delta})}{2(2k+1)!}\cdot
      (m\delta)^{2k}\right|\\ & \le C_{2,k}\|h^{(2k+1)}\|_{\infty}\delta^{2k}\, .
    \end{align*}
    Therefore
    \begin{align*}
      |h'(0)|&\le  \big|    D_{2k}h(0) \big| + \big|    D_{2k}h(0) -h'(0)\big| \\
      &\le C_{1,k}\frac{\|h\|_{\infty}}{\delta}+ C_{2,k}\|h^{(2k+1)}\|_{\infty}\delta^{2k}\, .
    \end{align*}
    This proves the desired bound.

    We are left wih the task of proving claims $(i)$ and $(ii)$ about the operator $D_{2k}$. Substituting,
    these are equivalent to the claims
    \begin{align}
        \sum_{m=1}^k\alpha_m&=1\, ,\\
      \sum_{m=1}^k\alpha_m[m^{\ell-1}+(-m)^{\ell-1}]&=0\;\;\;\;\forall \ell\in\{2,\dots, 2k\}\, .
    \end{align}
    We will prove the second equation, since the first one is proved
    analogously. Notice that this claim is obviously true for $\ell-1$ odd, and we will therefore assume $\ell-1$ even. Setting $r=\ell-1$, and
    choosing $A_k=2^{-1}\binom{2k}{k}$, we have
    \begin{align*}
      A_k  &\sum_{m=1}^k\alpha_m[m^{r}+(-m)^{r}]\\
           &= \sum_{m=1}^k(-1)^{m+1}\binom{2k}{k-m}m^{r}+
                                               \sum_{m=1}^k(-1)^{m+1}\binom{2k}{k+m}m^{r}\\
                                             & = (-1)^{k-1}\sum_{j=0}^{2k}(-1)^j\binom{2k}{j}(k-j)^r \\
                                             & = (-1)^{k-1}\left.\frac{d^r\phantom{x}}{dx^r}\right|_{x=0} \sum_{j=0}^{2k}(-1)^j\binom{2k}{j} e^{(k-j)x}\\
                                             & = (-1)^{k-1}\left.\frac{d^r\phantom{x}}{dx^r}\right|_{x=0}(e^{x/2}-e^{-x/2})^{2k} = 0\, .
    \end{align*}
    Here the last identity holds for $0<r<2k$, i.e. $1<\ell\le 2k$ as claimed.
  \end{proof}

  We further state and proof a lemma about a modified Stieltjes transform.
  \begin{lemma}\label{lemma:Usol}
    Let $s_1\ge \dots\ge s_p> 0$ be such that  $s_1\le M$ and $\#\{i:\, s_i \ge 1/M\}\ge p/M$ for some large constant
    $M>1$, and $M^{-1}\le p/n\le M$. Further assume $|1-(p/n)|\ge  1/M$
    and $n^{-1}\sum_{i=1}^p s_i^{-1}1_{s_i>0}<M$.
      For $(\lambda,\eta)\in (0,M]\times (-1/(10M^2),1/(10M^2))$ let $u=u(\lambda,\eta)$ be the only solution in $(0,\infty)$
    of the equation
    \begin{align}
      \frac{u}{1-\eta u} = \lambda +\frac{1}{n}\sum_{i=1}^p\frac{s_iu}{u+s_i}\, .\label{eq:Ueta}
    \end{align}
    Then
    \begin{enumerate}
    \item[$(a)$]  $\lambda\mapsto u(\lambda,\eta)$ is analytic and non-decreasing with 
      $0<\partial_{\lambda}u(\lambda,\eta)\le C(M)$ for a constant $C(M)$ uniquely dependent on $M$.
    \item[$(b)$] $u(\lambda,\eta) >\lambda/2$.
      \item[$(c)$] If further  $\lambda\vee s_p\ge 1/M$, then  $|\partial^k_{\eta} u(\lambda,\eta)|\le C_k(M)$ for all $k$.
    \end{enumerate}
  \end{lemma}
  \begin{proof}
    We begin from poinr $(a)$. Since $\eta$ is fixed here, we write $u(\lambda) = u(\lambda,\eta)$, dropping the dependence on $\eta$.
    Note that, for $\eta>0$, any solution $u$ must be such that $u\in (0,1/\eta)$ because otherwise the left hand side is negative
    and the right-hand side is positive. Call $x=u/(1-\eta u)$ whence the above equation is equivalent to
    the following equation on $x\in (0,\infty)$:
    \begin{align}
       x = \lambda +\frac{1}{n}\sum_{i=1}^p\frac{s_i(\eta)x}{s_i(\eta)+x}\, ,\label{eq:FixPointX}
    \end{align}
    where $s_i(\eta) = s_i/(1+\eta s_i)$. For $\eta\in (-1/(10M^2),1/(10M^2))$, we have $(1-(10M)^{-1})s_i<s_i(\eta)<(1+(10M)^{-1})s_i$,
    and therefore the sequence $(s_i(\eta))_{i\le p}$ satisfies the same assumptions as $(s_i)_{i\le p}$, with $M$ replaced by
    $2M$. We will therefore drop the argument $\eta$ and study the fixed point equation \eqref{eq:FixPointX}
    (which is equivalent to the case $\eta=0$ of Eq.~\eqref{eq:Ueta}).

    Let $g(x) := n^{-1}\sum_{i=1}^ps_i/(s_i+x)$, and note that this function is mononone decreasing on $[0,\infty)$
    with $g(0)\le C_0(M)$, and $g(x)\le M^2/x$. Equation \eqref{eq:FixPointX} can be rewritten as
    \begin{align}
       1 = \frac{\lambda}{x} +g(x) \, ,\label{eq:FixPointX2}
    \end{align}
    Since $x\mapsto \hat g(x) := (\lambda/x)+g(x)$ is monotone decreasing with $\hat g(0+)= +\infty$ and $\hat g(x)\to 0$
    as $x\to\infty$, this equation has a unique solution, denoted by $x(\lambda)$.
    By the implicit function theorem, we have
    \begin{align}
      \partial_{\lambda}x(\lambda) = \left[1-\frac{1}{n}\sum_{i=1}^p\frac{s_i^2}{(s_i+x(\lambda))^2}\right]^{-1}>0\, .\label{eq:DeX}
    \end{align}
    The second inequality follows for instance by noting that $n^{-1}\sum_{i=1}^ps_i^2/(s_i+x(\lambda))^2<g(x(\lambda))=1-\lambda/x(\lambda)<1$.
    Notice that $g(0)=p/n$. By assumption $|g(0)-1|\ge 1/M$. We will distinguish the two cases
    $g(0)<1$ and $g(0) >1$.

    \noindent\emph{Case $g(0)<1$.}
 Since $g(x)$ is decreasing, \eqref {eq:FixPointX2} implies $1\le \lambda x^{-1}+g(0)$
    whence
    \begin{align}
      x(\lambda)\le \frac{\lambda}{1-g(0)}  \le M\lambda\, .
   \end{align}
    Further, in this case
    \begin{align}
      \frac{1}{n}\sum_{i=1}^p\frac{s_i^2}{(s_i+x(\lambda))^2} \le g(x(\lambda))\le  g(0) \le 1-\frac{1}{M}\, ,
    \end{align}
    and therefore, by Eq.~\eqref{eq:DeX}, we have  $\partial_{\lambda}x(\lambda) \le M$ for all $\lambda\in(0,\infty)$.

    \noindent\emph{Case $g(0)>1$.} Let $x_0=g^{-1}(1)>0$. Since $x\mapsto g(x)$ is convex, $x_0\ge (1-g(0))/|g'(0)|$.
    Note that $|g'(0)|=n^{-1}\sum_{i=1}^p s_i^{-1}1_{s_i>0}\le M$. Therefore $x_0\ge 1/C(M)$ for some finite constant $C(M)$.
    
    Since $\hat g(x) = \lambda/x+g(x)>g(x)$ on $(0,\infty)$,
    we have $x(\lambda)>x(0+)=x_0\ge 1/C(M)$ for all $\lambda\in(0,\infty)$. Further,
    \begin{align*}
      \frac{1}{n}\sum_{i=1}^p\frac{s_i^2}{(s_i+x(\lambda))^2} &\le   \frac{1}{n}\sum_{i=1}^p\frac{s_i}{s_i+x(\lambda)}
                                                               -\frac{1}{n}\sum_{i=1}^{p/M}\frac{s_i}{s_i+x(\lambda)}\Big[1-\frac{s_i}{s_i+x(\lambda)}\Big]\\
                                                             &\le 1-\frac{\lambda}{x(\lambda)}-\frac{x_0}{ x_0+M} \frac{1}{n}\sum_{i=1}^{p/M}\frac{s_i}{s_i+x(\lambda)}\\
                                                             &\le 1-\frac{\lambda}{x(\lambda)} -\frac{ x_0}{x_0+M} \cdot \frac{1}{M}\cdot  \frac{1}{n}\sum_{i=1}^{p}\frac{s_i}{s_i+x(\lambda)}\\
      &\le \Big(1-\frac{\lambda}{x(\lambda)}\Big)\Big(1-\frac{x_0}{M(x_0+M)}\Big) \le 1-\frac{1}{C(M)}\, .
    \end{align*}
    Therefore by Eq.~\eqref{eq:DeX},  $0<\partial_{\lambda}x(\lambda)\le C(M)$ also in this case.

    Finally note that by Eq.~\eqref{eq:FixPointX}, we have
    \begin{align}
      x(\lambda)\le \lambda+\frac{1}{n}\sum_{i=1}^p s_i(\eta) \le M+2M^2\le 3M^2\, .\label{eq:BoundXX}
    \end{align}
    and therefore $7/10\le (1+\eta x(\lambda))\le 13/10$, whence
    \begin{align}
      \partial_{\lambda} u(\lambda) = \frac{1}{(1+\eta x(\lambda))^2}\partial_{\lambda} x(\lambda)\le 2\partial_{\lambda}x(\lambda)\le C(M)\, ,.
    \end{align}
   This proves point $(a)$.

   Claim $(b)$ simply follows since $u/(1-\eta u)\ge \lambda$ by Eq.~\eqref{eq:Ueta}. Since $(1-\eta u) = 1/(1+\eta x)\ge 1/2$ as just shown,
   this implies $u \ge \lambda/2$.
   
    Finally, for claim $(c)$, notice that $u(\lambda,\eta)$ is a solution of the equation
    \begin{align}
      0= f(u,\eta):= \frac{u}{1-\eta u}-\lambda-\frac{1}{n}\sum_{i=1}^p \frac{s_iu}{u+s_i}\, .
    \end{align}
    By the implicit function theorem, it is sufficient to prove that
    $\big|\frac{\partial^{k+l}f}{\partial u^{k}\partial\eta^{l}}\big|\le C_{k,l}(M)$ for all $l\ge 1, k\ge 0$ and $l=0$, $k\ge 2$,
  and that $\big|\frac{\partial f}{\partial u}\big|\ge 1/C(M)$.

  In order to bound the first derivative away from $0$, note that, again by the implicit function theorem,
    \begin{align*}
      \Big|\frac{\partial f}{\partial u}\Big| = \Big|\frac{\partial u}{\partial \lambda} \Big|^{-1}\ge \frac{1}{C(M)}\, ,
    \end{align*}
    where the last inequality follows from point $(a)$ above.
    In order to bound higher order derivatives, write $f(u,\eta) = f_1(\eta u)/\eta-\lambda-f_2(u)$ where
    $f_1(v) := v/(1-v)$.
    Recall that $x= u/(1-\eta u)$, and $x\le 3M^2$ by  Eq.~\eqref{eq:BoundXX}.
    Therefore
    \begin{align*}
      0\le \partial^k f_1(\eta u) = \frac{1}{(1-\eta u)^{k+1}}= (1+\eta x)^{k+1}\le C_k(M)\, .
    \end{align*}
    Further
    \begin{align*}
      \big| \partial^k f_2(u) \big| &= \Big|\frac{1}{n}\sum_{i=1}^p\frac{s_i^2}{(s_i+u)^{k+1}}\Big|\\
                                       &\le \frac{p}{n} \frac{s_1^2}{(s_p+u)^{k+1}}\\
                                       &   \le  \frac{p}{n} \frac{s_1^2}{(s_p+C(M)\lambda)^{k+1}}\le C_k(M)\, .
    \end{align*}
    We therefore obtain, for $k\ge 2$
    \begin{align*}
      \left|\frac{\partial^k f}{\partial u^k}(u,\eta)\right| \le \eta^{k-1}\big|\partial^k f_1(\eta u)\big|+\big|\partial^k f_2( u)\big|\le C_k(M)\, .
    \end{align*}
    For the other derivvatives, notice that
    \begin{align*}
      \frac{\partial^{k+l} f}{\partial u^k\partial\eta^{l}}(u,\eta) = \sum_{m_1,m_2,m_3\le C_0}c_1(m_1,m_2,m_3)
      \eta^{m_1}u^{m_2}\partial^{m_3} f_1(\eta u)\, ,
      \end{align*}
      where $C_0$, $c_1(m_1,m_2,m_3)$ are coefficients depending uniquly on $k,l$. 
This sum is bounded since each term is bounded, by the above estimate on $\partial^{k} f_1$, and
     further using the fact that
    $u = x/(1+\eta x)\le 10 M^2$ since $x\le 3M^2$ by  Eq.~\eqref{eq:BoundXX}.
  \end{proof}

 For the proof of  Theorem \ref{thm:risk_ridge}, it 
 is convenient to introduce the notations $S_Z := Z^TZ/n$ and $S_X := X^TX/n= \Sigma^{1/2}S_Z\Sigma^{1/2}$.
We begin by recalling the expressions for bias and variance of ridge regression at regularization parameter $\lambda$:
\begin{align}
  B_X(\hbeta;\beta) &= \lambda^2\<\beta,\big(S_X+\lambda I\big)^{-1}\Sigma \big(S_X+\lambda I\big)^{-1}\beta\>\, ,
  \label{eq:BX}\\
  V_X(\hbeta;\beta) &= \frac{\sigma^2}{n} \Tr\big(\Sigma S_X\big(\lambda I+S_X\big)^{-2}\big)\, .\label{eq:VX}
\end{align}
The bias and variance for min-norm interpolation are recovered by taking $\lambda\to 0+$. 
In the following, we will occasionally consider $\lambda$ to be complex with $\Re(\lambda)>0$.

\subsection{Proof of Theorem \ref{thm:risk_ridge}, bias term}

Note that the bias is homogeneous (of degree $2$) in
$\|\beta\|_2$. Hence, without loss of generality we can and will assume $\|\beta\|_2=1$.
For  $\Re(\eta)> - 1/\lambda_{\max}(\Sigma)$ define
\begin{align}
  \oF_n(\eta,\lambda) & := \lambda\<\beta,\big(S_X+\lambda I+\lambda\eta \Sigma\big)^{-1}\beta\>\label{eq:FirstG}\\
      &= \lambda\<\beta_{\eta}, \big(\Sigma_{\eta}^{1/2}S_Z\Sigma_{\eta}^{1/2}+\lambda I\big)^{-1}\beta_{\eta}\>\, ,\label{eq:SecondG}
\end{align}
where
\begin{align}
  \Sigma_{\eta}:= \Sigma(I+\eta\Sigma)^{-1}\,,\;\;\;\; \beta_{\eta}:= \big(I+\eta\Sigma\big)^{-1/2}\beta\, .
  \label{eq:SigmaEta}
\end{align}
Recall that $0\preceq \Sigma\preceq M\cdot I_p$, and therefore  $\oF_n$ is an analytic function in $\bbD := \{(\lambda,\eta)\in\complex^2:\; \Re(\lambda)>0, \Re(\eta)>-1/(2M)\}$.
For future reference, we also define $\bbD_0 := \{(\lambda,\eta)\in\complex\times\reals
:\; 0<\Re(\lambda)<M, |\Im(\lambda)|<M,
\Re(\eta)>-1/(2M),\}$.
Using the  expressions \eqref{eq:FirstG} it is easy to see that
\begin{align}
  B_X(\hbeta;\beta) & = -\frac{\partial \oF_n}{\partial\eta}(0,\lambda)\, .\label{eq:BiasDG}
\end{align}
Note that Eq.~\eqref{eq:SecondG} has the form of the matrix element of a resolvent. By using
the anisotropic local law for covariance matrices of \cite[Theorem 3.16.$(i)$]{knowles2017anisotropic} (see also Remark 3.17 in the same paper), we obtain that for any $\eps>0$, $\eps_0>0$, $D>0$, there exist $C=C(\eps,\eps_0,D)$ such that
the following holds with probability at least $1-Cn^{-D}$. For all $(\lambda,\eta)\in\bbD_0$, $0<\Im(-\lambda)<M$,
$\Re(\lambda)> n^{-2/3+\eps_0}$, we have
\begin{align}
  &\left|\oF_n(\lambda,\eta)  - F_n(\lambda,\eta)\right|
  \le \sqrt{\frac{\Im (r_n(-\lambda,\eta))}{\Im(-\lambda)} \cdot n^{-1+\eps}}\, ,\label{eq:GF}\\
   &F_n(\lambda,\eta): =-\<\beta_{\eta},\big(I+r_n(-\lambda,\eta)\Sigma_{\eta}\big)^{-1}\beta_{\eta}\>\, .
  \end{align}
  Here, for $\Im(z)>0$, $\Re(\eta)=0$, $r_n=r_n(z,\eta)$ is defined as the unique solution $\Im(r_n(z,\eta))>0$  of
  \begin{align}
    \frac{1}{r_n} = -z+\gamma\, \frac{1}{p}\sum_{i=1}^p\frac{s_{i}(\eta)}{1+s_i(\eta)r_n}\, .\label{eq:RDef}
  \end{align}
  where $s_1(\eta)\ge s_2(\eta)\ge\dots\ge s_p(\eta)$ are the eigenvalues of $\Sigma_{\eta}$.
  By Eq.~\eqref{eq:SigmaEta} we have $s_i(\eta)=s_i/(1+\eta s_i)$. 
  Existence and uniqueness of the solution $r_n(z,\eta)$ is given, for instance, in \cite[Lemma 2.2]{knowles2017anisotropic}.
  This lemma also states that $r_n(z,\eta)$ is the Stieltjes transorm of a probability measure
  $\rho_{\eta}$ with support in $[0,C(M)]$. As a consequence, for $\lambda= \lambda_1+i\lambda_2$, $\lambda_1>0$
  \begin{align*}
    \Im(r_n(-\lambda,\eta)) &= \Im \int \frac{1}{x+\lambda}\, d \rho_{\eta}(x) =
                              \int \frac{-\lambda_2}{(x+\lambda_1)^2+\lambda_2^2}\, d \rho_{\eta}(x) \, .\\
   \Rightarrow &  \big|\Im(r_n(-\lambda,\eta)) \big| \le \frac{|\Im(\lambda)|}{\Re(\lambda)^2}\, .
  \end{align*}

  Therefore, taking the limit $\Im(\lambda)\to 0$ in Eq.~\eqref{eq:GF}, we obtain,
  with probability at least $1-Cn^{-D}$
  \begin{align}
 \left|\oF_n(\lambda,\eta)  - F_n(\lambda,\eta)\right|
    \le \frac{1}{n^{(1-\eps)/2}\lambda} \;\;\;\; \forall \lambda\in(n^{-2/3+\eps_0},\infty), \;\;\eta\in \Big(-\frac{1}{2M},\infty\Big)\,. \label{eq:GFReal}
    \end{align}

  Since $s_i(\eta)$ is differentiable in $\eta$ for $(\lambda,\eta) \in \bbD_0$, it follows that $r_n(z,\eta)$ is also
  differentiable in $\eta$. Defining $\om_n(z,\eta) = (1-\gamma+zr_n(z,\eta))/(\gamma z)$, we obtain that $\om_n$ is the unique solution
  \begin{align}
    \om_n =\frac{1}{p}\sum_{i=1}^p\frac{1}{s_i(\eta)[1-\gamma-\gamma z \om_n]-z}\, ,
    \end{align}
    with $s_i(\eta)=s_i/(1+\eta s_i)$.   Since $\eta=0$
    is in the domain $\bbD_0$, we can Taylor expand the solution of this equation, to obtain
    \begin{align}
\om_{n}(z,\eta) = m_n(z) + m_{n,1}(z)\eta+O(\eta^2)\, ,\label{eq:OmExpansion}
    \end{align}
    where $m_n(z)$, $m_{n,1}(z)$ are defined in Eqs.~\eqref{eq:m0def} and \eqref{eq:m1def}.
    Since $r_n(-\lambda,\eta)$ is analytic for $(\lambda,\eta)\in \bbD$, we also have that $F_n(\lambda,\eta)$ is analytic in the same  domain. In particular, we can differentiate  $F_n(\lambda,\eta)$  with respect to $\eta$
    at $\eta=0$. Using Eq.~\eqref{eq:OmExpansion}, we get
    \begin{align}
      \frac{\partial F_n}{\partial \eta}(\lambda,0)
      & = -\lambda^2\big(1+\gamma m_{n,1}(-\lambda)\big)\<\beta,\big(\lambda I+(1-\gamma+\gamma\lambda m_n(-\lambda))\Sigma    \big)^{-2}\Sigma\beta\>\, \nonumber\\
      & = -\lambda^2\|\beta\|^2 \big(1+\gamma m_{n,1}(-\lambda)\big) \int \frac{s }{[\lambda +(1-\gamma+\gamma\lambda m_n(-\lambda))s]^2} d\hG_n(s)\nonumber\\
      &= - \cuB(\lambda;\hH_n,\hG_n,\gamma)\, .\label{eq:DFBias}
    \end{align}
    Here the last identity is simply the definition of the predicted bias in Eq.~\eqref{eq:BiasRidge}.
 
    Notice that $\eta\mapsto \oF(\lambda,\eta)$ is infinitely differentiable in $(-1/(2M),\infty)$ with
    \begin{align*}
      \frac{\partial^k \oF_n}{\partial\eta^k}(\lambda,\eta) &= (-1)^k\lambda^{k+1}
      \<\beta, R (\Sigma R)^{k}\beta\>\, ,\\
      R&:=(S_X+\lambda I+\lambda\eta\Sigma^{-1})^{-1}\, ,
    \end{align*}
    Note that  $\|\Sigma\|_{op}\le M$, $\|\beta\|_{2}=1$ (we assumed this to hold without loss of generality)
    and $\|R\|_{\op}\le 1/\lambda$, whence
\begin{align}
  \left|\frac{\partial^k \oF_n}{\partial\eta^k}(\lambda,\eta) \right|\le \lambda^{k+1}\|\beta\|^2\|\Sigma\|_{op}^k\|R\|^{k+1}\le M^k\, . \label{eq:BoundDerivativesG}. 
  \end{align}

  We claim that a similar bound holds for the derivative of $F_n$, namely
  \begin{align}
  \left|\frac{\partial^k F_n}{\partial\eta^k}(\lambda,\eta) \right|\le C_k(M)\, .\label{eq:BoundDerivativesF}
  \end{align}
  In order to prove this bound, we write
   \begin{align*}
     F_n(\lambda,\eta) &= \cF(u_n(\eta,\lambda))\, ,\;\;\;\;\; u_n(\eta,\lambda) := (\eta+r_n(\eta,-\lambda))^{-1}\, ,\\
    \cF(u) &:= -u \<\beta,\big(uI_p+\Sigma\big)^{-1}\beta\>\, . 
   \end{align*}
   In what follows we use the shorhand $u_n = u_n(\eta,z)$. Note that, by Eq.~\eqref{eq:RDef}, $u_n$
   solves Eq.~\eqref{eq:Ueta}, and we can therefore apply Lemma \ref{lemma:Usol}.
   Note that
   \begin{align*}
     \Big|\frac{\partial^k\cF}{\partial u^k}(u_n) \Big|&= \<\beta,(u_nI_p+\Sigma)^{-k-1}\Sigma\beta\>\\
     &\le M(u_n+\lambda_{\min}(\Sigma))^{-k-1} \\
     & \stackrel{(a)}{\le} C_0(M) (\lambda \vee \lambda_{\min}(\Sigma))^{-k-1}\stackrel{(b)}{\le} C(M)\, .
   \end{align*}
   where $(a)$ follows from Lemma \ref{lemma:Usol}.$(b)$ and $(b)$ by the theorem's assumption.

   Further, by Lemma \ref{lemma:Usol}.$(c)$ we have $|\partial^k_{\eta}u_n|\le C'_k(M)$ for all $k$,
   whence the desired claim \eqref{eq:BoundDerivativesF} follows.

     Using   Eqs.~\eqref{eq:BoundDerivativesG}, \eqref{eq:BoundDerivativesF}, and \eqref{eq:GFReal} in 
     Lemma \ref{lemma:ConvDerivative} we get, for any $\eps>0$, $k$ and any $\delta<k/(4M)$, with probability at least $1-n^{-C}$
     \begin{align}
     \left|\frac{\partial \oF_n}{\partial \eta}(\lambda,0)  -\frac{\partial F_n}{\partial \eta}(\lambda,0) \right|
       \le \frac{1}{n^{(1-\eps)/2}\lambda} \cdot\frac{1}{\delta} +C_{2k+1}(M)\delta^{2k}\, .
     \end{align}
     The claim follows by taking $\delta = n^{-\eps/2}$ and $k= \lceil 1/\eps\rceil$, and recalling that,
     by Eqs.~\eqref{eq:BiasDG} and \eqref{eq:DFBias}, we have
     \begin{align}
     \left|\frac{\partial \oF_n}{\partial \eta}(\lambda,0)  -\frac{\partial F_n}{\partial \eta}(\lambda,0) \right| = 
       \left| B_X(\hbeta;\beta)-\cuB(\lambda;\hH_n,\hG_n,\gamma)\right|\, .
       \end{align}

\subsection{Proof of Theorem \ref{thm:risk_ridge}, variance term}

Note that by Eq.~\eqref{eq:VX}
\begin{align}
  V_X(\hbeta;\beta) &= \sigma^2\gamma\, \frac{\partial\phantom{\lambda}}{\partial\lambda}\, \Big\{\frac{\lambda}{p}
                      \Tr\big(\Sigma(S_X+\lambda I)^{-1}\big)\Big\}\, .
\end{align}
Further the right-hand side is an analytic function of $\lambda$, for $\Re(\lambda)>0$.
Using again \cite[Theorem 3.16.$(i)$]{knowles2017anisotropic} , we get that for any $D>0$, $\eps>0$, where
exists $C=C(D,\eps)$ such that, with probability at least $1-Cn^{-D}$,
\begin{align}
  &\left|\frac{\lambda}{p}\Tr\big(\Sigma(S_X+\lambda I)^{-1}\big)-L_n(\lambda) \right|\le
  \sqrt{\frac{\Im(r_n(-\lambda))}{\Im(-\lambda) n^{1-\eps}}}\, ,\\
  &L_n(\lambda) := \frac{1}{p}\Tr\big(\Sigma(I+r_n(-\lambda)\Sigma)^{-1}\big)\, .
\end{align}
uniformly over $\Re(\lambda)>n^{-2/3+(1/M)}$, $\Im(-\lambda)>0$. 
Here $r_n(z)$ is defined as in the previous section, namely as the  solution of Eq.~\eqref{eq:RDef} with $\eta=0$.
As shown there $\Im(r_n(-\lambda))|\le |\Im(\lambda)|/\Re(\lambda)^2$, whence
\begin{align}
\left|\frac{\lambda}{p}\Tr\big(\Sigma(S_X+\lambda I)^{-1}\big)-L_n(\lambda) \right|\le
  \frac{1}{|\Re(\lambda)|n^{(1-\eps)/2}}\, .
  \end{align}
  Since both functions $\lambda\mapsto \frac{\lambda}{p}\Tr\big(\Sigma(S_X+\lambda I)^{-1}$ and $\lambda\mapsto L_n(\lambda)$
  are analytic on $\Re(\lambda)>0$, the last bound implies a similar bound on their derivarives, namely
\begin{align}
\left| V_X(\hbeta;\beta) -\sigma^{2}\gamma \frac{\partial L_n}{\partial\lambda}(\lambda) \right|\le
  \frac{1}{|\Re(\lambda)|^2n^{(1-\eps)/2}}\, .
  \end{align}

  We are left with the task of computing the derivative of $L_n$:
  \begin{align}
    \frac{\partial L_n}{\partial\lambda}(\lambda) &:= \frac{1}{p}\Tr\big(\Sigma^2 r'_n(-\lambda)(I+r_n(-\lambda)\Sigma)^{-2}\big)\\
                                                  & = \int
                                                    \frac{s^2r'_n(-\lambda)}{[1+r_n(-\lambda)s]^2}\, d\hH_n(s)\\
    &=\int \frac{s^2(1-\gamma+\gamma\lambda^2 m'_n(-\lambda))}{[\lambda+s(1-\gamma+\gamma \lambda m_n(-\lambda))]^2} d\hH_n(s)\, .
  \end{align}
  where the last equality follows by the definition $m_n(x) = (1-\gamma+\gamma z r_n(z))/(\gamma z)$ given above.
  
 \subsection{Generalization to heavier tail covariates: Proof of Theorem \ref{thm:risk_ridge_asymp}}

 This result follows from Theorem \ref{thm:risk_gen} via a standard
 truncation argument. We outline the argument for the bias as the same proof holds almost unchanged for the variance term.
  We write the bias $B_{X,\Sigma}(\hbeta;\beta)$ to emphasize its dependende upon the
 population covariance. Using Eq.~\eqref{eq:BX}, we obtain, for some $C=C(\lambda)<\infty$
 \begin{align}
   \big|B_{X_1,\Sigma_1}(\hbeta_{\lambda};\beta) -B_{X_2,\Sigma_2}(\hbeta_{\lambda};\beta)\big|&\le C\|S_{X_1}-S_{X_2}\|_{op}+C\|\Sigma_1-\Sigma_2\|_{op}\\
   &\le \frac{C}{n}(\|X_1\|_{op}+\|X_2\|_{op})\cdot\|X_1-X_2\|_{op} ++C\|\Sigma_1-\Sigma_2\|_{op}\, .\label{eq:BiasContinuity}
 \end{align}
 For $M>0$, decompose the  random vriables $z_{ij}$ as
 \begin{align*}
   z_{ij} &= a_Mz^{M}_{ij} + \tz_{ij}^M\, ,\\
   \tz^M_{ij}&:=  z_{ij}1_{|z_{ij}|>M} - \E\{z_{ij}1_{|z_{ij}|>M} \}\, ,\\
   z^M_{ij}&:=  \frac{1}{a_M}\big(z_{ij}1_{|z_{ij}|\le M} - \E\{z_{ij}1_{|z_{ij}|\le M} \}\big)\, ,\\
   a_M&:= \E\{\big(z_{ij}1_{|z_{ij}|\le M} - \E\{z_{ij}1_{|z_{ij}|\le M} \}\big)^2\}^{1/2}\, .
 \end{align*}
 Notice that $z_{ij}^{M}$,  $\tz_{ij}^{M}$  have zero mean and that  $z_{ij}^{M}$ is bounded by $2M$ and has unit variance. Further,
 by the condition $\E\{|z_{ij}|^{4+\delta}\}<C$, we also have $|a_M-1|\le \eps_M$, and $\E\{|\tz^M_{ij}|^4\}\le \eps_M^4$,
 for some $\eps_M\to 0$ as $M\to\infty$.

 We correspondingly decompose the matricex $X,Z$ as
 \begin{align}
   Z&= a_MZ_M+\tZ_M\, \\
   X & = X_M+\tX_M\, ,\;\;\;\; X_M =Z_M\Sigma_M^{1/2}\, ,\; \tX_M = \tZ_M\Sigma^{1/2}\, ,
 \end{align}
 where $\Sigma_M = a_M^2\Sigma$. Using Eq.~\eqref{eq:BiasContinuity}, we get
 \begin{align}
   \big|B_{X,\Sigma}(\hbeta_{\lambda};\beta) -B_{X_M,\Sigma_M}(\hbeta_{\lambda};\beta)\big|&\le
                                                                                             C\big(\|X\|_{op}+\|X_M\|_{op}\big)\|\tX_M\|_{op}+C\|\Sigma_M-\Sigma\|_{op}\\
   &\le C\big(\|Z\|_{op}+\|Z_M\|_{op}\big)\|\tZ_M\|_{op}+C|a_M^2-1|\, .                                                      
 \end{align}
 By the Bai-Yin law \cite{bai1993limit}, almost surely $\lim\sup_{n\to\infty} (\|Z\|_{op}+\|Z_M\|_{op})\le C<\infty$ and $\lim\sup_{n\to\infty}\|\tZ_M\|_{op}<C\eps_M$. Therefore
 \begin{align}
   \lim\sup_{n,p\to\infty}\big|B_{X,\Sigma}(\hbeta_{\lambda};\beta) -B_{X_M,\Sigma_M}(\hbeta_{\lambda};\beta)\big|\le C\eps_M\, .
   \label{eq:Truncation1}
 \end{align}

 Now we can apply Theorem \ref{thm:risk_ridge} to $X_M,\Sigma_M$, since the variables $z^M_{ij}$ have bounded moments of all orders.
 Denoting by $\hH_n^M$, $\hG_n^M$ the corresponding empirical measures, we have, with probability at least $1-Cn^{-2}$:
 \begin{align*}
   \big| B_{X_M,\Sigma_M}(\hbeta_{\lambda};\beta)-\cuB(\lambda;\hH_n^M,\hG^M_n,\gamma)\big|\le \frac{C}{n^{0.49}}\, .
 \end{align*}
 Note that $\hH^M_n(s) = \hH_n(s/a^2_M)$ and $\hG^M_n(s) = \hG_n(s/a^2_M)$, hence these measures converge weakly to
 $H^M(s) = H(s/a^2_M)$ and $G^M(s) = G(s/a^2_M)$. Therefore, by Borel-Cantelli
 \begin{align}
   \lim\sup_{n\to\infty}\big| B_{X_M,\Sigma_M}(\hbeta_{\lambda};\beta)-\cuB(\lambda;H^M,G^M,\gamma)\big| = 0\, .
     \label{eq:Truncation2}
 \end{align}
 The proof is completed by putting together Eqs.~\eqref{eq:Truncation1} and \eqref{eq:Truncation2},
 taking $M\to\infty$ and noting that $H^M\Rightarrow H$, $G^M\Rightarrow G$ in that limit.
 
\section{Proofs for the linear model: Risk min-norm regression}
\label{sec:ProofGeneralFormulaMinNorm}

\subsection{Proof of Theorem \ref{thm:risk_gen}}

The proof follows from Theorem \ref{thm:risk_ridge} by approximating min-norm regression with ridge regression with a small value
of $\lambda$. For this reason we will assume $\lambda\le 1$ throughout.
We treat separately the bias and variance terms. Since the variance is independent of $\beta$ and the bias is homogeneous in $\|\beta\|_2$,
we can assume, without loss of generality, $\|\beta\|_2\le 1$.

\noindent\emph{Bias term.} Recall the notation $S_X=  X^TX/n$. Denote the eigenvalue decomposition of $S_X$ by $S_X= UD_XU^T$, where
$D_X\in\reals^{p\times p}$ is diagonal and $U\in\reals^{p\times p}$ is orthogonal. Using  Eq.~\eqref{eq:BX},
we can rewrite  the bias of min norm
and ridge regression  as (here $1_{D_X=0}$ is the diagonal matrix with $(i,i)$-th entry equal to $1$ if $(D_X)_{ii}=0$ and
equal to $0$ otherwise)
\begin{align}
  B_X(\hbeta;\beta) &=\big\|\Sigma^{1/2}U1_{D_X=0}U^T\beta\big|_2^2 \, \label{eq:BX0Formula}\\
  B_X(\hbeta_{\lambda};\beta) &=\big\|\Sigma^{1/2}U\lambda(\lambda I_p+D_X)^{-1}U^T\beta\big|_2^2 \, . \label{eq:BXFormula}
\end{align}
Using triangular inequality, we get
\begin{align*}
  \big| B_X(\hbeta_{\lambda};\beta)^{1/2}- B_X(\hbeta;\beta)^{1/2}\big| &\le \big\|\Sigma^{1/2}U\big[\lambda(\lambda I_p+D_X)^{-1}-1_{D_X=0}\big]U^T\beta\big\|_2\\
                                                           &\le \|\Sigma\|^{1/2}_{op}\big\|\lambda(\lambda I_p+D_X)^{-1}1_{D_X>0}\big\|_2\\
  &\le \frac{M\lambda}{\sigma_{\min}(X)^2/n}\le C(M) \frac{\lambda}{\sigma_{\min}(Z)^2/n}\, ,
\end{align*}
where $\sigma_{\min}$ denotes the smallest non-vanishing singular value and
the last inequality holds since by assumption $\lambda_{\min}(\Sigma)\ge 1/M$.
Finally, since $|p/n-1|\ge 1/M$, we have that $\sigma_{\min}(Z)/\sqrt{n}\ge 1/C(M)$ with probability at least $1-Cn^{-D}$
\cite{bai1993limit}. Hence, with the same probability  $\big| B_X(\hbeta_{\lambda};\beta)^{1/2} -B_X(\hbeta;\beta)^{1/2}\big| \le C(M)\lambda$.
Also notice that, by Eq.~\eqref{eq:BXFormula},  $B_X(\hbeta;\beta)\le \|\Sigma\|_{\op}\le M$.
Therefore, we obtain
\begin{align}
  \big| B_X(\hbeta_{\lambda};\beta) -B_X(\hbeta;\beta)\big| &\le C(M)\lambda\, .\label{eq:BBdiff1}
\end{align}

We next claim that  Eqs.~\eqref{eq:BiasPredMinNorm} and \eqref{eq:BiasRidge} imply the bound
\begin{align}
  \big| \cuB(\lambda;\hH_n,\hG_n,\gamma) -\cuB(\hH_n,\hG_n,\gamma)  \big| &\le C(M)\lambda\, . \label{eq:BBdiff2}
\end{align}
In to prove this estimate, recall that $\gamma>1+M^{-1}$, and that, by assumption,
$\hH_n$ is supported on $[M^{-1},M]$. Define
$c_*(\lambda)$ via $m_n(-\lambda) = (1-\gamma-\lambda\gamma c_*(\lambda))/(-\gamma\lambda)$.
Note that $c_*(\lambda) = r_n(-\lambda)/\gamma$, where $r_n(z)$ is the compaion Stieltjes transform already
introduced in Eq.~\eqref{eq:RDef}. In particular, by \cite[Lemma 2.2]{knowles2017anisotropic}, $c_*$ it is
non-negative and monotone decreasing in $\lambda$. 
Further Eq.~\eqref{eq:m0def}, $c_*$ satisfies
\begin{align}
  1-\frac{1}{\gamma}=-\lambda c_*+\int \frac{1}{1+c_*\gamma s} \, d \hH_n(s)=:f(c_*;\lambda)\, .
\end{align}
It follows from ${\rm supp}(\hH_n)\in [M^{-1},M]$   that $f(x;\lambda)$ is monotone decreasing on $(0,\infty)$
for any $\lambda\ge 0$ with $f(0;\lambda) =1$ and $\lim_{x\to\infty}f(x;0)=0$, $\lim_{x\to\infty}f(x;\lambda)=-\infty$
for any $\lambda>0$. Further $f(x;\lambda)$ is monotone decreasing in $\lambda$ with $f(x;\lambda)\to f(x;0)$
as $\lambda\to 0$. Therefore $c_*(\lambda)$ is monotone decreasing in $\lambda$, with $\lim_{\lambda\to 0}c_*(\lambda)= c_0>0$
(with $c_0$ defined as per Eq.~\eqref{eq:c0def}).  Further $\underline{f}(x;\lambda) \le f(x;\lambda)\le \overline{f}(x;\lambda)$ where $\overline{f}(x;\lambda)$
is obtained by replacing $\hH_n$ by $\delta_{1/M}$,  and $\underline{f}(x;\lambda)$
is obtained by replacing $\hH_n$ by $\delta_{M}$. Therefore $C(M)^{-1}\le c_*(\lambda)\le c_0\le C(M)$ for some finite
constant $C(M)$ depending uniquely on $M$. In particular, this implies that $\partial_{\lambda}f(c_*;\lambda)$ is uniformly bounded.
Finally
\begin{align*}
  \partial_xf(x;\lambda) = -\lambda -\int \frac{\gamma s}{(1+x\gamma s)^2} \, d \hH_n(s)\, .
\end{align*}
For $s\in [M^{-1},M]$ and $x\in [C(M)^{-1},C(M)]$, the integrand $\gamma s/(1+x\gamma s)^2$ is bounded and
bounded away from zero. Therefore $C(M)^{-1}\le -\partial_xf(x;\lambda) \le C(M)$ for some (possibly different) constant $C(M)$.
By the implicit function theorem, it follows that $|\partial_{\lambda}c_*(\lambda)|\le C(M)$ is bounded for $\lambda\in[0,M]$, and therefore
$|c_*(\lambda)-c_0|\le C(M)\lambda$, whence
\begin{align}
  m_n(-\lambda) = \big(1-\gamma^{-1}\big)\frac{1}{\lambda} + c_0+O_*(\lambda)\, .
 \end{align} 
 (Here $O_*(\lambda)$ denotes a quantity bounded uniformly as $|O_*(\lambda)|\le C(M)\lambda$.). Sustituting this in  Eq.~\eqref{eq:m1def},
 and repeating similar arguments, we obtain
 \begin{align*}
 m_{n,1}(-\lambda) &= c_1+O_*(\lambda)\, ,\\ 
c_1&:= c_0\frac{\int \frac{s^2}{(1+c_0\gamma s)^2}\, d\hH_n(s)}
   {\int \frac{s}{(1+c_0\gamma s)^2}\, d\hH_n(s)}\, .
 \end{align*}
Finally, substituting in Eq.~\eqref{eq:BiasRidge}, we obtain the claimed bound \eqref{eq:BBdiff2}.

Using Eqs.~\eqref{eq:BBdiff1}, \eqref{eq:BBdiff2} and Theorem \ref{thm:risk_ridge} we deduce that,
for all $M,D,\eps$,  and all $n^{-2/3+1/M}\lambda\le 1$, with probability at least $1-Cn^{-D}$ ,
\begin{align*}
  \big|B_X(\hbeta;\beta) -\cuB(\hH_n,\hG_n,\gamma)  \big| &\le C(M)
                                                                                    \Big(\lambda +\frac{1}{\lambda n^{(1-\eps)/2}}\Big)\, . 
\end{align*}
The desired bound is obtained by setting $\lambda = n^{-1/7}$ and choosing $\eps$ small enough.

\noindent\emph{Variance term.} With the same  notations introduced above,  Eq.~\eqref{eq:VX} implies 
\begin{align}
  V_X(\hbeta;\beta) &= \frac{\sigma^2}{n}\Tr\big(\Sigma U D_X^{-1}1_{D_X>0}U^T\big) \, ,\label{eq:VX0Formula}\\
  V_X(\hbeta_{\lambda};\beta) &=\frac{\sigma^2}{n}\Tr\big(\Sigma U D_X(\lambda I+D_X)^{-2}U^T\big) \, .
\end{align}
Taking the difference and bounding $|x(x+\lambda)^{-2}-x^{-1}|\le 2\lambda/x^2$ for $x,\lambda>0$, we get
\begin{align*}
  \big|  V_X(\hbeta_{\lambda};\beta) -V_X(\hbeta;\beta) \big|&\le  \frac{2\lambda\sigma^2}{n}\Tr\big(\Sigma U D_X^{-2}1_{D_X>0}U^T\big)\\
                                                             &\le \frac{2\lambda M\sigma^2}{\sigma_{\min}(X)^4/n^2}\le C(M) \lambda\, .
\end{align*}
The rest of the proof is analogous to one for the bias term and is omitted.
                                                               
\section{Other proofs for the linear model}

\subsection{Proof of Proposition  \ref{thm:risk_lo}}

Write $X=Z\Sigma^{1/2}$.  Note that
\begin{align}
\lambda_{\min}(X^T X/n) \geq \lambda_{\min}(Z^T Z/n) \lambda_{\min}(\Sigma) 
\geq (c/2) (1-\sqrt\gamma)^2,\label{eq:LowerBD_Eigenvalue}
\end{align}
where the second inequality holds almost surely for all $n$ large enough, following from 
$\lambda_{\min}(\Sigma) \geq c$, and the Bai-Yin theorem
\citep{bai1993limit}, which implies that the smallest eigenvalue of $Z^T Z/n$ is
almost surely larger than \smash{$(1-\sqrt\gamma)^2/2$} for sufficiently large
$n$.  As the right-hand side in the above display is strictly positive, we have
that $X^T X$ is almost surely invertible.  Therefore by the bias and variance
results from Lemma \ref{lem:bias_var}, we have almost surely $\Pi=0$ and 
\smash{$B_X(\hbeta;\beta) = 0$}, and also
\begin{align*}
V_X(\hbeta; \beta) &= \frac{\sigma^2}{n} \Tr (\hSigma^{-1} \Sigma) \\ 
&=  \frac{\sigma^2}{n} \Tr\bigg(\Sigma^{-1/2} \Big(\frac{Z^T Z}{n}\Big)^{-1}
  \Sigma^{-1/2} \Sigma \bigg) \\
&= \frac{\sigma^2 p}{n} \int \frac{1}{s} \, dF_{Z^T Z/n}(s)\\
& = \frac{\sigma^2 p}{n} \int \left(\frac{1}{s}\wedge M\right) \, dF_{Z^T Z/n}(s),
\end{align*}
where \smash{$F_{Z^T Z/n}$} is the spectral measure of $Z^T Z/n$
and the last identity holds if $M^{-1}<(c/2) (1-\sqrt\gamma)^2$  almost surely for all $n\ge N_0$ (with
$N_0$ a random number) by Eq.~\eqref{eq:LowerBD_Eigenvalue}.
By
the Marchenko-Pastur theorem
\citep{marchenko1967distribution,silverstein1995strong}, which says that  
\smash{$F_{Z^T Z/n}$} converges weakly, almost surely, to the Marchenko-Pastur
law $F_\gamma$ (depending only on $\gamma$), whenc, almost surely :
$$
\int_0^\infty \left(\frac{1}{s}\wedge M\right)  \, dF_{Z^T Z/n}(s) \to   
\int_{0}^\infty \left(\frac{1}{s}\wedge M\right)  \, dF_\gamma(s) = \int_{0}^\infty \frac{1}{s} \, dF_\gamma(s)\, ,
$$
where the last equality follows since 
$F_\gamma$ is supported in $[(1-\sqrt{\gamma})^2, (1+\sqrt\gamma)^2]$, provided $M> (1-\sqrt{\gamma})^{-2}$.
Thus the last display implies that as $n,p \to \infty$,
almost surely,
\begin{equation}
\label{eq:var_lo}
V_X(\hbeta; \beta) \to \sigma^2 \gamma \int \frac{1}{s} \, dF_\gamma(s).
\end{equation}
It remains to compute the right-hand side above.  We recognize the right-hand side as the evaluation of the
Stieltjes transform $m(z)$ of Marchenko-Pastur law at $z=0$.  Fortunately, this
has an explicit form (e.g., Lemma 3.11 in \citealt{bai2010spectral}), for real
$z>0$:    
\begin{equation}
\label{eq:stieltjes_mp}
m(-z) = \frac{-(1 - \gamma + z) + \sqrt{(1 - \gamma + z)^2 + 
    4 \gamma z}}{2 \gamma z}.
\end{equation}
The proof is thus completed by taking the limit $z\to 0+$ in this expression.

\subsection{Proof of Theorem \ref{thm:risk_iso}}
\label{sec:ProofRiskIso}

Notice that this result can be obtained as a special case of Theorem \ref{thm:risk_gen_asymp}.
For the reader's convenience, we present here a simpler self-contained proof that assumes
the covariates to satisfy the stronger condition $\E(|z_{ij}|^{8+\delta})\le C<\infty$.

\noindent\emph{Variance term.} Recalling the expression for the bias from Lemma \ref{lem:bias_var} (where now
$\Sigma=I$), we have 
$$
V_X(\hbeta; \beta) = \frac{\sigma^2}{n} \sum_{i=1}^n \frac{1}{s_i},
$$
where $s_i = \lambda_i(X^T X/n)$, $i=1,\ldots,n$ are the nonzero eigenvalues of
$X^T X/n$.  Let $t_i = \lambda_i(XX^T/p)$, $i=1,\ldots,p$ denote the eigenvalues
of $XX^T/p$.  Then we may write $s_i = (p/n) t_i$, $i=1,\ldots,n$, and 
$$
V_X(\hbeta; \beta) = \frac{\sigma^2}{p} \sum_{i=1}^n \frac{1}{t_i} =  
\frac{\sigma^2 n}{p} \int \frac{1}{t} \, dF_{XX^T/p}(t),
$$
where $F_{XX^T/p}$ is the spectral measure of $XX^T/p$.  Now as $n/p \to \tau =
1/\gamma < 1$, we are back precisely in the setting of Theorem
\ref{thm:risk_lo}, and by the same arguments, we may conclude that almost surely  
$$
V_X(\hbeta; \beta) \to \frac{\sigma^2 \tau}{1-\tau} = \frac{\sigma^2}{\gamma-1}, 
$$
completing the proof.

\noindent\emph{Bias term.}
Recall the expression for the bias from Lemma \ref{lem:bias_var} (where now 
$\Sigma=I$), and note the following key characterization of the pseudoinverse of a
rectangular matrix $A$,
\begin{equation}
\label{eq:pinv}
(A^T A)^+ A^T = \lim_{z \to 0^+} \, (A^T A + zI)^{-1} A^T.
\end{equation}
We can apply this to \smash{$A=X/\sqrt{n}$}, and rewrite the bias as
\begin{align}
\nonumber
B_X(\hbeta; \beta) &= \lim_{z \to 0^+} \, \beta^T 
\big( I-(\hSigma + zI)^{-1} \hSigma \big) \beta \\ 
\label{eq:bias_iso_lim}
&= \lim_{z \to 0^+} \, z \beta^T (\hSigma + zI)^{-1} \beta, 
\end{align}
where in the second line we added and subtracted $zI$ to \smash{$\hSigma$} and
simplified.  By Theorem 1 in \citet{rubio2011spectral}, which may be seen as a  
generalized Marchenko-Pastur theorem, we have that for any $z>0$, and any
deterministic sequence of matrices $\Theta_n \in \R^{p \times p}$,
$n=1,2,3,\ldots$ with uniformly bounded trace norm, it holds as $n,p \to 
\infty$, almost surely,        
\begin{equation}
\label{eq:gen_mp}
\Tr\Big( \Theta_n \big((\hSigma + zI)^{-1} - c_n(z) I \big)\Big) \to 0, 
\end{equation}
for a deterministic sequence $c_n(z)>0$, $n=1,2,3,\ldots$ (defined for each
$n$ via a certain fixed-point equation).  Taking $\Theta_n=I/p$ in
the above, note that this reduces to the almost sure convergence of the  
Stieltjes transform of the specral distribution of \smash{$\hSigma$}, and
hence by the (classical) Marchenko-Pastur theorem, we learn that $c_n(z)
\to m(-z)$, where $m$ denotes the Stieltjes transform of the Marchenko-Pastur
law $F_\gamma$.  Further, taking $\Theta_n=\beta\beta^T/p$, we see from
\eqref{eq:gen_mp} and $c_n(z) \to m(-z)$ that, almost surely, 
\begin{equation}
\label{eq:bias_iso_z}
z \beta^T (\hSigma + zI)^{-1} \beta \to z m(-z) r^2.
\end{equation}
Define \smash{$f_n(z)=z\beta^T (\hSigma + zI)^{-1} \beta$}.  Notice that $|f_n(z)|
\leq r^2$, and \smash{$f_n'(z) = \beta^T (\hSigma + zI)^{-2} \hSigma \beta$}, so   
$$
|f_n'(z)| \leq r^2 
\frac{\lambda_{\max}(\hSigma)}{(\lambda_{\min}^+(\hSigma) + z)^2}   
\leq 8r^2 \frac{(1+\sqrt\gamma)^2}{(1-\sqrt\gamma)^4},
$$
where \smash{$\lambda_{\max}(\hSigma)$} and \smash{$\lambda_{\min}^+(\hSigma)$}    
denote the largest and smallest nonzero eigenvalues, respectively, of
\smash{$\hSigma$}, and the second inequality holds almost surely for large
enough $n$, by the Bai-Yin theorem \citep{bai1993limit}.  As its derivatives are
bounded, the sequence $f_n$, $n=1,2,3,\ldots$ is equicontinuous, and
by the Arzela-Ascoli theorem, we deduce that $f_n$ converges uniformly to its
limit. By the Moore-Osgood theorem, we can exchange limits (as $n,p \to \infty$
and $z \to 0^+$) and conclude from \eqref{eq:bias_iso_lim},
\eqref{eq:bias_iso_z} that as $n,p \to \infty$, almost surely,  
$$
B_X(\hbeta; \beta) \to r^2 \lim_{z \to 0^+} \, z m(-z).
$$
Finally, relying on the fact that the Stieltjes transform of the
Marchenko-Pastur law has the explicit form in \eqref{eq:stieltjes_mp}, 
we can compute the above limit: 
\begin{align*}
\lim_{z \to 0^+} \, z m(-z) 
&= \lim_{z \to 0^+} \, \frac{-(1 - \gamma + z) + 
\sqrt{(1 - \gamma + z)^2+ 4 \gamma z}}{2 \gamma} \\
&= \frac{-(1-\gamma) + (\gamma-1)}{2 \gamma} 
= 1-1/\gamma,
\end{align*}
completing the proof.

\subsection{Proof of Corollary \ref{coro:equidistributed}}

The corollary follows immediately from  Theorem \ref{thm:risk_gen_asymp}, once we
estabilish that, for $\|\beta\|^2\to r^2$,
$\cuB(H,H,\gamma) \to \cuB_{\sequi}(H,\gamma)$. By homogeneity, we can assume without loss of generality
that $\|\beta\|^2 = r^2=1$. By using  Eq.~\eqref{eq:BiasPredMinNorm}, we get
\begin{align}
  \cuB(H,H,\gamma) &=  \int\frac{s}{(1+c_0\gamma s)^2}\, d H(s)+\gamma  c_0\int \frac{s^2}{(1+c_0\gamma s)^2}\, d H(s)\\
                   & = \int\frac{s}{1+c_0\gamma s}\, d H(s)\\
  & = \frac{1}{c_0\gamma}\Big[1- \int\frac{1}{1+c_0\gamma s}\, d H(s)\Big] = \frac{1}{c_0\gamma^2}\, ,
  \end{align}
  where that last equality follows from the definition of $c_0$ in Eq.~\eqref{eq:c0def}.
 Comparing with the definition of $cuB_{\sequi}(H,\gamma)$ in Eq.~\eqref{eq:Bequi} yields the desired claim.

\subsection{The case of equicorrelated features}
\label{app:risk_ec}

As an illustration of Corollary \ref{coro:equidistributed}, we consider the case of equicorrelated covariates.
Namellt, for  $\rho \in
[0,1)$, we assume that $\Sigma_{ii}=1$ for all $i$, and $\Sigma_{ij}=\rho$ for all
$i \not= j$.  

\begin{corollary}
\label{cor:risk_ec}
Assume the conditions of Theorem \ref{thm:risk_gen}, and moreover, assume that 
$\Sigma$ has $\rho$-equicorrelation structure for all $n,p$, and some $\rho \in
[0,1)$.  Then as $n,p \to \infty$, with $p/n \to \gamma > 1$, we have almost
surely  
$$
R_X(\hbeta) \to r^2(1-\rho) (1-1/\gamma) + \frac{\sigma^2}{\gamma-1}.    
$$
\end{corollary}
\begin{proof}
Let $H_\rho$ denote the weak limit of $F_\Sigma$, when $\Sigma$ has
$\rho$-equicorrelation structure for all $n,p$.  A short calculation shows that 
such a matrix $\Sigma$ has one eigenvalue value equal to $1+(p-1)\rho$, and 
$p-1$ eigenvalues equal to $1-\rho$.  Thus the weak limit of its spectral
measure is simply $H_\rho = 1_{[1-\rho,\infty)}$, i.e.,
$dH_\rho=\delta_{1-\rho}$, a point mass at $1-\rho$ of probability one. 

We remark that the present case, strictly speaking, breaks the conditions
that we assume in Theorem \ref{thm:risk_gen}, because
$\lambda_{\max}(\Sigma)=1+(p-1)\rho$ clearly diverges with $p$.
However, by decomposing $\Sigma=(1-\rho) I + \rho\one\one^T$ (where $\one$  
denotes the vector of all 1s), and correspondingly decomposing the functions 
$f_n,g_n,h_n$ defined in the proof of Theorem \ref{thm:risk_gen}, to handle the
rank one part $\rho\one\one^T$ properly, we can ensure the appropriate
boundedness conditions. Thus, the result in the theorem still holds when
$\Sigma$ has $\rho$-equicorrelation structure.  

Now denote by $v_\rho$ the companion Stieltjes transform of the empirical
spectral distribution \smash{$F_{H_\rho,\gamma}$}, to emphasize its dependence 
on $\rho$. Recall the Silverstein equation \eqref{eq:silverstein}, which for the
equicorrelation case, as $dH_\rho=\delta_{1-\rho}$, becomes
$$
-\frac{1}{v_\rho(z)} = z - \gamma \frac{1-\rho}{1 + (1-\rho) v_\rho(z)},
$$
or equivalently,
$$
-\frac{1}{(1-\rho) v_\rho(z)} = \frac{z}{1-\rho} - \gamma \frac{1}{1 + (1-\rho)
  v_\rho(z)}.
$$
We can hence recognize the relationship 
$$
(1-\rho) v_\rho(z) = v_0\big(z/(1-\rho)\big),
$$
where $v_0$ is the companion Stieltjes transform in the $\Sigma=I$ case, 
the object of study in Appendix \ref{app:risk_iso2}.  From the results for $v_0$
from \eqref{eq:v0} and \eqref{eq:v1}, invoking the relationship in the above
display, we have 
$$
\lim_{z \to 0^+} \, v_\rho(-z) = \frac{1}{(1-\rho)(\gamma-1)}  
\quad \text{and} \quad 
\lim_{z \to 0^+} \, v'_\rho(-z) = \frac{\gamma}{(1-\rho)^2(\gamma-1)^3}.  
$$
Plugging these into the asymptotic risk expression from Theorem
\ref{thm:risk_gen} gives
$$
\frac{r^2}{\gamma} (1-\rho)(\gamma-1) + \sigma^2\bigg( 
\frac{\gamma (1-\rho)^2 (\gamma-1)^2}{(1-\rho)^2(\gamma-1)^3} - 1 \bigg) 
= r^2(1-\rho) \Big(1-\frac{1}{\gamma}\Big) + \frac{\sigma^2}{\gamma-1}.  
$$
as claimed.
\end{proof}

\subsection{Autoregressive features} 
\label{app:risk_ar}

We consider a {\it $\rho$-autoregressive} structure for $\Sigma$, for a constant
$\rho \in [0,1)$, meaning that $\Sigma_{ij}=\rho^{|i-j|}$ for all $i,j$.  In
this case, it is not clear that a closed-form exists for $v(0)$ or $v'(0)$.  
However, we can compute these numerically.  In fact, the strategy we describe 
below applies to any situation in which we are able to perform numerical
integration against $dH$, where $H$ is the weak limit of the spectral measure
$F_\Sigma$ of $\Sigma$.

The critical relationship that we use
is the {\it Silverstein equation} \citep{silverstein1995strong}, which relates
the companion Stieltjes transform $v$ to $H$ via
\begin{equation}
\label{eq:silverstein}
-\frac{1}{v(z)} = z - \gamma \int \frac{s}{1 + sv(z)} \, dH(s).
\end{equation}
Taking $z \to 0^+$ yields 
\begin{equation}
\label{eq:v0_fixed_point}
\frac{1}{v(0)} = \gamma \int \frac{s}{1 + sv(0)} \, dH(s). 
\end{equation}
Therefore, we can use a simple univariate root-finding algorithm (like the
bisection method) to solve for $v(0)$ in \eqref{eq:v0_fixed_point}.  With $v(0)$  
computed, we can compute $v'(0)$ by first differentiating \eqref{eq:silverstein}
with respect to $z$ (see \citealt{dobriban2015efficient}), and then taking $z
\to 0^+$, to yield   
\begin{equation}
\label{eq:v0_deriv}
\frac{1}{v'(0)} = \frac{1}{v(0)^2} - \gamma \int \frac{s^2}{(1 + sv(0))^2} \,
dH(s).
\end{equation}

When $\Sigma$ is of $\rho$-autoregressive form, it is known to have eigenvalues 
\citep{trench1999asymptotic}:
$$
s_i = \frac{1-\rho^2}{1-2\rho\cos(\theta_i)+\rho^2}, \quad i=1,\ldots,p,  
$$
where
$$
\frac{(p-i)\pi}{p+1} < \theta_i < \frac{(p-i+1)\pi}{p+1}, \quad i=1,\ldots,p. 
$$
This allows us to efficiently approximate an integral with respect to $dH$
(e.g., by taking each $\theta_i$ to be in the midpoint of its interval given
above), solve for $v(0)$ in \eqref{eq:v0_fixed_point}, $v'(0)$ in
\eqref{eq:v0_deriv}, and evaluate the asymptotic risk as per Theorem
\ref{thm:risk_gen}.  

Figure \ref{fig:risk_ar} shows the results from using such a numerical scheme to
evaluate the asymptotic risk, as $\rho$ varies from 0 to 0.75.

\begin{figure}[p]
\vspace{-40pt}
\centering
\includegraphics[width=0.725\textwidth]{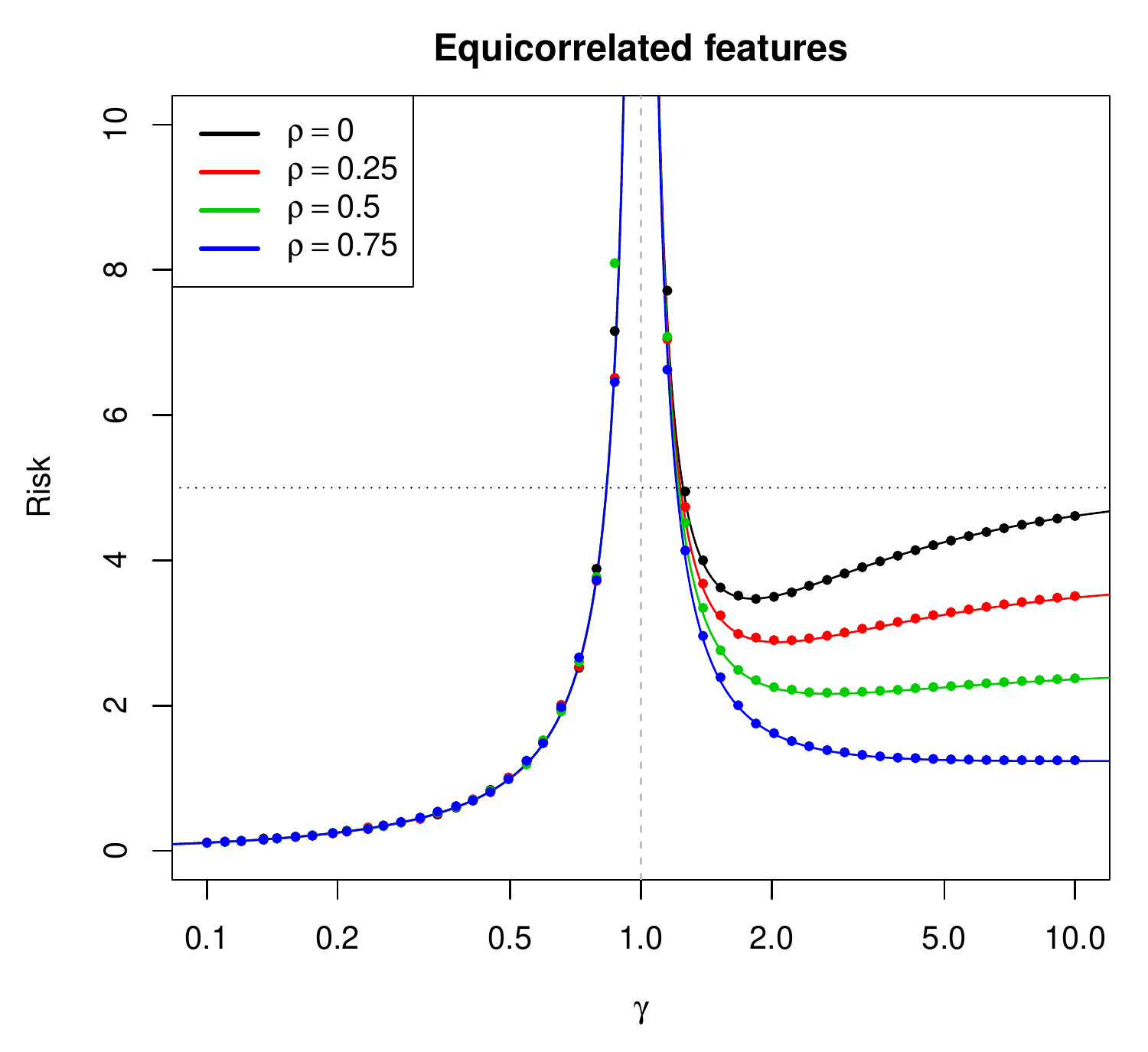} 
\caption{\footnotesize Asymptotic risk curves for the min-norm least squares
  estimator when $\Sigma$ has equicorrelation structure (Theorem
  \ref{thm:risk_lo} for $\gamma<1$, and Corollary \ref{cor:risk_ec} for
  $\gamma>1$), as $\rho$ varies from 0 to 0.75. Here $r^2=5$ and $\sigma^2=1$,
  thus $\snr=5$. The null risk $r^2=5$ is marked as a dotted black line. The
  points denote finite-sample risks, with $n=200$, $p=[\gamma n]$, across
  various values of $\gamma$, computed from appropriately constructed Gaussian
  features.}      
\label{fig:risk_ec} 

\bigskip
\includegraphics[width=0.725\textwidth]{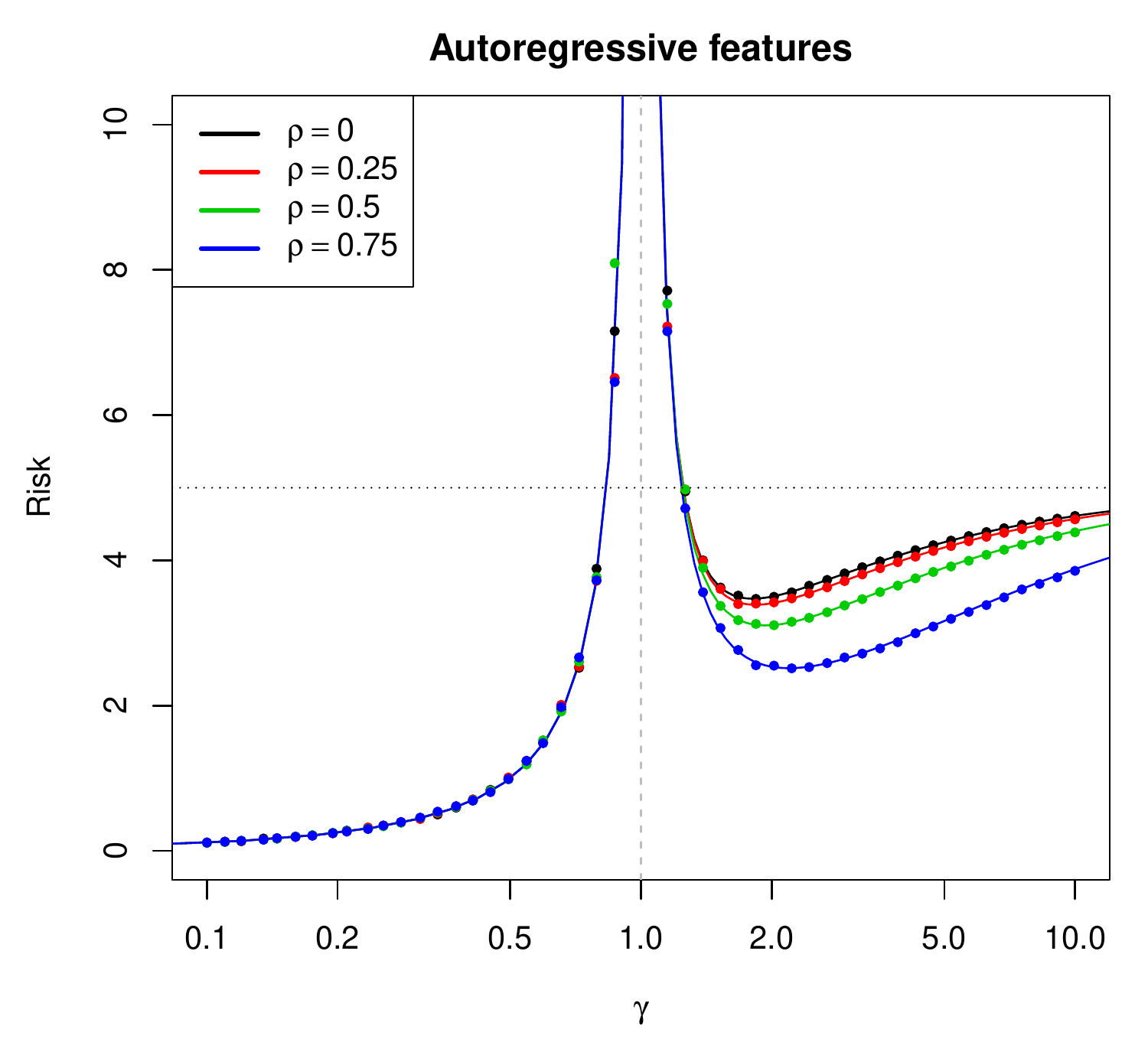}
\caption{\footnotesize Asymptotic risk curves for the min-norm least squares
  estimator when $\Sigma$ has autoregressive structure (Theorem
  \ref{thm:risk_lo} for $\gamma<1$, and Theorem \ref{thm:risk_gen} for
  $\gamma>1$, evaluated numerically, as described in Appendix
  \ref{app:risk_ar}), as $\rho$ varies from 0 to 0.75. Here $r^2=5$ and
  $\sigma^2=1$, thus $\snr=5$. The null risk $r^2=5$ is marked as a dotted black
  line. The points are again finite-sample risks, with $n=200$, $p=[\gamma n]$,
  across various values of $\gamma$, computed from appropriately constructed
  Gaussian features.}      
\label{fig:risk_ar}
\end{figure}

\begin{figure}[htb]
\centering
\includegraphics[width=0.475\textwidth]{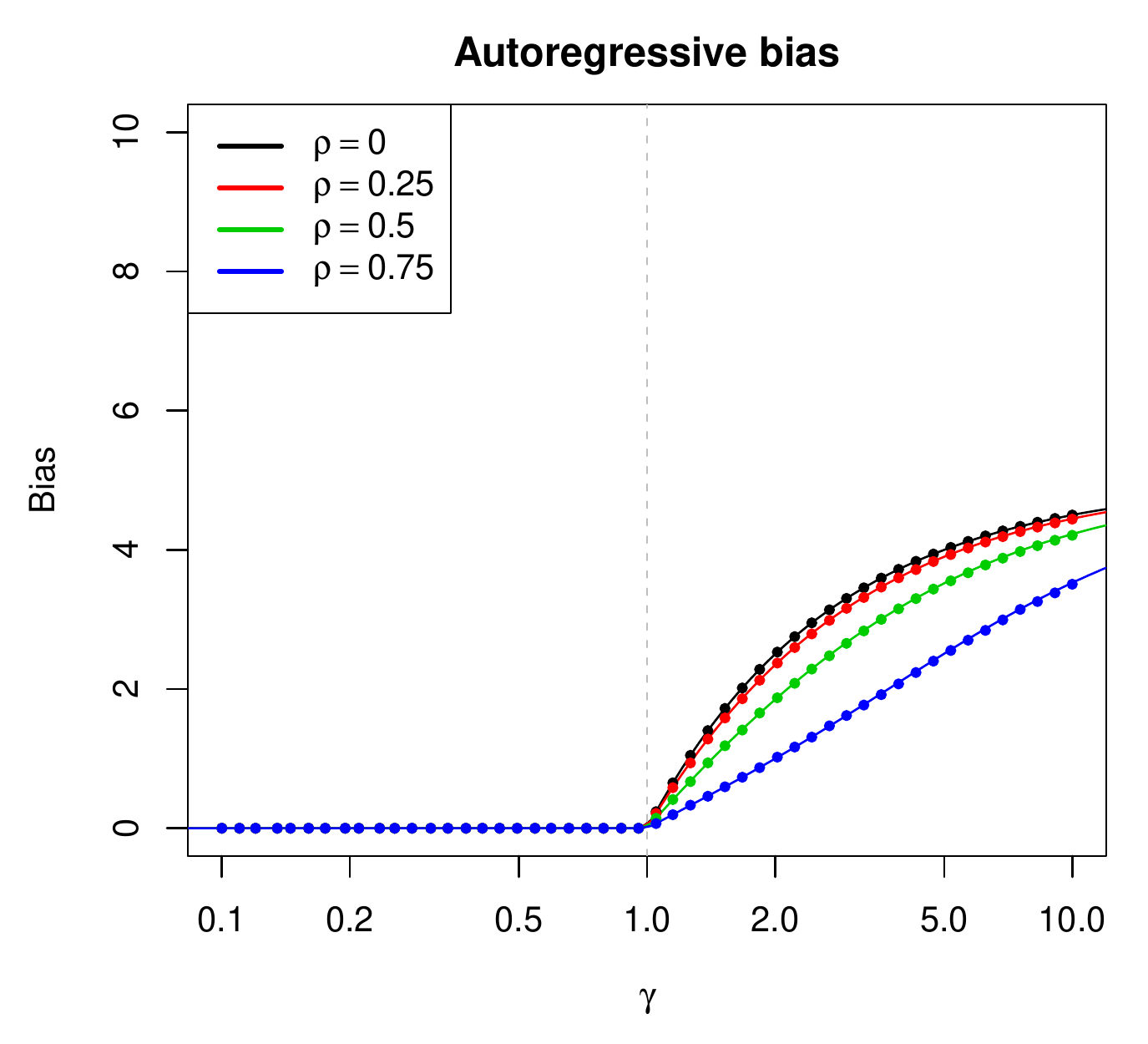} 
\includegraphics[width=0.475\textwidth]{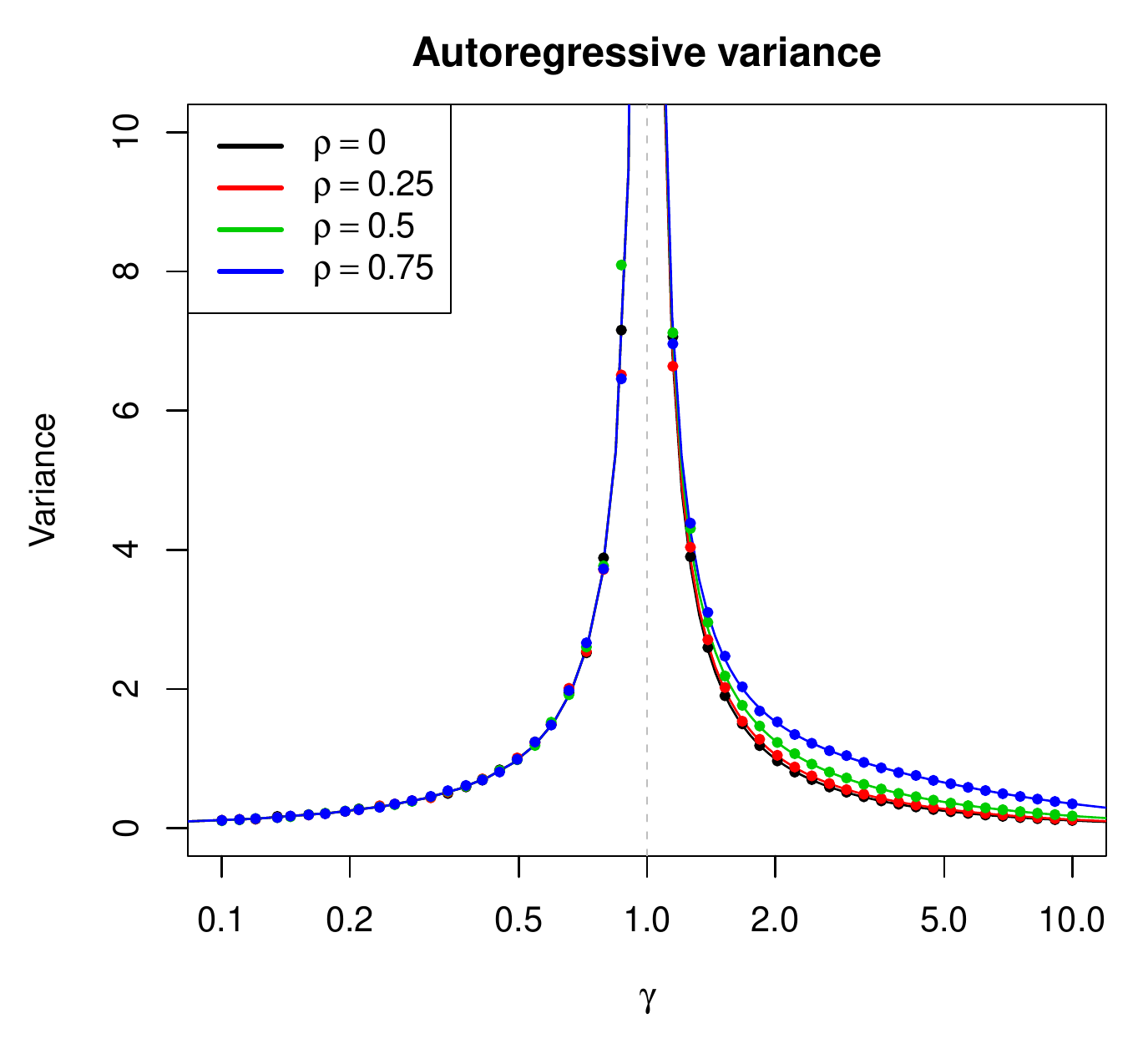} 
\caption{\footnotesize Asymptotic bias (left panel) and variance (right panel)
  for the min-norm least squares estimator when $\Sigma$ has autoregressive  
  structure (Theorem \ref{thm:risk_lo} for $\gamma<1$, and Theorem
  \ref{thm:risk_gen} for $\gamma>1$, evaluated numerically, as described in
  Appendix \ref{app:risk_ar}), as $\rho$ varies from 0 to 0.75. Here $r^2=5$ and  
  $\sigma^2=1$, hence $\snr=5$. The points mark finite-sample biases and 
  variances, with $n=200$, $p=[\gamma n]$, computed from appropriately
  constructed Gaussian features.}   
\label{fig:bias_var_ar} 
\end{figure}

\subsection{Proof of Theorem \ref{thm:risk_cv}}
\label{app:risk_cv}

We begin by recalling an alternative to leave-one-out cross-validation, for
linear smoothers, called {\it generalized cross-validation} (GCV)
\citep{craven1978smoothing,golub1979generalized}.  The GCV error of the ridge  
regression estimator at a tuning parameter value $\lambda$ defined as 
\begin{equation}
\label{eq:gcv}
\gcv_n(\lambda) = 
\frac{1}{n} \sum_{i=1}^n \bigg( \frac{y_i - \hf_\lambda(x_i)}  
{1-\Tr(S_\lambda)/n} \bigg)^2.
\end{equation}
Compared to the shortcut formula for leave-one-out CV in \eqref{eq:cv_shortcut},
we can see that GCV in \eqref{eq:gcv} swaps out the $i$th diagonal element
$(S_\lambda)_{ii}$ in the denominator of each summand with the average diagonal
element $\Tr(S_\lambda)/n$.  This modification makes GCV rotationally invariant
\citep{golub1979generalized}.

It turns out that the GCV error is easier to analyze, compared to the CV error.
Thus we proceed by first studying GCV, and then relating CV to GCV.  We break
up the exposition below into these two parts accordingly.

\subsubsection{Analysis of GCV}  

Let us rewrite the GCV criterion in \eqref{eq:gcv} as 
\begin{equation}
\label{eq:gcv_frac}
\gcv_n(\lambda) = \frac{y^T (I-S_\lambda)^2 y/n}{(1-\Tr(S_\lambda)/n)^2}.  
\end{equation}
We will treat the almost sure convergence of the numerator and denominator
separately.  

\paragraph{GCV denominator.} The denominator is an easier calculation.  Denoting
$s_i=\lambda_i(X^T X/n)$, $i=1,\ldots,p$, we have 
$$
\Tr(S_\lambda)/n = \frac{1}{n} \sum_{i=1}^p \frac{s_i}{s_i+\lambda} 
\to \gamma \int \frac{s}{s+\lambda} \, dF_\gamma(s),
$$
where this convergence holds almost surely as $n,p \to \infty$, a direct
consequence of the Marchenko-Pastur theorem, and $F_\gamma$ denotes the
Marchenko-Pastur law. Meanwhile, we can rewrite this asymptotic limit as  
$$
\gamma \int \frac{s}{s+\lambda} \, dF_\gamma(s) = 
\gamma\big(1 - m(-\lambda)\big), 
$$
where \smash{$m=m_{F_\gamma}$} denotes the Stieltjes transform of the
Marchenko-Pastur law $F_\gamma$, and therefore, almost surely, 
\begin{equation}
\label{eq:gcv_denom}
\big(1-\Tr(S_\lambda)/n\big)^2 \to 
\Big(1 - \gamma\big(1 - m(-\lambda)\big) \Big)^2. 
\end{equation}

\paragraph{GCV numerator.} The numerator requires only a bit more difficult
calculation.  Let $y=X\beta+\epsilon$ and \smash{$c_n=\sqrt{p} (\sigma/r)$}.
Observe   
\begin{align*}
y^T (I-S_\lambda)^2 y / n
&= (\beta, \epsilon)^T \bigg( \frac{1}{n} 
\begin{bmatrix} X \\ I \end{bmatrix}^T 
(I-S_\lambda)^2 
\begin{bmatrix} X \\ I \end{bmatrix} \bigg) 
(\beta, \epsilon) \\
&= \underbrace{
\vphantom{\begin{bmatrix} X \\ c_n I \end{bmatrix}} 
(\beta, \epsilon/c_n)^T}_{\delta^T}
\underbrace{ \bigg( \frac{1}{n}
\begin{bmatrix} X \\ c_n I \end{bmatrix}^T 
(I-S_\lambda)^2 
\begin{bmatrix} X \\ c_n I \end{bmatrix} \bigg)}_{A} 
\underbrace{ 
\vphantom{\begin{bmatrix} X \\ c_n I \end{bmatrix}}
(\beta, \epsilon/c_n)}_{\delta}.
\end{align*}
Note that $\delta$ has independent entries with mean zero and variance $r^2/p$,
and further, note that $\delta$ and $A$ are independent.  Therefore we can
use the almost sure convergence of quadratic forms, from Lemma 7.6 in
\citet{dobriban2018high}, which is adapted from Lemma B.26 in
\citet{bai2010spectral}.\footnote{As written, Lemma 7.6 of
  \citet{dobriban2018high} assumes i.i.d.\ components for the random vector in
  question, which is not necessarily true of $\delta$ in our case.  However, an
  inspection of their proof shows that they only require independent components
  with mean zero and common variance, which is precisely as stated in Lemma B.26
  of \citet{bai2010spectral}.}  This result asserts that, almost surely, 
$$
\delta^T A \delta - (r^2/p) \Tr(A) \to 0.
$$
Now examine
\begin{align*}
r^2 \Tr(A) / p &= \frac{r^2}{p} 
\Tr \big( (I-S_\lambda)^2 (XX^T/n + (c_n^2/n) I\big) \\
&= \underbrace{\frac{r^2}{p} \Tr \big( X^T (I-S_\lambda)^2 X/n \big)}_{a} +  
\underbrace{\frac{\sigma^2}{n} \Tr\big( (I-S_\lambda)^2 \big)}_{b}.  
\end{align*}
A short calculation and application of the Marchenko-Pastur theorem gives that,
almost surely,   
$$
a = \frac{r^2 \lambda^2}{p} \bigg(\sum_{i=1}^p \frac{1}{s_i+\lambda} -  
\lambda \sum_{i=1}^p \frac{1}{(s_i+\lambda)^2} \bigg) \to
r^2 \lambda^2 \big(m(-\lambda) - \lambda m'(-\lambda)\big),
$$
where for the second sum, we used Vitali's theorem to show convergence of the
derivative of the Stieltjes transform of the spectral distribution of $X^T X/n$
to the derivative of the Stieltjes transform of $F_\gamma$ (note that Vitali's
theorem applies as the function in question is bounded and analytic). By a
similar calculation, we have almost surely,    
$$
b = \frac{\sigma^2}{n} \bigg(\sum_{i=1}^p \frac{\lambda^2}{(s_i+\lambda)^2} +
(n-p) \bigg) \to \sigma^2 \gamma \lambda^2 m'(-\lambda) + \sigma^2 (1-\gamma).         
$$
Hence we have shown that, almost surely, 
\begin{equation}
\label{eq:gcv_numer}
y^T (I-S_\lambda)^2 y / n \to \lambda^2 \Big( r^2 
\big(m(-\lambda) - \lambda m'(-\lambda)\big) - 
\sigma^2 \gamma m'(-\lambda) \Big) + \sigma^2(1-\gamma). 
\end{equation}

\paragraph{GCV convergence.} Putting \eqref{eq:gcv_denom}, \eqref{eq:gcv_numer}
together with \eqref{eq:gcv_frac}, we have, almost surely, 
$$
\gcv_n(\lambda) \to \frac{\lambda^2 (r^2 
(m(-\lambda) - \lambda m'(-\lambda)) - 
\sigma^2 \gamma m'(-\lambda)) + \sigma^2(1-\gamma)} 
{(1 - \gamma(1 - m(-\lambda)))^2}.
$$
To show that this matches to the asymptotic prediction error of ridge
regression requires some nontrivial calculations.  We start by
reparametrizing the above asymptotic limit in terms of the companion Stieltjes
transform, abbreviated by \smash{$v = v_{F_\gamma}$}.  This satisfies   
$$
v(z) + 1/z = \gamma \big(m(z) + 1/z\big), 
$$
hence
$$
z v(-z) - 1= \gamma\big( z m(-z) - 1\big),
$$
and also 
$$
z^2 v'(-z) - 1 = \gamma\big( z^2 m'(-z) - 1 \big).
$$
Introducing $\alpha=r^2/(\sigma^2 \gamma)$, the almost sure limit of GCV is 
\begin{align*}
&\frac{\lambda^2 (r^2 (m(-\lambda) - \lambda m'(-\lambda)) -  
\sigma^2 \gamma m'(-\lambda)) + \sigma^2(1-\gamma)} 
{(1 - \gamma(1 - m(-\lambda)))^2} \\
&\qquad\quad= \frac{\sigma^2 \lambda (\alpha \gamma(\lambda m(-\lambda)-1) -
  \alpha \gamma (\lambda^2 m'(-\lambda) - 1) + (\gamma/\lambda) (\lambda^2
  m'(-\lambda) - 1) + \gamma/\lambda + (1-\gamma)/\lambda)}{(1 - \gamma(1 -
  m(-\lambda)))^2} \\    
&\qquad\quad= \frac{\sigma^2 \lambda (\alpha (\lambda v(-\lambda) - 1) - 
  \alpha (\lambda^2 v'(-\lambda) - 1) + (1/\lambda)(\lambda^2 v'(-\lambda) - 1)
  + 1/\lambda)}{\lambda^2 v(-\lambda)^2} \\ 
&\qquad\quad= \frac{\sigma^2 (v(-\lambda) + (\alpha \lambda - 1)(v(-\lambda) - 
  \lambda v'(-\lambda))} {\lambda v(-\lambda)^2},
\end{align*}
where in the second line we rearranged, in the third line we applied the
companion Stieltjes transform facts, and in the fourth line we simplified. 
We can now recognize the above as $\sigma^2$ plus the asymptotic risk of ridge
regression at tuning parameter $\lambda$, either from the proof of Theorem
\ref{thm:risk_gen}, or from Theorem 2.1 in \citet{dobriban2018high}.  In terms
of the Stieltjes transform itself, this is \smash{$\sigma^2 + \sigma^2 \gamma
  (m(-\lambda) - \lambda(1 - \alpha \lambda) m'(-\lambda))$}, which proves the
first claimed result.

\paragraph{Uniform convergence.} It remains to prove the second claimed result,
on the convergence of the GCV-tuned ridge estimator. Denote
$f_n(\lambda)=\gcv_n(\lambda)$, and $f(\lambda)$ for its almost sure limit.
Notice that $|f_n|$ is almost surely bounded on $[\lambda_1,\lambda_2]$, for
large enough $n$, as  
\begin{align*}
|f_n(\lambda)| &\leq \frac{\|y\|_2^2}{n} \frac{\lambda_{\max}(I-S_\lambda)^2}   
{(1-\Tr(S_\lambda)/n)^2} \\
&\leq \frac{\|y\|_2^2}{n} \frac{(s_{\max}+\lambda)^2}{\lambda^2} \\ 
& \leq 2(r^2 + \sigma^2) \frac{(2+\lambda_2)^2}{\lambda_1^2}.
\end{align*}
In the second line, we used \smash{$\Tr(S_\lambda)/n = (1/n) \sum_{i=1}^p
  s_i/(s_i+\lambda) \leq s_{\max} / (s_{\max}+\lambda)$}, with
$s_{\max}=\lambda_{\max}(X^T X/n)$, and in the third line, we used $\|y\|_2^2/n 
\leq 2(r^2+\sigma^2)$ almost surely for sufficiently large $n$, by the strong 
law of large numbers, and $s_{\max} \leq 2$ almost surely for sufficiently large
$n$, by the Bai-Yin theorem \citep{bai1993limit}.  Furthermore, writing
$g_n,h_n$ for the numerator and denominator of $f_n$, respectively, we have
$$
f_n'(\lambda) = \frac{g_n'(\lambda) h_n(\lambda) - g_n(\lambda)
  h_n'(\lambda)}{h_n(\lambda)^2}.
$$
The above argument just showed that $|g_n(\lambda)|$ is upper bounded on  
$[\lambda_1,\lambda_2]$, and $|h_n(\lambda)|$ is lower bounded on
$[\lambda_1,\lambda_2]$; also, clearly $|h_n(\lambda)| \leq 1$; therefore to
show that $|f_n'|$ is almost surely bounded on $[\lambda_1,\lambda_2]$, it
suffices to show that both $|g_n'|,|h_n'|$ are. Denoting by $u_i$,
$i=1,\ldots,p$ the eigenvectors of $X^T X / n$ (corresponding to eigenvalues
$s_i$, $i=1,\ldots,p$), a short calculation shows
$$
|g_n'(\lambda)| = \frac{2\lambda}{n} \sum_{i=1}^p (u_i^T y)^2 
\frac{s_i}{(s_i+\lambda)^3} \leq \frac{2 \lambda}{n} \|y\|_2^2 
\leq 4 \lambda (r^2 + \sigma^2),
$$
the last step holding almost surely for large enough $n$, by the law of large 
numbers.  Also, 
$$
|h_n'(\lambda)| = 2\bigg(1 - \frac{1}{n} \sum_{i=1}^n \frac{s_i}{s_i+\lambda}
\bigg) \frac{1}{n} \sum_{i=1}^n \frac{s_i}{(s_i+\lambda)^2} \leq
\frac{4}{\lambda_1^2},
$$
the last step holding almost surely for large enough $n$, by the Bai-Yin
theorem.  Thus we have shown that $|f_n'|$ is almost surely bounded on 
$[\lambda_1,\lambda_2]$, for large enough $n$, and applying the Arzela-Ascoli
theorem, $f_n$ converges uniformly to $f$.  With \smash{$\lambda_n =
  \arg\min_{\lambda \in [\lambda_1,\lambda_2]} f_n(\lambda)$}, this means for
any for any $\lambda \in [\lambda_1,\lambda_2]$, almost surely 
\begin{align*}
f(\lambda_n) - f(\lambda) &=
\big(f(\lambda_n)-f_n(\lambda_n)\big) + 
\big(f_n(\lambda_n) - f_n(\lambda)\big) + 
\big(f_n(\lambda) - f(\lambda)\big) \\
&\leq \big(f(\lambda_n)-f_n(\lambda_n)\big) + 
\big(f_n(\lambda) - f(\lambda)\big) \to 0,
\end{align*}
where we used the optimality of $\lambda_n$ for $f_n$, and then uniform
convergence.  In other words, almost surely,  
$$
f(\lambda_n) \to f(\lambda^*) = \sigma^2 + \sigma^2 \gamma m(-1/\alpha). 
$$ 
As the almost sure convergence \smash{$R_X(\hbeta_\lambda) +
  \sigma^2 \to f(\lambda)$} is also uniform for $\lambda \in
[\lambda_1,\lambda_2]$ (by similar arguments, where we bound the risk and its 
derivative in $\lambda$), we conclude that almost surely
\smash{$R_X(\hbeta_\lambda) \to  \sigma^2 \gamma m(-1/\alpha)$}, completing the
proof for GCV.

\subsubsection{Analysis of CV}  

Let us rewrite the CV criterion, starting in its shortcut form
\eqref{eq:cv_shortcut}, as 
\begin{equation}
\label{eq:cv_quad}
\cv_n(\lambda) = y^T (I-S_\lambda) D^{-2}_\lambda (I-S_\lambda) y / n,
\end{equation}
where $D_\lambda$ is a diagonal matrix that has diagonal elements
\smash{$(D_\lambda)_{ii} = 1-(S_\lambda)_{ii}$}, $i=1,\ldots,n$.  

\paragraph{CV denominators.} First, fixing an arbitrary $i=1,\ldots,n$, we will
study the limiting behavior of 
$$
1 - (S_{\lambda})_{ii} = 1 - x_i^T (X^T X/n + \lambda I)^{-1} x_i /n. 
$$
Since $x_i$ and $X^T X$ are not independent, we cannot immediately apply the 
almost sure convergence of quadratic forms lemma, as we did in the previous
analysis of GCV.  But, letting $X_{-i}$ denote the matrix $X$ with the $i$th row 
removed, we can write \smash{$(X^T X/n + \lambda I)^{-1} = (X_{-i}^T 
  X_{-i}/n + \lambda I + x_ix_i^T/n)^{-1}$}, and use the
Sherman-Morrison-Woodbury formula to separate out the dependent and
independent parts, as follows.  Letting \smash{$\delta_i = x_i/\sqrt{n}$},
\smash{$A_i= (X_{-i}^T X_{-i}/n + \lambda I)^{-1}$}, and $A = (X^T X/n+\lambda
I)$, we have 
\begin{align*}
1 - (S_{\lambda})_{ii} &= 1 - \delta_i^T A \delta_i \\ 
&= 1 - \delta_i^T \bigg( A_i - \frac{A_i \delta_i \delta_i^T A_i}{1 + \delta_i^T
  A_i \delta} \bigg) \delta_i \\ 
&= \frac{1}{1 + \delta_i^T A_i \delta_i}.
\end{align*}
Note $\delta_i$ and $A_i$ are independent (i.e., $x_i$ and \smash{$X_{-i}^T
  X_{-i}$} are independent), so we can now use the almost sure convergence of 
quadratic forms, from Lemma 7.6 in \citet{dobriban2018high}, adapted
from Lemma B.26 in \citet{bai2010spectral}, to get that, almost surely
\begin{equation}
\label{eq:cv_denom_lim}
\delta_i^T A_i \delta_i - \Tr(A_i) / n \to 0.
\end{equation}
Further, as $\Tr(A_i) /n \to \gamma m(-\lambda)$ almost surely by the
Marchenko-Pastur theorem, we have, almost surely,   
\begin{equation}
\label{eq:cv_denom_lim2}
1 - (S_{\lambda})_{ii} \to \frac{1}{1 + \gamma m(-\lambda)}.  
\end{equation}

\paragraph{Replacing denominators, controlling remainders.} The strategy
henceforth, based on the result in \eqref{eq:cv_denom_lim2}, is to replace the 
denominators \smash{$1 -  (S_{\lambda})_{ii}$}, $i=1,\ldots,n$ in the summands
of the CV error by their asymptotic limits, and then control the remainder
terms.  More precisely, we define \smash{$\bar{D}_\lambda = (1+\gamma
  m(-\lambda))^{-1} I$}, and then write, from \eqref{eq:cv_quad},      
\begin{equation}
\label{eq:cv_quad_diff}
\cv_n(\lambda) = \underbrace{y^T (I-S_\lambda) \bar{D}^{-2}_\lambda
  (I-S_\lambda) y/n}_{a} + \underbrace{y^T (I-S_\lambda) (D^{-2}_\lambda -
  \bar{D}^{-2}_\lambda) (I-S_\lambda) y/n}_{b}.
\end{equation}
We will first show that almost surely $b \to 0$.  Observe, by the
Cauchy-Schwartz inequality,  
\begin{align*}
b &\leq \frac{1}{n} \|(I-S_\lambda) y \|_2^2 \,
\lambda_{\max}(D^{-2}_\lambda - \bar{D}^{-2}_\lambda) \\ 
&\leq 2(r^2+\sigma^2) \underbrace{\max_{i=1,\ldots,n} \, \Big|(1 + \delta_i^T
  A_i \delta_i)^2 - \big(1+\gamma m(-\lambda)\big)^2 \Big|}_{c}, 
\end{align*}
where in the second step, we used $\|(I-S_\lambda)y\|_2^2/n \leq \|y\|_2^2/n
\leq 2(r^2+\sigma^2)$, which holds almost surely for large enough $n$, by the
strong law of large numbers. Meanwhile,   
\begin{multline*}
c \leq \underbrace{\max_{i=1,\ldots,n} \, \Big|(1+\delta_i^T A_i \delta_i)^2 -  
    \big(1+\Tr(A_i)/n\big)^2\Big|}_{d_1} + \underbrace{\max_{i=1,\ldots,n} \,  
    \Big|\big(1+\Tr(A_i)/n\big)^2 -\big(1+\Tr(A)/n\big)^2\Big|}_{d_2} + \\   
 \underbrace{\Big|\big(1+\Tr(A)/n\big)^2 - \big(1+\gamma
  m(-\lambda)\big)^2\Big|}_{d_3}.   
\end{multline*}
By the Marchenko-Pastur theorem, we have $d_3 \to 0$ almost surely. 
Using $u^2-v^2 = (u-v)(u+v)$ on the maximands in $d_2$,    
\begin{align*}
d_2 &= \max_{i=1,\ldots,n} \, \big|\Tr(A_i)/n - \Tr(A)/n\big|
  \big|2 + \Tr(A_i)/n + \Tr(A)/n\big| \\ 
&\leq \frac{2(1+1\lambda)}{n} \max_{i=1,\ldots,n} \, |\Tr(A_i-A)| \\   
&\leq \frac{2(1+1\lambda)}{n} \max_{i=1,\ldots,n} \,
  \frac{|\Tr(A \delta_i\delta_i^T A)|}{|1  - \delta_i^T A \delta_i|} \\ 
&\leq \frac{2(1+1\lambda)}{n} \frac{s_{\max}+\lambda}{\lambda^3}   
  \max_{i=1,\ldots,n} \, \|\delta_i\|_2^2 \\
&\leq \frac{2(1+1\lambda)}{n} \frac{(s_{\max}+\lambda) s_{\max}}{\lambda^3} \\ 
&\leq \frac{4(1+1\lambda)(2+\lambda)}{n \lambda^3} \to 0.
\end{align*}
In the second line above, we used $\Tr(A_i)/n \leq \lambda_{\max}(A_i) \leq
1/\lambda$, $i=1,\ldots,n$, and also $\Tr(A)/n \leq 1/\lambda$.  In the 
third line, we used the Sherman-Morrison-Woodbury formula.  In the fourth
line, for each summand $i=1,\ldots,n$, we upper bounded the numerator by 
$\lambda_{\max}(A)^2 \|\delta_i\|_2^2 \leq \|\delta_i\|_2^2/\lambda^2$, and
lower bounded the denominator by $\lambda/(s_{\max}+\lambda)$, with   
$s_{\max}=\lambda_{\max}(X^T X/n)$ (which follows from a short calculation
using the eigendecomposition of $X^T X/n$). In the fifth line, we used 
$\|\delta_i\|_2^2 = e_i^T (XX^T /n) e_i \leq s_{\max}$, and in the sixth line,
we used $s_{\max} \leq 2$ almost surely for sufficiently large $n$, by the
Bai-Yin theorem \citep{bai1993limit}. 

Now we will show that $d_1 \to 0$ almost surely by showing that the convergence
in \eqref{eq:cv_denom_lim} is rapid enough.  As before, using $u^2-v^2 = 
(u-v)(u+v)$ on the maximands in $d_1$, we have
\begin{align*}
d_1 &= \max_{i=1,\ldots,n} \, \big|\delta_i^T A_i \delta_i- \Tr(A_i)/n\big|
  \big|2 + \delta_i^T A_i \delta_i + \Tr(A_i)/n\big|^2 \\  
&\leq (2 + 3/\lambda) \max_{i=1,\ldots,n} \, 
  \underbrace{\big|\delta_i^T A_i \delta_i-  \Tr(A_i)/n \big|}_{\Delta_i}.   
\end{align*}
Here we used that for $i=1,\ldots,n$, we have $\Tr(A_i)/n \leq
1/\lambda$, and similarly, $\delta_i^T A_i  \delta_i \leq \|\delta_i\|_2^2 /
\lambda \leq s_{\max}/\lambda \leq 2/\lambda$,
with the last inequality holding almost surely for sufficiently large $n$, by
the Bai-Yin  theorem \citep{bai1993limit}.  To show \smash{$\max_{i=1,\ldots,n}
  \Delta_i \to 0$} almost surely, we start with the union bound and Markov's 
inequality, where $q=2+\eta/2$, and $t_n$ is be specified later, 
$$
\P\bigg(\Big(\max_{i=1,\ldots,n} \, \Delta_i\Big) > t_n\bigg) 
\leq \sum_{i=1}^n \E(\Delta_i^{2q}) t_n^{-2q} 
\leq C \lambda^{-q} n p^{-q} t_n^{-2q}.   
$$
In the last step, we used that for $i=1,\ldots,n$, we have $\E(\Delta_i^{2q})
\leq C \lambda_{\max}(A_i)^q p^{-q} \leq C \lambda^{-q} p^{-q}$ for a constant  
$C>0$ that depends only on $q$ and the assumed moment bound, of order 
$2q=4+\eta$, on the entries of $P_x$.  This expectation bound is a consequence of 
the trace lemma in Lemma B.26 of \citet{bai2010spectral}, as explained in the
proof of Lemma 7.6 in \citet{dobriban2018high}.  Hence choosing
$t_n^{-2q}=p^{\eta/4}$, from the last display we have 
$$
\P\bigg(\Big(\max_{i=1,\ldots,n} \, \Delta_i\Big) > t_n\bigg)
 \leq C \lambda^{-q/2} n p^{-2-\eta/2} p^{\eta/4} 
\leq \frac{2C \lambda^{-q/2}}{\gamma} p^{-1-\eta/4}, 
$$
where in the last step we used $n/p \leq 2/\gamma$ for large enough $n$. 
Since the right-hand side above is summable in $p$, we conclude by the
Borel-Cantelli lemma that \smash{$\max_{i=1,\ldots,n} \Delta_i \to 0$} almost  
surely, and therefore also $d_1 \to 0$ almost surely.  This finishes the proof 
that $b \to 0$.

\paragraph{Relating back to GCV.} Returning to \eqref{eq:cv_quad_diff}, observe   
$$
a = y^T (I-S_\lambda)^2 y/n \cdot \big(1+\gamma m(-\lambda)\big)^2.
$$
The first term in the product on the right-hand side above is precisely the
numerator in GCV, whose convergence was established in \eqref{eq:gcv_numer}.
Therefore, to show that the limit of $a$, i.e., of $\cv_n(\lambda)$, is the same
as it was for $\gcv_n(\lambda)$, we will need to show that the second term in
the product above matches the limit for the GCV denominator, in
\eqref{eq:gcv_denom}.  That is, we will need to show that 
\begin{equation}
\label{eq:cv_denom_claim}
\big(1+\gamma m(-\lambda)\big)^2 = \frac{1}{(1-\gamma(1-m(-\lambda)))^2}, 
\end{equation}
or equivalently,
$$
1-\gamma\big(1-m(-\lambda)\big) = \frac{1}{1+\gamma m(-\lambda)}.  
$$
As before, it helps to reparametrize in terms of the companion Stieltjes
transform, giving
$$
v(-\lambda) = \frac{1}{\lambda + \lambda v(-\lambda) + \gamma - 1},
$$
of equivalently,
$$
\big(\lambda + \lambda v(-\lambda) + \gamma - 1\big) v(-\lambda) = 1.  
$$
Adding $v(-\lambda)$ to both sides, then dividing both sides by $1+v(-\lambda)$,
yields 
$$
\frac{(\lambda + \lambda v(-\lambda) + \gamma) v(-\lambda)}{1 + v(-\lambda)} 
= 1,  
$$
and dividing through by $v(-\lambda)$, then rearranging, gives
$$
\frac{1}{v(-\lambda)} = \lambda + \frac{\gamma}{1 + v(-\lambda)}.
$$
This is precisely the {\it Silverstein equation} \citep{silverstein1995strong},
which relates the companion Stieltjes transform $v$ to the weak limit $H$ of the
spectral measure of the feature covariance matrix $\Sigma$, which in the current
isotropic case $\Sigma=I$, is just $H=\delta_1$ (a point mass at 1).  Hence, we
have established the desired claim \eqref{eq:cv_denom_claim}, and from the
limiting analysis of GCV completed previously, we have that, almost surely, 
\smash{$a \to \sigma^2 + \sigma^2 \gamma (m(-\lambda) - \lambda(1 - \alpha
  \lambda) m'(-\lambda))$}.  This is indeed also the almost sure limit of
$\cv_n(\lambda)$, from \eqref{eq:cv_quad_diff} and $b \to 0$
almost surely. The proof of uniform convergence, and thus convergence of the
CV-tuned risk, follows similar arguments to the GCV case, and is omitted.    

\subsection{Misspecified model results}

These results follow because, as argued in Section \ref{sec:risk_mis_iso}, the
misspecified case can be reduced to a well-specified case after we make the 
substitutions in \eqref{eq:subs}.

\subsection{CV and GCV simulations}
\label{app:sims_cv}

Figures \ref{fig:risk_cv} and \ref{fig:risk_cv_ab} investigate the effect of
using CV and GCV tuning for ridge regression, under the same simulation setup as
Figures \ref{fig:risk_rg} and \ref{fig:risk_rg_ab}, respectively.  It is worth
noting that for these figures, there is a difference in how we compute
finite-sample risks, compared to what is done for all of the other figures in
the paper. In all others, we can compute the finite-sample risks exactly,
because, recall, our notion of risk is conditional on $X$,  
$$
R_X(\hbeta;\beta) = \E\big[\|\hbeta-\beta\|_\Sigma^2 \,|\, X \big], 
$$ 
and this quantity can be computed analytically for linear estimators like
min-norm least squares and ridge regression. When we tune ridge by minimizing CV
or GCV, however, we can no longer compute its risk analytically, and so in our
experiments we approximate the above expectation by an average over 20
repetitions. Therefore, the finite-sample risks in Figures \ref{fig:risk_cv} and
\ref{fig:risk_cv_ab} are (only a little bit) farther from their asymptotic
counterparts compared to all other figures in this paper, and we have included
the finite-sample risk of the optimally-tuned ridge estimator in these figures,
when the risk is again computed in the same way (approximated by an average
over 20 repetitions), as a reference.  All this being said, we still see
excellent agreement between the risks under CV, GCV, and optimal tuning, and
these all lie close to their (common) asymptotic risk curves, throughout. 

\begin{figure}[p]
\vspace{-40pt}
\centering
\includegraphics[width=0.725\textwidth]{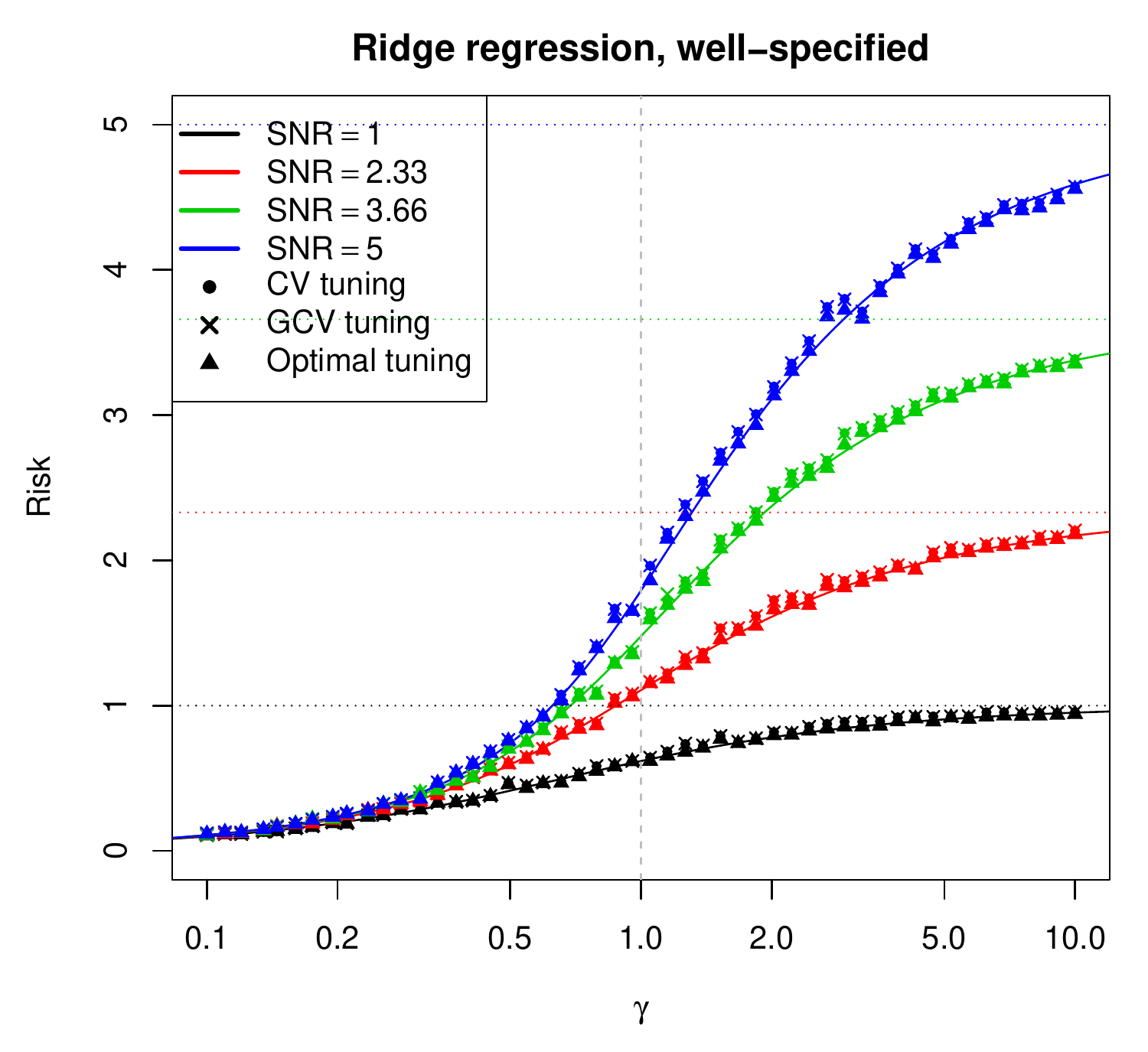} 
\caption{\footnotesize Asymptotic risk curves for the optimally-tuned ridge
  regression estimator (from Theorem \ref{thm:risk_ridge}), under the same setup
  as Figure \ref{fig:risk_rg} (well-specified model).  Finite-sample risks for
  ridge regression under CV, GCV, and optimal tuning are plotted as circles,
  ``x'' marks, and triangles, respectively.  These are computed with $n=200$,
  $p=[\gamma n]$, across various values of $\gamma$, from features $X$ having 
  i.i.d.\ $N(0,1)$ entries.}   
\label{fig:risk_cv}

\bigskip
\includegraphics[width=0.725\textwidth]{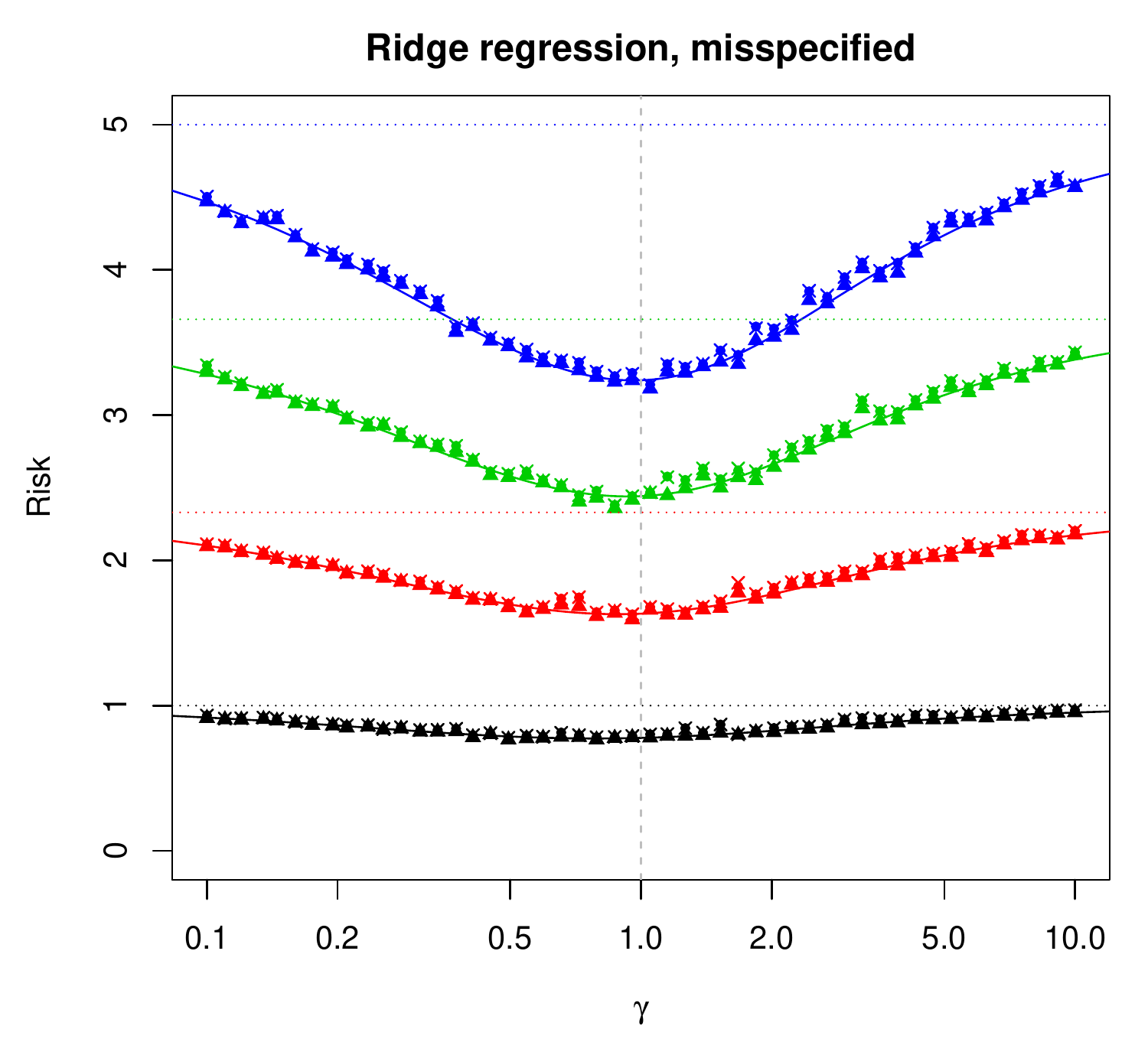}
\caption{\footnotesize Asymptotic risk curves for the optimally-tuned ridge
  regression estimator (from Theorem \ref{thm:risk_ridge}), under the same setup
  as Figure \ref{fig:risk_rg_ab} (misspecified model).  Finite-sample risks for
  ridge regression under CV, GCV, and optimal tuning are again plotted as
  circles, ``x'' marks, and triangles, respectively.  These are again computed
  with $n=200$, $p=[\gamma n]$, across various values of $\gamma$, from features
  $X$ having i.i.d.\ $N(0,1)$ entries.}
\label{fig:risk_cv_ab}
\end{figure}

\section{Nonlinear model: Proof of Theorem \ref{cor:purely_nonlinear}}

In this section we present the proof of  Theorem
\ref{cor:purely_nonlinear}. In fact, in the next section we will state a more general
result which applies to activation functions $\varphi$ eith a non-vanishing linear
component $\E[G\varphi(G)]\neq 0$ (here and below $G$ denotes a standard normal random variable). 

\subsection{A more general theorem}
\label{app:Moregeneral}

We denote by $\bQ \in \reals^{p\times p}$
the Gram matrix of the weight vectors $\bw_i$, $i=1,\ldots,p$ (rows of $W \in
\R^{p \times d}$), with diagonals set to zero. Namely, for each $i,j$,  
$$
Q_{ij} = \<\bw_i,\bw_j\> 1\{i \neq j\}.
$$
Let $N=p+n$ and define the symmetric matrix $\bA(s,t) \in \reals^{N\times N}$, 
for $s \geq t \geq 0$, with the block structure:
\begin{equation}
\bA(s,t) = \begin{bmatrix}
s\id_p +t\bQ & \frac{1}{\sqrt{n}} \bX^{\sT} \\ 
\frac{1}{\sqrt{n}} \bX & 0_pn
\end{bmatrix},\label{eq:Adef}
\end{equation}
where $I_p \in \R^{p\times p}$, $0_n\in\R^{n\times n}$ are the
identity and zero matrix, respectively.
Note that the matrix $A(s)$ introduced in the main text, see Eq.~\eqref{eq:aden}, is recovered by setting $t=0$.
We introduce the following resolvents (as usual, these are defined for
$\Im(\xi)>0$ and by analytic continuation, whenever possible, for $\Im(\xi) =
0$):  
\begin{align*}
m_{1,n}(\xi,s,t) &= 
\E\Big\{\big(\bA(s,t) - \xi\id_{N}\big)^{-1}_{1,1}\Big\} = 
\E M_{1,n}(\xi,s,t),\\   
 M_{1,n}(\xi,s,t) &=
\frac{1}{p}\Tr_{[1,p]}\Big\{\big(\bA(s,t) - \xi\id_{N}\big)^{-1}\Big\},\\  
m_{2,n}(\xi,s,t) &= 
\E\Big\{\big(\bA(s,t) -\xi\id_{N}\big)^{-1}_{p+1,p+1}\Big\} =
\E M_{2,n}(\xi,s,t),\\   
 M_{2,n}(\xi,s,t) &= 
\frac{1}{n}\Tr_{[p+1,p+n]}\Big\{\big(\bA(s,t) - \xi\id_{N}\big)^{-1}\Big\}.  
\end{align*}
Whenever clear from
the context, we will omit the arguments from block matrix and resolvents, and 
write $\bA=\bA(s,t)$, $m_{1,n} =m_{1,n}(\xi,s,t)$, and $m_{2,n}
=m_{2,n}(\xi,s,t)$.    

The next theorem is the the central random matrix results, and characterizes the asymptotics of $m_{1,n}$, $m_{2,n}$.
It provides a generalization of Lemma \ref{thm:ResolventKernel}.
\begin{theorem}\label{thm:ResolventKernel-gen}
Assume the model \eqref{eq:data_x}, \eqref{eq:data_y}, where each
$x_i = \varphi(W z_i) \in \R^p$, for $z_i \in \R^d$ having i.i.d.\
entries from $N(0,1)$, $W \in \R^{p \times d}$ having i.i.d.\ entries from  
$N(0,1/d)$ (with $W$ independent of $z_i$), and for $\varphi$ an activation
function that acts componentwise.  Assume that $|\varphi(x)|\le
c_0(1+|x|)^{c_0}$ for a constant $c_0>0$. Also, for $G \sim N(0,1)$, assume that
the standardization conditions hold: $\E[\varphi(G)]=0$ and
$\E[\varphi(G)^2]=1$.

  Consider $\Im(\xi)>0$
or $\Im(\xi)=0$, $\Re(\xi)<0$, with $s\ge t\ge 0$. Let $m_1$ and $m_2$ be the
unique solutions of the following fourth degree equations: 
\begin{align}
m_2&= \bigg(-\xi-\gamma m_1+\frac{\gamma
     c_1m_1^2(c_1m_2-t)}{m_1(c_1m_2-t)-\psi}\bigg)^{-1},\label{eq:M2}\\ 
m_1&=  \bigg(-\xi-s-\frac{t^2}{\psi}m_1-m_2 +
     \frac{t^2\psi^{-1}m_1^2(c_1m_2-t)-2tc_1m_1m_2+c_1^2m_1m^2_2}
     {m_1(c_1m_2-\psi)-\psi}\bigg)^{-1},\label{eq:M1}   
\end{align}
subject to the condition of being analytic functions for $\Im(z)>0$, and
satisfying $|m_1(z,s,t)|, |m_2(z,s,t)| \le 1/\Im(z)$ for $\Im(z)>C$ (with $C$ a
sufficiently large constant).  Then, as $n,p,d \to \infty$, such that $p/n \to
\gamma$ and $d/p \to \psi$, we have almost surely (and in $L^1$),
\begin{align}
\lim_{n,p,d\to\infty} M_{1,n}(\xi,s,t) &= m_1(\xi,s,t),\label{eq:AS-1}\\ 
\lim_{n,p,d\to\infty} M_{2,n}(\xi,s,t) &= m_2(\xi,s,t).\label{eq:AS-2}
\end{align}
\end{theorem}

The proof of this theorem is given in Appendix \ref{app:ResolventKernel}. Now
define   
\begin{equation}
S_n(z) = \frac{1}{p}\Tr\big((\hSigma-z\id_p)^{-1}\big),\label{eq:Sn}  
\end{equation}
where recall \smash{$\hSigma=X^T X/n$}. As a corollary of the above, we obtain
the asymptotic Stieltjes transform of the eigenvalue distribution of
\smash{$\hSigma$} (this result was first derived in
\cite{pennington2017nonlinear}).   

\begin{corollary}\label{coro:Stieltjes-gen}
Assume the conditions of Theorem \ref{thm:ResolventKernel-gen}.  Consider
$\Im(\xi)>0$. As $n,p,d \to \infty$, with $p/n \to \gamma$ and $d/p \to \psi$,
the Stieltjes transform of spectral distribution of \smash{$\hSigma$} in
\eqref{eq:Sn} satisfies almost surely (and in $L^1$) $S_{n}(\xi)\to s(\xi)$
where $s$ is a nonrandom function that uniquely solves the following
equations (abbreviating $s=s(\xi)$):
\begin{align}
-1-\xi^2 s &= \om_1\om_2-
\frac{c_1^2\om_1^2\om_2^2}{c_1\om_1\om_2-\psi},\label{eq:coroS1}\\  
\om_1 & =  \xi s,\label{eq:coroS2}\\
\om_2 &= \frac{\gamma-1}{\xi}+\gamma \xi s,\label{eq:coroS3} 
\end{align}
subject to the condition of being analytic for $\Im(z)>0$, and satisfying 
$|s(z^2)|<1/\Im(z^2)$ for $\Im(z)>C$ (where $C$ is a large enough constant). 
When $c_1=0$, the function $s$ is the Stieltjes transform of the
Marchenko-Pastur distribution.    
\end{corollary}

We refer to Appendix \ref{app:CoroStieltjes} for a proof of this corollary. The
next lemma connects the above resolvents to the variance of min-norm
least squares.
\begin{lemma}\label{lemma:ResolventToVar-gen}
Assume the conditions of Theorem \ref{thm:ResolventKernel-gen}.
Let $m_1$, $m_2$ be the asymptotic resolvents given in Theorem
\ref{thm:ResolventKernel-gen}. Define
$$
m(\xi,s,t) = \gamma m_1(\xi,s,t) + m_2(\xi,s,t).
$$
Then for $\gamma\neq 1$, the following Taylor-Laurent expansion holds around
$\xi=0$: 
\begin{align}
-\partial_x m(\xi,x,c_1x)\big|_{x=0} 
= \frac{D_{-1}}{\xi^2}+D_0+O(\xi^2),\label{eq:Laurent}
\end{align}
with each $D_i=D_i(\gamma,\psi,c_1)$. Furthermore, for the ridge regression 
estimator \smash{$\hbeta_\lambda$} in \eqref{eq:ridge}, as $n,p,d \to \infty$,
such that $p/n \to \gamma \in (0,\infty)$, $d/p \to \psi \in (0,1)$, the
following ridgeless limit holds almost surely:
$$
\lim_{\lambda \to 0^+} \, \lim_{n,p,d\to\infty} \, 
V_X(\hbeta_\lambda;\beta) = D_0.
$$
\end{lemma}

The proof of this lemma can be found in Appendix
\ref{app:ResolventToVar}.

\subsection{Proof of Theorem \ref{thm:ResolventKernel-gen}}
\label{app:ResolventKernel}

The proof follows the approach developed by \cite{cheng2013spectrum} to determine the asymptotic spectral measure
of symmetric kernel random matrices. The latter is in turn inspired to the classical resolvent proof of the semicircle law for Wigner matrices, see 
\cite{anderson2009introduction}.  The present calculation  is somewhat longer because of the block structure of the matrix $\bA$, 
but does not present technical difficulties. We will therefore provide a proof outline, referring to \cite{cheng2013spectrum} for further detail.

Let $\mu_d$ be the distribution of $\<\bw,\bx\>$ when $\bw\sim\normal(0,\id_d/d)$, $\bx\sim\normal(0,\id_d)$. 
By the central limit theorem $\mu_d$ converges weakly to $\mu_G$ (the standard Gaussian measure) as $d\to\infty$.
Further  $\E\{f(\<\bw,\bx\>)\}\to \int f(z)\, \mu_G(\de z)$ for any continuous function $f$, with $|f(z)|\le C(1+|z|^{C})$ for some constant $C$. 
As shown in \cite{cheng2013spectrum},
we can construct the following orthogonal decomposition of  the activation function $\varphi$ in $L_2(\reals,\mu_d)$:
\begin{align}
\varphi(x) = a_{1,d}\, x +\varphi_{\perp}(x)\, .
\end{align}
This decomposition satisfies the following properties:
\begin{enumerate}
\item As mentioned, it is an orthogonal decomposition: $\int x \varphi_{\perp}(x)\, \mu_d(\de x) =0$. Further,
by symmetry and the normalization assumption, $\int x \, \mu_d(\de x) =\int  \varphi_{\perp}(x)\, \mu_d(\de x) =0$.
\item $a_{1,d}^2\to c_1$ as $d\to\infty$.
\item $\int \varphi_{\perp}(x)^2\, \mu_d(\de x) \to 1-c_1$, $\int \varphi_{\perp}(x)^2\, \mu_G(\de x) \to 1-c_1$,   as $d\to\infty$.
\end{enumerate}

Finally, recall the following definitions for the random resolvents:
\begin{align}
M_{1,n}(\xi,s,t) =  \frac{1}{p}\Tr_{[1,p]}\big[(\bA(n)-\xi\id_{N})^{-1}\big]\, ,\;\;\;\;\;
M_{2,n}(\xi,s,t) =  \frac{1}{n}\Tr_{[p+1,p+n]}\big[(\bA(n)-\xi\id_{N})^{-1}\big]\, .\label{eq:Mdef}
\end{align}

In the next section,  we will summarize some basic  facts about the resolvent $m_{1,n}$, $m_{2,n}$
and their analyticity, and some concentration properties of $M_{1,n}, M_{2,n}$. 
We will then derive Eqs.~(\ref{eq:M2}) and (\ref{eq:M1}). Thanks to the analyticity properties, we will assume throughout these
derivations $\Im(\xi)\ge C$ for $C$ a large enough constant. Also, we will assume $\gamma, \psi, \varphi$ to be fixed throughout, and will not 
explicitly point out the dependence with respect to these arguments.

\subsubsection{Preliminaries}

The functions $m_{1,n}$, $m_{2,n}$ are Stieltjes transform
of suitably defined probability measures on $\reals$, and therefore enjoy some important properties.
\begin{lemma}\label{lemma:BasicProperties}
Let $\complex_+=\{z:\; \Im(z)>0\}$ be the upper half of the complex plane.
The functions $\xi\mapsto m_{1,n}(\xi)$, $\xi\mapsto m_{2,n}(\xi)$ have the following properties:
\begin{enumerate}
\item[$(a)$] $\xi\in\complex_+$, then $|m_{1,n}|$, $|m_{2,n}|\le 1/\Im(\xi)$.
\item[$(b)$] $m_{1,n}$, $m_{2,n}$ are analytic on $\complex_+$ and map $\complex_+$ into $\complex_+$.
\item[$(c)$] Let $\Omega\subseteq\complex_+$ be a set with an accumulation point. If $m_{a,n}(\xi)\to m_a(\xi)$ for all $\xi\in\Omega$,
then $m_a(\xi)$ has a unique analytic continuation to $\complex_+$ and  $m_{a,n}(\xi)\to m_a(\xi)$ for all $\xi\in\complex_+$.
\end{enumerate}
\end{lemma}
\begin{proof}
Consider, to be definite $m_{1,n}$.
Denoting by $(\lambda_i)_{i\le N}$, $(\bv_i)_{i\le N}$,  the eigenvalues and eigenvectors of $\bA(n)$, we have
\begin{align*}
m_{1,n}(\xi) &= \E\frac{1}{p}\Tr_{[1,p]}\big[(\bA(n)-\xi\id_{N})^{-1}\big]\\
 &= \E\sum_{i=1}^N\frac{1}{\lambda_i-\xi}\, \frac{1}{p}\|\Proj_{[1,p]}\bv_i\|^2\\
&= \int \frac{1}{\lambda-\xi} \, \mu_{1,n}(\de\lambda)\, .
\end{align*}
Where we defined the probability measure
\begin{align*}
\mu_{1,n} \equiv \E\frac{1}{p}\sum_{i=1}^N \|\Proj_{[1,p]}\bv_i\|^2\delta_{\lambda_i}\, .
\end{align*}
Properties $(a)$-$(c)$ are then standard consequences of $m_{1,n}$ being a Stieltjes transform (see, for instance, \cite{anderson2009introduction})
\end{proof}

\begin{lemma}\label{eq:BD-Diff}
Let $\bW\in\reals^{N\times N}$ be a symmetric positive-semidefinite matrix $\bW\succeq \bfzero$, and denote by $\bw_i$ its $i$-th column,
with the $i$-th entry set to $0$. Let $\bW^{(i)} \equiv \bW-\bw_i\bfe_i^{\sT}-\bfe_i\bw_i^{\sT}$, where $\bfe_i$ is the $i$-th element of the canonical basis
(in other words, $\bW^{(i)}$ is obtained from $\bW$ by zeroing all elements in the $i$-th row and column except on the diagonal). 
Finally, let $\xi\in\complex$ with $\Im(\xi)\ge\xi_0>0$ or $\Im(\xi) =0$ and $\Re(\xi)\le -\xi_0<0$. 

Then
\begin{align}
\Big|\Tr_{S}\big[(\bW-\xi\id_N)^{-1}\big]- \Tr_{S}\big[(\bW^{(i)}-\xi\id_N)^{-1}\big]\Big|\le \frac{3}{\xi_0}\, .
\end{align}
\end{lemma}
\begin{proof}
Without loss of generality, we will assume $i=N$, and write $\bW_*= (W_{ij})_{i,j\le N-1}$ $\bw = (\bw_{N,i})_{i\le N-1}$, $\omega = W_{NN}$.
 Define
\begin{align}
\left[
\begin{matrix}
\bR & \br\\
\br^{\sT} & \rho
\end{matrix}\right]\equiv (\bW-\xi\id_N)^{-1}\, ,
\end{align}
and $\bR^{(N)}$, $\rho^{(N)}$ the same quantities when $\bW$ is replaced by $\bW^{(N)}$.
 Then, by Schur complement formula, it is immediate to get the formulae
\begin{align}
\bR &= (\tbW_*-\xi\id_{N-1})^{-1} \, ,\;\;\;\; \tbW_* \equiv \bW-\frac{\bw\bw^{\sT}}{\omega-\xi}\, ,\\
\bR^{(N)} & = (\bW_*-\xi\id_{N-1})^{-1}\, ,\\
\rho & = \frac{1}{\omega-\xi-\bw^{\sT}\bR^{(N)}\bw} \, ,\\
\rho^{(N)} & = \frac{1}{\omega-\xi}\, .
\end{align}
Notice that since $\bW\succeq 0$. we must have $\omega\ge \bw^{\sT}\bW_*^{-1}\bw$.
For $\xi\in \reals$, $\xi\le -\xi_0$, this implies $\omega\ge \bw^{\sT}\bR^{(N)}\bw$, and therefore 
$\rho\le 1/\xi_0$. For $\Im(\xi)\ge \xi_0$, we get $\Im(\bw^{\sT}\bR^{(N)}\bw)\ge 0$, and therefore $\Im(\omega-\xi-\bw^{\sT}\bR^{(N)}\bw)\le -\xi_0$.
We thus conclude that, in either case
\begin{align}
|\rho|\le \frac{1}{\xi_0}\, .
\end{align}
Let $S_0\equiv  S\setminus \{N\}$ and $S_1 \equiv S\cap\{N\}$ (i.e. $S_1=\{N\}$ or $S_1=\emptyset$ depending whether $N\in S$ or not).
Define $\Delta_A \equiv \big|\Tr_{A}\big[(\bW-\xi\id_N)^{-1}\big]- \Tr_{A}\big[(\bW^{(N)}-\xi\id_N)^{-1}\big]\big|$. Then, using the bounds given above,
\begin{align}
\Delta_{S_1} &\le  \left|\rho-\rho^{(N)}\right| \le |\rho|+|\rho^{(N)}|\le \frac{2}{\xi_0} \, .\label{eq:DeltaS1}
\end{align}

Next consider $\Delta_{S_0}$.  Notice that
\begin{align}
\bR-\bR^{(N)} & = \rho\, \bR^{(N)}\bw (\bR^{(N)}\bw )^{\sT}\\
& = \frac{(\bW_*-\xi\id)^{-1}\bw[(\bW_*-\xi\id)^{-1}\bw]^{\sT}}{\omega-\xi-\bw^{\sT}(\bW_*-\xi\id_{N-1})\bw}\, .\label{ew:DiffR}
\end{align}
We will distinguish two cases.

\noindent{\bf Case I: $\xi\in\reals$, $\xi\le -\xi_0<0$.} 
Notice that 
\begin{align}
\Delta_{S_0} &= \big|\Tr_{S_0}(\bR-\bR^{(N)})\big|  = |\rho| \sum_{i\in S_0} (\bR^{(N)}\bw )_i^2\\
&\le  \big|\Tr_{[N-1]}(\bR-\bR^{(N)})\big|.
\end{align}
Denoting by $\{\tlambda_i\}_{i\le N}$ the eigenvalues of $\tbW_*$, we thus get
\begin{align}
\Delta_{S_0} &\le \left|\sum_{i=1}^N\frac{1}{\tlambda_i-\xi}-\sum_{i=1}^N\frac{1}{\lambda_i-\xi} \right|\, .
\end{align}
Since $\tbW_*$ is a rank-one deformation of $\bW_*$, the eigenvalues $\tlambda_i$ interlace the $\lambda_i$'s.
Since both sets of eigenvalues are nonnegative, the difference above is upper bounded by the total variation of the function
$x\mapsto (x-\xi)^{-1}$ on $\reals_{\ge 0}$ which is equal to $1/\xi_0$. We thus get
\begin{align}
\Delta_{S_0} &\le \frac{1}{\xi_0}\, .\label{eq:DeltaS0}
\end{align}

\noindent{\bf Case II: $\xi\in\complex$, $\Im(\xi)\ge \xi_0>0$.} Let $\Proj_{S_0}:\reals^{N-1}\times\reals^{N-1}$ the projector which zeroes the coordinates outside $S_0$
(i.e. $\Proj_{S_0} = \sum_{i\in S_0}\bfe_i\bfe_i^{\sT}$).
Denote by $(\lambda_i,\bv_i)$ the eigenpairs of $\bW_*$.
Using Eq.~\eqref{ew:DiffR}, we get
\begin{align}
\Delta_{S_0}&\equiv  \frac{A_1}{A_2}\, ,\\
A_1 &\equiv \left|\sum_{i,j=1}^N \<\bv_i,\Proj_{S_0}\bv_j\>\frac{\<\bv_i,\bw\>\<\bv_j,\bw\>}{(\lambda_i-\xi)(\lambda_j-\xi)}\right|\, .\\
A_2 &\equiv \left|\omega-\xi-\sum_{i=1}^N \frac{\<\bv_i,\bw\>^2}{\lambda_i-\xi}\right|\, .
\end{align}
Letting $V_{ij} =  \<\bv_i,\Proj_{S_0}\bv_j\>$, we have $\|\bV\|_{\op}\le \|\Proj_{S_0}\|_{\op}\le 1$. Therefore, defining $u_i= \<\bw,\bv_i\>/(\lambda_i-\xi)$,
\begin{align}
A_1&= \<\bu,\bV\bu\>\le \|\bu\|_2^2\\
& \le \sum_{i=1}^N \frac{\<\bv_i,\bw\>^2}{|\lambda_i-\xi|^2}\\
& \le \sum_{i=1}^N \frac{\<\bv_i,\bw\>^2}{(\lambda_i-\Re(\xi))^2+\Im(\xi)^2}\, .
\end{align}
Further
\begin{align}
A_2 &\ge \left|\Im(\xi)+\sum_{i=1}^N \Im\left(\frac{\<\bv_i,\bw\>^2}{\lambda_i-\xi}\right)\right|\\
&\ge \left|\Im(\xi)+\sum_{i=1}^N \frac{\<\bv_i,\bw\>^2\Im((\xi)}{(\lambda_i-\Re(\xi))^2+\Im(\xi)^2}\right|\\
&\ge \xi_0 A_1\, ,
\end{align}
which implies that Eq.~\eqref{eq:DeltaS0} also holds in this case.

Using Eqs.~\eqref{eq:DeltaS1} and \eqref{eq:DeltaS0} yields the desired claim.
\end{proof}

\begin{lemma}\label{lemma:Concentration}
If $\Im(\xi) \ge \xi_0>0$, or $\Re(\xi)\le -\xi_0<0$, with $s\ge t\ge 0$. Then there exists $c_0= c_0(\xi_0)$ such that, for $i\in\{1,2\}$,
\begin{align}
\prob\big(\big|M_{i,n}(\xi,s,t) -m_{i,n}(\xi,s,t)\big|\ge u\big) \le 2\, e^{-c_0n u^2}\, .\label{eq:Concentration}
\end{align}
In particular, for $\Im(\xi) >0$, or $\Re(\xi)<0$, then $|M_{i,n}(\xi,s,t) -m_{i,n}(\xi,s,t)|\to 0$ almost surely.
\end{lemma}
\begin{proof}
The proof is completely analogous to \cite[Lemma 2.4]{cheng2013spectrum}, except that we use Azuma-Hoeffding instead of Burkholder's inequality. 
Consider $M_{1,n}(\xi)$ (we omit the arguments $s,t$ for the sake of simplicity). We construct the martingale
\begin{align}
Z_{\ell} = \begin{cases}
\E\{M_{1,n}(\xi)|\{\bw_a\}_{a\le \ell}\}& \;\; \mbox{ if $1\le \ell\le p$,}\\
\E\{M_{1,n}(\xi)|\{\bw_a\}_{a\le p}, \{\bz_{i}\}_{i\le \ell-p}\}& \;\; \mbox{ if $p+1\le \ell\le N$,}
\end{cases}
\end{align}
By Lemma \ref{eq:BD-Diff} $Z_{\ell}$ has bounded differences $|Z_{\ell}-Z_{\ell-1}|\le 3/\xi_0$, whence the claim \eqref{eq:Concentration} follows.
The almost sure convergence of $|M_{i,n}(\xi,s,t) -m_{i,n}(\xi,s,t)|\to 0$ is implied by Borel-Cantelli. 
\end{proof}

\begin{lemma}\label{lemma:Contraction}
Let $\bF:\complex^2\to \complex^2$ be the mapping $(m_1,m_2)\mapsto \bF(m_1,m_2)$ defined by the right-hand side of Eqs.~\eqref{eq:M1}, \eqref{eq:M2},
namely
\begin{align}
F_1(m_1,m_2) & \equiv  \left(-\xi-s-\frac{t^2}{\psi}m_1-m_2+\frac{t^2\psi^{-1}m_1^2(c_1m_2-t)-2tc_1m_1m_2+c_1^2m_1m^2_2}{m_1(c_1m_2-\psi)-\psi}\right)^{-1}\, ,\\
F_2(m_1,m_2) & \equiv \left(-\xi-\gamma m_1+\frac{\gamma c_1m_1^2(c_1m_2-t)}{m_1(c_1m_2-t)-\psi}\right)^{-1}\, .
\end{align}
Define $\disk(r) = \{z:\;  |z|<r\}$ to be the disk of radius $r$ in the complex plane. Then,
there exists $r_0>0$ such that,  for any $r, \delta>0$  there exists $\xi_0=\xi_0(s,t,r,\delta)>0$ such that, if $\Im(\xi)>\xi_0$, 
then $\bF$ maps $\disk(r_0)\times\disk(r_0)$ into $\disk(r)\times\disk(r)$ and further is Lipschitz continuous, with Lipschitz constant at most $\delta$
on that domain.

In particular Eqs.~\eqref{eq:M1}, \eqref{eq:M2} admit a unique solution with $|m_1|,|m_2|<r_0$ for $\xi>\xi_0$. 
\end{lemma}
\begin{proof}
Setting $\bm \equiv (m_1,m_2)$, we note that $\bF(\bm) = (-\xi+\bG(\bm))^{-1}$, where $\bm\mapsto \bG(\bm)$
is $L$-Lipschitz continuous in a neighborhood of of $\bfzero$, $\disk(r_0)\times\disk(r_0)$, with $\bG(\bfzero) =(s,0)$. We therefore have, for $|m_i|\le r_0$,
\begin{align}
|F_i(\bm)|\le \frac{1}{|\xi|-|F_i(\bm)|} \le \frac{1}{\xi_0-|s|-2Lr_0} \, .
\end{align}
whence $|F_i(\bm)|< r$  by taking $\xi_0$ large enough. Analogously
\begin{align}
\|\nabla F_i(\bm)\|\le \frac{1}{(|\xi|-|F_i(\bm)|)^2}\,  \|\nabla G_i(\bm)\| \le \frac{L}{(\xi_0-|s|-2Lr_0)^2} \, .
\end{align}
Again, the claim follows by taking $\xi_0$ large enough.
\end{proof}

The next lemma allow to restrict ourselves to cases in which $\varphi$
is polynomial with degree independent of  $n$. 
\begin{lemma}\label{lemma:ApproxPhi}
Let $\varphi_A$, $\varphi_B$ be two activation functions, and denote by $m_{1,n}^A$, $m_{2,n}^A$, $m_{1,n}^B$, $m_{2,n}^B$, denote
the corresponding resolvents as defined above. Assume $\xi$ to be such that either $\xi\in\reals$, $\xi\le -\xi_0<0$, or 
$\Im(\xi)\ge \xi_0>0$. Then there exists a constant $C(\xi_0)$ such that, for all $n$ large enough
\begin{align}
\big|m_{1,n}^A(\xi)- m_{1,n}^B(\xi)\big|& \le C(\xi_0) \E\big\{[\varphi_A(G)-\varphi_B(G)]^2\big\}^{1/2}\, ,\\
\big|m_{2,n}^A(\xi)- m_{2,n}^B(\xi)\big|& \le C(\xi_0) \E\big\{[\varphi_A(G)-\varphi_B(G)]^2\big\}^{1/2}\, .
\end{align}
\end{lemma}
\begin{proof}
The proof is essentially the same as for Lemma 4.4 in \cite{cheng2013spectrum}.
\end{proof}

\begin{lemma}
Let $m_{1,n,p}(\xi)$ and $m_{2,n,p}(\xi)$ be the resolvent defined above where we made explicit the dependence upon the
dimensions $n$, $p$.  Assume $\xi$ to be such that either $\xi\in\reals$, $\xi\le -\xi_0<0$, or 
$\Im(\xi)\ge \xi_0>0$. Then, there exist $C=C(\xi_0)<\infty$ such that
\begin{align}
\big|m_{1,n,p}(\xi)-m_{1,n-1,p}(\xi)\big|+\big|m_{1,n,p}(\xi)-m_{1,n,p-1}(\xi)\big|& \le \frac{C}{n}\, ,\\
\big|m_{2,n,p}(\xi)-m_{2,n-1,p}(\xi)\big|+\big|m_{2,n,p}(\xi)-m_{2,n,p-1}(\xi)\big|& \le \frac{C}{n}\, .
\end{align}
(Here $d$ is understood to be the same in each case.)
\end{lemma}
\begin{proof}
This follows immediately from Lemma \ref{eq:BD-Diff}. Denote, with a slight abuse of notation, by $\bA(n,p)$ the matrix in 
Eq.~\eqref{eq:Adef}.
Consider for instance the difference 
\begin{align}
\big|m_{1,n,p}(\xi)-m_{1,n,p-1}(\xi)\big| &= \frac{1}{p}\left|\E \Tr_{[1,p]}\big[(\bA(n,p)-\xi\id)^{-1}\big]-\E \Tr_{[1,p]}\big[(\bA(n,p-1)-\xi\id)^{-1}\big]\right|\\
&= \frac{1}{p}\left|\E \left\{\Tr_{[1,p]}\big[(\bA(n,p)-\xi\id)^{-1}\big]-\Tr_{[1,p]}\big[(\bA^*(n,p)-\xi\id)^{-1}\big]\right\}\right|\, ,
\end{align}
where $\bA^*(n,p)$ is obtained from $\bA(n,p)$ by zero-ing the last row and column. Lemma \ref{eq:BD-Diff} then implied 
$|m_{1,n,p}(\xi)-m_{1,n,p-1}(\xi)|\le 3/\xi_0$. The other terms are treated analogously.
\end{proof}

\subsubsection{Derivation of Eqs.~(\ref{eq:M2})}

Throughout this appendix, we will assume $\Im(\xi)\ge C$ a large enough constant. This is sufficient by Lemma \ref{lemma:BasicProperties}.$(c)$. 
Also, we can restrict ourselves to $\varphi$ a fixed finite polynomial (independent of $n$). This is again sufficient by Lemma \ref{lemma:ApproxPhi}
(polynomials are dense in $L^2(\reals,\mu_G)$).

We denote by $\bA_{\dot,m}\in\reals^{N-1}$ the $m$-th column of $\bA$, with the $m$-th entry removed. By the Schur complement formula, we have
\begin{align}
m_{2,n} = \E\left\{\big(-\xi-\bA_{\cdot,n+p}^{\sT}(\bA_*-\xi\id_{N-1}) \bA_{\cdot,n+p}\big)^{-1}\right\}\, ,\label{eq:SchurM2}
\end{align}
where $\bA_*\in\reals^{(N-1)\times (N-1)}$ is the submatrix comprising the first $N-1$ rows and columns of $\bA$.

We let $\eta_a$ denote the projection of $\bw_a$ along $\bz_n$, and by $\tbw_a$ the orthogonal component. Namely we write
\begin{align}
\bw_a = \eta_a\frac{\bz_n}{\|\bz_n\|}+\tbw_a\, ,
\end{align}
where $\<\tbw_a,\bz_n\> = 0$.
We further let $\bX_*\in\reals^{(n-1)\times p}$ denote the submatrix of $\bX$ obtained by removing the last row.
With these notations, we have
\begin{align}
\bA_*&= \left[\begin{matrix}
s\id_p +t\bQ & \frac{1}{\sqrt{n}} \bX_*^{\sT}\\
\frac{1}{\sqrt{n}} \bX_* & \bfzero
\end{matrix}\right]\, ,\\
Q_{ab} &= \eta_a\eta_b \bfone_{a\neq b}+\<\tbw_a,\tbw_b\>\bfone_{a\neq b}\, ,\;\;\;\;\; 1\le a,b\le p,\\
X_{ja} & = \varphi\left(\<\tbw_a,\bz_j\> +\eta_a\frac{\<\bz_j,\bz_n\>}{\|\bz_n\|}\right)\, ,\;\;\;\;\; 1\le a\le p,\; 1\le j\le n-1\, ,
\end{align}
and
\begin{align}
\bA_{\cdot,n+p} &= \left[\begin{matrix}
\frac{1}{\sqrt{n}}\varphi(\|\bz_n\|\eta_1)\\
\vdots\\
\frac{1}{\sqrt{n}}\varphi(\|\bz_n\|\eta_p)\\
0\\
\vdots\\
0
\end{matrix}\right]\, ,
\end{align}

We next decompose
\begin{align}
X_{ja} & = \tX_{ja}+ a_{1,d}u_j\eta_a + \sqrt{n}\, E_{1,ja}\, ,\\
\tX_{ja}& \equiv \varphi\left(\<\tbw_a,\bz_j\>\right)\, ,\\
u_j & \equiv \frac{1}{\sqrt{n}}\frac{\<\bz_j,\bz_n\>}{\|\bz_n\|}\, ,\\
E_{1,ja} & \equiv \frac{1}{\sqrt{n}}\varphi_{\perp}\left(\<\tbw_a,\bz_j\> +\eta_a\frac{\<\bz_j,\bz_n\>}{\|\bz_n\|}\right)- \frac{1}{\sqrt{n}} \varphi_{\perp}\left(\<\tbw_a,\bz_j\>\right)\, .
\end{align}
and 
\begin{align}
Q_{ab} & = \tQ_{ab}+ \eta_a\eta_b + E_{0,ab}\, ,\\
\tQ_{ab} & \equiv \<\tbw_a,\tbw_b\>\, ,\\
E_{0,ab} &  \equiv  -\eta_a^2\bfone_{a=b}\, .
\end{align}
We therefore get
\begin{align}
\bA_* &= \tbA_* + \bDelta +\bE\, ,\\
\tbA_* & = \left[\begin{matrix}
s\id_p +t\tbQ & \frac{1}{\sqrt{n}} \tbX_*^{\sT}\\
\frac{1}{\sqrt{n}} \tbX_* & \bfzero
\end{matrix}\right]\, ,\;\;\;\;\;\;
\bDelta   = \left[\begin{matrix}
t\etab\etab^{\sT} & a_{1,d}\etab\bu^{\sT}\\
a_{1,d}\bu\etab^{\sT}& \bfzero
\end{matrix}\right]\, ,\;\;\;\;\;
\bE  = \left[\begin{matrix}
\bE_0 & \bE_1^{\sT}\\
\bE_1 & \bfzero
\end{matrix}\right]\, ,
\end{align}
It is possible to show that $\|\bE\|_{\sop}\le \eps_n \equiv (\log n)^M/n^{1/2}$ with probability at least $1-O(n^{-1})$, where $M$ is an absolute constant
(this can be done as in Section 4.3, Step 2 of \cite{cheng2013spectrum}, using intermediate value theorem). 
Further $\bDelta$ is a rank-$2$ matrix that can be written as $\bDelta = \bU\bC\bU^{\sT}$, where
$\bU\in\reals^{(N-1)\times 2}$ and $\bC\in\reals^{2\times 2}$  are given by
\begin{align}
\bU = \left[\begin{matrix}\etab &\bzero\\ \bzero & \bu \end{matrix}\right]\, ,\;\;\;\;\;\;\;\;\;\;
\bC=  \left[\begin{matrix}t & a_{1,d}\\ a_{1,d} & 0 \end{matrix}\right]\, .\label{eq:UC-def}
\end{align}
Using Eq.~\eqref{eq:SchurM2}, and Woodbury's formula, we get (writing for simplicity $\bv=\bA_{\cdot,n+p}$
\begin{align}
&m_{2,n} = \E\left\{\left(-\xi-\bv^{\sT}(\tbA_*+\bDelta+\bE-\xi\id_{N-1})^{-1} \bv\right)^{-1}\right\}\nonumber\\
&=  \E\left\{\left(-\xi-\bv^{\sT}(\tbA_*+\bDelta+\bE-\xi\id_{N-1})^{-1} \bv\right)^{-1}\bfone_{\|bE\|_{\op}\le\eps_n}\right\}\nonumber\\
&\phantom{AAAA}+\frac{C}{\Im(\xi)} O\left(\prob(\|bE\|_{\op}>\eps_n)\right) \label{eq:Woodbury}\\
&= \E\left\{\left(-\xi-\bv^{\sT}(\tbA_*+\bDelta-\xi\id_{N-1})^{-1} \bv\right)^{-1}\bfone_{\|bE\|_{\op}\le\eps_n}\right\}+O(\eps_n)\nonumber\\
& = \E\left\{\left(-\xi-\bv^{\sT}(\tbA_*-\xi\id_{N-1})^{-1} \bv+
\bv^{\sT} (\tbA_*-\xi\id_{N-1})^{-1} \bU \bS^{-1}\bU^{\sT}(\tbA_*-\xi\id_{N-1})^{-1} \bv
\right)^{-1}\right\}+O(\eps_n)\nonumber\\
&\bS \equiv \bC^{-1} +\bU^{\sT}(\tbA_*-\xi\id_{N-1})^{-1} \bU\, . \label{eq:Woodbury1}
\end{align}

Note that, conditional on $\|\bz_n\|_n$, $\bv$ is independent of $\tbA_*$ and has independent entries. Further its entries
are independent with $n\E\{v_i^2|\|\bz_n\|_2\} = n\E\{\varphi(\|\bz_n|_2G/\sqrt{d})^2|\|\bz_n\|_2\} = 1+O(n^{-1/2})$.
Hence 
\begin{align}
\E\big\{\bv^{\sT}(\tbA_*-\xi\id_{N-1})^{-1} \bv\big\} &= \frac{p}{n}\E\left\{\frac{1}{p}\Tr [(\tbA_*-\xi\id_{N-1})^{-1}]\right\}\\
& = \gamma \, m_{1,n}+ O(\eps_n)\, ,
\end{align}
By a concentration of measure argument (see, e.g., \cite[Section 2.4.3]{tao2012topics}),  the same statement also holds with high probability
\begin{align}
\bv^{\sT}(\tbA_*-\xi\id_{N-1})^{-1} \bv = \gamma \, m_{1,n}+ O_P(\eps_n)\label{eq:vAv}
\end{align}

We next consider $\bv^{\sT}(\tbA_*-\xi\id_{N-1})^{-1} \bU$. We decompose $\bv = \bv_{1}+\bv_2$, where 
\begin{align}
v_{1,i} = \frac{a_{1,d}\|\bz_n\|}{\sqrt{n}} \, \eta_i\, ,\;\;\;\; v_{2,i} = \frac{1}{\sqrt{n}} \, \varphi_{\perp}(\|\bz_n\|\eta_i)\, ,\;\;\; 1\le i\le p\, ,
\end{align}
Notice, that conditional on $\|\bz_n\|$, the pairs $\{(v_{1,i},v_{2,i})\}_{i\le p}$ are mutually independent.
Further $n\E\{v_{1,i}v_{2,i}|\|\bz_n\|\} = \E\{a_{1,d} G\varphi_\perp(G)\} +O(\eps_n) = O(\eps_n)$. Finally, again conditionally
on $\bz_n$, $\etab$ and $\bu$ are independent.
Therefore 
\begin{align}
\E\big\{\bv^{\sT}(\tbA_*-\xi\id_{N-1})^{-1} \bU\big\} &= \E\left\{\Big[\bv_1^{\sT}(\tbA_*-\xi\id_{N-1})^{-1}_{[1,p],[1,p]} \etab\Big|  0\Big]\right\}+O(\eps_n)\\
& = \frac{a_{1,d}\sqrt{d}}{\sqrt{n}}\left[\frac{1}{d}\E\Tr_{[1,p]}\big[ (\tbA_*-\xi\id_{N-1})^{-1}\big]\Big| 0 \right]+O(\eps_n)\\
& =\sqrt{c_1\gamma\psi^{-1}} \big[ m_{1,n}\, , \, 0\big] +O(\eps_n)\, .
\end{align}
By a concentration of measure argument, we also have
\begin{align}
\bv^{\sT}(\tbA_*-\xi\id_{N-1})^{-1} \bU =  \sqrt{c_1\gamma\psi^{-1}} \big[ m_{1,n}\, , \, 0\big] +O(\eps_n)\, .\label{eq:vAU}
\end{align}

We next consider $\bU^{\sT}(\tbA_*-\xi\id_{N-1})^{-1} \bU$. Since $\etab$ and $\bu$ are independent (and independent of $\tbA_*$) and zero mean, 
with $d\E\{\eta_a^2\} =1$, $n\E\{u_j^2\}= 1+O(\eps_n)$, we have
\begin{align*}
\E\big\{\bU^{\sT}(\tbA_*-\xi\id_{N-1})^{-1}& \bU\big\} =\\ 
=&\left[\begin{matrix}
\frac{1}{d}\E\Tr_{[1,p]} \{(\tbA_*-\xi\id_{N-1})^{-1}\} & 0\\
0 & \frac{1}{n}\E\Tr_{[p+1,N-1]} \{(\tbA_*-\xi\id_{N-1})^{-1}\} 
\end{matrix}\right] +O(\eps_n)\\
= &\left[\begin{matrix}
\psi^{-1}m_1 & 0\\
0 & m_2
\end{matrix}\right] +O(\eps_n) \, .
\end{align*}
As in the previous case, this estimate also holds for $\bU^{\sT}(\tbA_*-\xi\id_{N-1})^{-1} \bU$ (not just its expectation)
with high probability.  Substituting this in Eq.~\eqref{eq:Woodbury}, we get
\begin{align}
\bS = \left[\begin{matrix}
\psi^{-1}m_1 & 1/a_{1,d}\\
1/a_{1,d} & m_2-(t/a_{1,d}^2)
\end{matrix}\right] +O(\eps_n) 
\end{align}
By using this together with Eqs.~\eqref{eq:vAv} and \eqref{eq:vAU}, we get
\begin{align}
 m_{2,n} &= \left(-\xi-\gamma m_{1,n}+\frac{\gamma c_1m_{1,n}^2(c_1m_{2,n}-t)}{m_1(c_1m_{2,n}-t)-\psi}\right)^{-1}+O(\eps_n)\, .\label{eq:M2n}
\end{align}

\subsubsection{Derivation of Eqs.~(\ref{eq:M1})}

The derivation si analogous to the one in the previous section (notice that we will redefine some of the notations used in the last section),
and hence we will only outline the main steps.

Let $\bA_{\dot,m}\in\reals^{N-1}$ be the $m$-th column of $\bA$, with the $m$-th entry removed, and 
$\bB_*\in\reals^{(N-1)\times (N-1)}$ be the submatrix obtained by removing the $p$-th column and $p$-th row from $\bA$.
By the Schur complement formula, we have
\begin{align}
m_{1,n} = \E\left\{\big(-\xi-s-\bA_{\cdot,p}^{\sT}(\bB_*-\xi\id_{N-1}) \bA_{\cdot,p}\big)^{-1}\right\}\, ,\label{eq:SchurM2_2}
\end{align}
where $\bA_*\in\reals^{(N-1)\times (N-1)}$ is the submatrix comprising the first $N-1$ rows and columns of $\bA$.

We decompose $\bw_a$, $1\le a\le p-1$, and $\bz_i$, $1\le i\le n$ into their components along $\bw_p$, and the orthogonal
complement:
\begin{align}
\bw_a = \eta_a\frac{\bw_p}{\|\bw_p\|}+\tbw_a\, ,\;\;\;\;\;\; \bz_i = \sqrt{n}\, u_i\frac{\bw_p}{\|\bw_p\|}+\tbz_i\, ,
\end{align}
where $\<\tbw_a,\bw_p\> = \<\tbz_i,\bw_p\>=0$ (the factor $\sqrt{n}$ in the second expression is introduced for future
convenience). With these notations, we have
\begin{align}
\bB_*&= \left[\begin{matrix}
s\id_p +t\bQ_*  & \frac{1}{\sqrt{n}} \bX^{\sT}\\
\frac{1}{\sqrt{n}} \bX & \bfzero
\end{matrix}\right]\, ,\\
Q_{ab} &= \eta_a\eta_b \bfone_{a\neq b}+\<\tbw_a,\tbw_b\>\bfone_{a\neq b}\, ,\;\;\;\;\; 1\le a,b\le p-1,\\
X_{ja} & = \frac{1}{\sqrt{n}}\varphi\left(\<\tbw_a,\tbz_j\> +\sqrt{n}\, u_j\eta_a\right)\, ,\;\;\;\;\; 1\le a\le p,\; 1\le j\le n\, ,
\end{align}
and 
\begin{align}
\bA_{\cdot,p} =\left[\begin{matrix}
t\eta_1\|\bw_p\|\\
\vdots\\
t\eta_{p-1}\|\bw_p\|\\
\frac{1}{\sqrt{n}}\varphi(\sqrt{n}  u_1\|\bw_p\|)\\
\vdots\\
\frac{1}{\sqrt{n}}\varphi(\sqrt{n}  u_n\|\bw_p\|)
\end{matrix}\right]\equiv \bh\, .
\end{align}

We next decompose $\bB_*$ as follows 
\begin{align}
\bB_*&= \tbB_* + \bDelta+\bE\, ,\\
\tbB_*& \equiv \left[\begin{matrix}
s\id_p +t\tbQ_*  & \frac{1}{\sqrt{n}} \tbX^{\sT}\\
\frac{1}{\sqrt{n}} \tbX & \bfzero
\end{matrix}\right]\,,\;\;\;\;\;
\bDelta \equiv \left[\begin{matrix}
t\etab\etab^{\sT}  & a_{1,d}\etab\bu^{\sT}\\
 a_{1,d}\bu \etab^{\sT} & \bfzero
\end{matrix}\right]\,,\;\;\;\;\;
\bE \equiv \left[\begin{matrix}
\bE_0 & \bE_1^{\sT}\\
 \bE_1 & \bfzero
\end{matrix}\right]\,,
\end{align}
where we defined matrices $\tbQ_*$, $\tbX$, $\bE_0$, $\bE_1$ with the following entries:
\begin{align}
\tQ_{*,ab} &= \<\tbw_a,\tbw_b\>\bfone_{a\neq b}\, ,\;\;\;\;\;\;\; 1\le a,b\le p-1\, ,\\
\tX_{ja} & = \varphi(\<\tbw_a,\tbz_j\>) \, ,\;\;\;\;\;\;\; 1\le a\le p-1,\, 1\le j\le n\, ,\\
E_{0,ab}& =  -t\eta_a^2\bfone_{a=b}\, ,\;\;\;\;\;\;\; 1\le a,b\le p-1\, ,\\
E_{1,ja} & = \frac{1}{\sqrt{n}} \varphi_{\perp}\big(\<\tbw_a,\tbz_j\>+\sqrt{n}\, \eta_au_j\big) -\frac{1}{\sqrt{n}} \varphi_{\perp}\big(\<\tbw_a,\tbz_j\>\big) \, ,
\;\;\;\; 1\le a\le p-1,\, 1\le j\le n\, .
\end{align}
As in the previous section, it can be shown that $\|\bE\|_{\sop}\le
\eps_n \equiv(\log n)^c/\sqrt{n}$ (with $c$ an absolute constant), and therefore Eq.~\eqref{eq:SchurM2_2} yields 
\begin{align}
m_{1,n} = \E\left\{\big(-\xi-s-\bA_{\cdot,p}^{\sT}(\tbB_*+\bDelta-\xi\id_{N-1}) \bA_{\cdot,p}\big)^{-1}\right\}+O(\eps_n)\, .\label{eq:SchurM2_err}
\end{align}

Note that $\bDelta = \bU\bC\bU^{\sT}$, where $\bU\in\reals^{(N-1)\times 2}$ and $\bC\in\reals^{2\times 2}$ take the same form as
in Eq.~\eqref{eq:UC-def}.  Hence, by Woodbury's formula  and recalling
the notation $\bh = \bA_{\cdot,p}$ w
\begin{align}
m_{1,n}= &\E\left\{\left(-\xi-\bh^{\sT}(\tbB_*-\xi\id_{N-1})^{-1} \bh+
\bh^{\sT} (\tbB_*-\xi\id_{N-1})^{-1} \bU \bS^{-1}\bU^{\sT}(\tbB_*-\xi\id_{N-1})^{-1} \bh
\right)^{-1}\right\}\nonumber\\
&+O(\eps_n) \label{eq:M1Woodbury}\\
\bS \equiv& \bC^{-1} +\bU^{\sT}(\tbB_*-\xi\id_{N-1})^{-1} \bU\ ,.\nonumber
\end{align}
We then proceed to compute the various terms on the right-hand side. The calculation is very similar to the one in the previous section,
and we limit ourselves to reporting the results:
\begin{align}
\bh^{\sT}(\tB_*-\xi\id_{N-1})^{-1}\bh &= t^2\psi^{-1}m_{1,n}+m_{2,n} +O_P(\eps_n)\, ,\label{eq:M1hh}\\
\bh^{\sT}(\tB_*-\xi\id_{N-1})^{-1}\bU &= \left[t\psi^{-1}\, m_{1,n}, \, a_{1,d} m_{2,n}\right] +O_P(\eps_n)\, ,\\
\bU^{\sT}(\tB_*-\xi\id_{N-1})^{-1}\bU &= \left[\begin{matrix}
\psi^{-1}\, m_{1,n} & 0\\
0& m_{2,n}
\end{matrix}\right] +O_P(\eps_n)\, ,
\end{align}
whence
\begin{align}
\bS = \left[\begin{matrix}
\psi^{-1}\, m_{1,n} & a_{1,d}^{-1}\\
a_{1,d}^{-1}& -ta^{-2}_{1,d} +m_{2,n}
\end{matrix}\right] +O_P(\eps_n)\, .\label{eq:M1S}
\end{align}
Substituting Eqs.~\eqref{eq:M1hh} to \eqref{eq:M1S} into Eq.~\eqref{eq:M1Woodbury}, we get 
\begin{align}
m_{1,n}&= \left(-\xi-s-\frac{t^2}{\psi}m_{1,n}-m_{2,n}+\frac{t^2\psi^{-1}m_{1,n}^2(c_1m_{2,n}-t)-2tc_1m_{1,n}m_{2,n}+c_1^2m_{1,n}m^2_{2,n}}{m_{1,n}(c_1m_{2,n}-\psi)-\psi}\right)^{-1}+
O(\eps_n)\, .\label{eq:M1n}
\end{align}

\subsubsection{Completing the proof}

Let $\varphi_{L}$ be a degree $L=L(\eps)$ polynomial such that $\E\{|\varphi(G)-\varphi_L(G)|^2\}\le \eps^2$, for $G\sim\normal(0,1)$.
We will denote by $\bm^L_n = (m_{1,n}^L, m_{2,n}^L)$ the corresponding expected resolvents.

Consider $\Im(\xi)\ge C_{0}$ a large enough constant. By Eqs.~\eqref{eq:M2n}, \eqref{eq:M2n}, we have
\begin{align}
\bm_n^L = \bF(\bm_n^L) +o_n(1)\, .
\end{align}
By Lemma \ref{lemma:BasicProperties} and Lemma \ref{lemma:Contraction}, taking $C_0$ sufficiently large, we can ensure that $\bm_n^L\in \disk(r_0)\times\disk(r_0)$,
and that $\bF$ maps  $\disk(r_0)\times\disk(r_0)$ into  $\disk(r_0/2)\times\disk(r_0/2)$, with Lipschitz constant at most $1/2$.
Letting $\bm^L$ denote the unique solution of $\bm^L=\bF(\bm^L)$, we have
\begin{align}
\|\bm_n^L- \bm^L\|= \|\bF(\bm_n^L) +o_n(1)-\bF(\bm^L)\|\le \frac{1}{2}\|\bm_n^L- \bm^L\|+o_n(1) \, ,
\end{align}
whence $\bm^L_n(\xi)\to\bm^L(\xi)$ for all $\Im(\xi)>C_0$.
Since $eps$ is arbitrary, and using Lemma \ref{lemma:ApproxPhi}, we get $m_{1,n}(\xi)\to m_1(\xi)$, $m_{2,n}(\xi)\to m_2(\xi)$ for all $\Im(\xi)>C_0$.
 By Lemma \ref{lemma:BasicProperties}.$(c)$, we have $m_{1,n}(\xi)\to m_1(\xi)$, $m_{2,n}(\xi)\to m_2(\xi)$ for all $\xi\in\complex_+$. 
Finally, by Lemma \ref{lemma:Concentration}, the almost sure convergence of Eqs.~\eqref{eq:AS-1}, \eqref{eq:AS-2} holds as well.

\subsection{Proof of Lemma \ref{lemma:ResolventToVar-gen}}
\label{app:ResolventToVar}

We begin approximating the population covariance $\bSigma=\frac{1}{n}\E\{\bX^{\sT}\bX|\bw\}$ by $\bSigma_0\equiv \id_p+c_1\bQ\in\reals^{p\times p}$.
\begin{lemma}\label{lemma:Sigma-Sigma0}
With the above definitions there exists a constant $C$ such that 
\begin{align}
&\prob\big\{\|\bSigma-\bSigma_0\|_F\ge (\log n)^C\big\}\ge 1-e^{-(\log n)^2/C}, ,\label{eq:ProbF-diff}\\
&\E\{\|\bSigma-\bSigma_0\|_F^2\big\}\le (\log n)^C\, .\label{eq:ExpF-diff}
\end{align}
\end{lemma}
\begin{proof}
First notice that
\begin{align}
\Sigma_{ij} &= \E\{\varphi(\<\bw_i,\bz\>\varphi(\<\bw_j,\bz\>)|\bw_i,\bw_j\} = \E\{\varphi_{\|\bw_i\|}(G_1)\varphi_{\|\bw_j\|}(G_2)\}\, ,\\
\left(\begin{matrix}
G_1\\
G_2
\end{matrix}\right) &\sim \normal\left(\bfzero,\left[\begin{matrix}1 & s_{ij}\\
s_{ij}& 1\end{matrix}\right]\right)\, ,\\
\varphi_t(x) &\equiv \varphi(tx)\,,
\end{align}
where $s_{ij} = \<\bw_i,\bw_j\>/\|\bw_i\|\|\bw_j\|$.
Let   $\varphi_t(x) = \alpha_{1}(t) x+ \varphi_{t,\perp}(x)$ be the orthogonal decomposition of $\varphi_t$ in $L_2(\reals,\mu_G)$,
and notice that $\alpha_{1}(t) = \alpha_1(1)t$, $\alpha_1(1)^2 = c_1$. On the event $\cG = \{ |\|\bw_i\|-1| \le \eps_n\}$ (with $\eps_n=(\log n)^c/\sqrt{n}$), we obtain
\begin{align}
\Sigma_{ij} & = a_{1,d}(\|\bw_i\|) a_{1,d}(\|\bw_j\|) s_{ij} +\E\{\varphi_{\|\bw_i\|,\perp}(G_1)\varphi_{\|\bw_j\|,\perp}(G_2)\}\\
& = c_1\<\bw_i,\bw_j\> +O(\<\bw_i,\bw_j\>^2)\, .
\end{align}
Therefore, on the same event (for a suitable constant $C$)
\begin{align}
\big\|\bSigma-\bSigma_0\big\|_F^2  &= \sum_{i=1}^p(\Sigma_{ii}-\Sigma_{0,ii})^2+2 \sum_{i<j}^p(\Sigma_{ii}-\Sigma_{0,ii})^2\\
&\le p\, O(\eps^2_n) + C\sum_{i<j}\<\bw_i,\bw_j\>^4\, .
\end{align}
This implies Eq.~(\ref{eq:ProbF-diff}) because with the stated probability we have $\|\bw_i\|\le 1+\eps_n$ for all $i$ and
$|\<\bw_i,\bw_j\>|\le \eps_n$ for all $i\neq j$.

Equation (\ref{eq:ProbF-diff})  follows by the above, together with
the remark that  $\E(\|\bSigma-\bSigma_0\|_F^{2+c_0})\le p^{C_0}$ by
the assumption that $\E(x_{ij}^{4+c_0}) < \infty$.
\end{proof}

We have
\begin{align}
&\left|V_{\bX}(\hbeta_\lambda;\bbeta) -\frac{1}{n}\Tr\left\{(\hSigma+\lambda\id)^{-2}\hSigma\bSigma_0\right\}\right|=
\left|\frac{1}{n}\Tr\left\{(\hSigma+\lambda\id)^{-2}\hSigma(\bSigma-\bSigma_0\right\}\right|\\
&\le \frac{1}{n}\sqrt{p}\big\|(\hSigma+\lambda\id)^{-2}\hSigma\big\|_{\op}\|\bSigma-\bSigma_0\|_{F}\\
&\le \sqrt{\frac{\gamma}{n}}\, \frac{1}{\lambda}\, \|\bSigma-\bSigma_0\|_{F}\, .
\end{align}
And therefore, by Lemma \ref{lemma:Sigma-Sigma0}, we obtain
\begin{align}
\lim_{n\to\infty} \left|V_{\bX}(\hbeta_\lambda;\bbeta)-
  \frac{1}{n}\Tr\left\{(\hSigma+\lambda\id)^{-2}\hSigma\bSigma_0\right\}\right|
= 0\, ,\label{eq:VVV}
\end{align}
where the convergence takes place almost surely and in $L^1$. 

Denote by $(\lambda_i)_{i\le p}$, $(\bv_i)_{i\le p}$ the eigenvalues and eigenvectors of $\hSigma$. The following qualitative behavior can be 
extracted from the asymptotic of the Stieltjes transform as stated in Corollary \ref{coro:Stieltjes}.
\begin{lemma}\label{lemma:Spectrum}
For any $\gamma\neq 1$, $c_1\in[0,1)$, there exists $\rho_0>0$ such that the following happens.
Let $S_+ \equiv \{i\in [p] :\; |\lambda_i|>\rho_0\}$, $S_- \equiv \{i\in [p] :\; |\lambda_i|\le\rho_0\}$. Then, the following limits hold almost surely 
\begin{align}
&\lim_{n\to\infty}\frac{1}{n}|S_+|  = (\gamma \vee 1)\, ,\\
&\lim_{n\to\infty}\frac{1}{n}|S_-|  = (\gamma-1)_{+}\, ,\\
&\gamma>1\;\;\;\Rightarrow\;\;\; \frac{1}{|S_-|}\sum_{i\in S_-}\delta_{\lambda_i} \Rightarrow \delta_0\, .
\end{align}
(Here $\Rightarrow$ denotes weak convergence of probability measures.)
\end{lemma}

Note that
\begin{align}
m_n(\xi,s,t) \equiv \gamma \, m_{1,n}(\xi,s,t) +m_{2,n}(\xi,s,t) = \frac{1}{n}\E\Tr\big[(\bA(n)-\xi\id_N)^{-1}\big]\, .
\end{align}
Denote by $\tbSigma_0$ the $N\times N$ matrix whose
principal minor corresponding to the first $p$ rows and columns is given by $\bSigma_0$. By simple linear algebra (differentiation inside the integral is allowed for $\Im(\xi)>0$
by dominated convergence and by analyticity elsewhere), we get
\begin{align}
 -\partial_x m_n(\xi,x,c_1x)\big|_{x=0}&= \left. \frac{1}{n}\E\Tr\big[(\bA-\xi\id)^{-1}\tbSigma_0 (\bA-\xi\id)^{-1}\big]\right|_{x=0}\\
& = \frac{1}{n}\E\Tr\left[\left(\xi^2\id_p+\hSigma-4\xi^2(\hSigma+\xi^2\id_p)^{-1}\hSigma\right)^{-1}\bSigma_0\right]\\
& = \frac{1}{n}\E\Tr\left[\left(\hSigma +\xi^2\id_p\right)\left(\hSigma-\xi^2\id_p\right)^{-2}\bSigma_0\right]\, 
\end{align}
Note that $m_n(\xi,x,c_1x)\to m(\xi,x,c_1x)$ as $n\to\infty$. Further, for $\Im(\xi)>0$ or $\Re(\xi)<0$, it it is immediate to show 
tha $\partial^2_x m_n(\xi,x,c_1x)$ is bounded in $n$ (in a neighborhood of $x=0$). Hence
\begin{align}
\left. \lim_{n\to\infty} \partial_x m_n(\xi,x,c_1x)\right|_{x=0}= \left.\partial_x m(\xi,x,c_1x)\right|_{x=0} \equiv q(\xi)\, ,
\end{align}
 and therefore
\begin{align}
q(\xi)&=\lim_{n\to \infty}  \E Q_n(\xi) \, \\
Q_n(\xi)&= \frac{1}{n}\Tr\left[\left(\hSigma +\xi^2\id_p\right)\left(\hSigma-\xi^2\id_p\right)^{-2}\bSigma_0\right]\\
& = \frac{1}{n}\sum_{i=1}^p \frac{\lambda_i+\xi^2}{(\lambda_i-\xi^2)^2} \, \<\bv_i,\bSigma_0\bv_i\>\,.
\end{align}
Since the convergence of $M_n = \gamma M_{n,1}+M_{n,2}$
(cf. Eq.~\eqref{eq:Mdef})  is almost sure, we also have $Q_n(\xi)\to q(\xi)$ almost surely. Define
the probability measure on $\reals_{ge 0}$
\begin{align}
\mu_n = \frac{1}{p}\sum_{i=1}^n\delta_{\lambda_i} \<\bv_i,\bSigma\bv_i\>\, .
\end{align}
Since $Q_n(\xi)\to q(\xi)$, a weak convergence argument implies $\mu_n\Rightarrow \mu_{\infty}$ almost surely.
Further, defining $\mu^+_n = \bfone_{(\rho_0,\infty)}\mu_n$, $\mu^-_n = \bfone_{[0,\rho_0]}\mu_n$, Lemma \ref{lemma:Spectrum} implies
$\mu^+_n \Rightarrow \mu^+_{\infty}$, $\mu^-_n \Rightarrow  c_0\delta_0$, where $\mu^+_{\infty}$ is a measure supported on $[\rho_0,\infty)$,
with $\mu^+_{\infty}([\rho_0,\infty)) = 1-c$. This in turns implies
\begin{align}
q(\xi) = \frac{\gamma c}{\xi^2}+q_+(\xi)\,,\;\;\;\;\; q_+(\xi) = \gamma\int_{[\rho_0,\infty)} \frac{x+\xi^2}{(x-\xi^2)^2} \mu^+_{\infty}(\de x)\, ,
\end{align}
In particular, $q_+$ is analytic in a neighborhood of $0$. This proves Eq.~\eqref{eq:Laurent}, with  $q_+(0) = D_0$, $\gamma c=D_{-1}$.

Further, we have
\begin{align}
D_0 = q_+(0) &= \gamma\int_{[\rho_0,\infty)} \frac{x+\xi^2}{(x-\xi^2)^2} \mu^+_{\infty}(\de x)\, .
\end{align}
On the other hand by Eq.~\eqref{eq:VVV}
\begin{align}
\lim_{n\to\infty}V_{\bX}(\hbeta_\lambda;\bbeta) &= \lim_{n\to \infty}\frac{1}{n}\sum_{i=1}^p \frac{\lambda_i}{(\lambda_i+\lambda)^2} \, \<\bv_i,\bSigma_0\bv_i\>\\
& = \gamma \int \frac{x}{(x+\lambda)^2}  \mu_{\infty}(\de x) =\gamma \int \frac{x}{(x+\lambda)^2}  \mu^+_{\infty}(\de x)\, .
\end{align}
Comparing the last two displays, we obtain our claim
\begin{align}
\lim_{\lambda \to 0^+}\,\lim_{n\to\infty}\,V_{\bX}(\hbeta_\lambda;\bbeta) &= D_0\, .
\end{align}

\subsection{Proof of Corollary \ref{coro:Stieltjes-gen} and Corollary \ref{coro:Stieltjes}}
\label{app:CoroStieltjes}

Throughout this section,  set $s = t = 0$ (and drop these arguments for the various functions), and let $M_n(\xi) \equiv \gamma M_{1,n}(\xi)+M_{2,n}(\xi)$,
$m_n(\xi) \equiv \gamma m_{1,n}(\xi)+m_{2,n}(\xi)$, $m(\xi) = \gamma m_1(\xi)+m_2(\xi)$. In this case we have
\begin{align}
\bA= \left[\begin{matrix}
\bfzero & \frac{1}{\sqrt{n}} \bX^{\sT}\\
\frac{1}{\sqrt{n}} \bX & \bfzero
\end{matrix}\right]\, .\label{eq:Adef-s0}
\end{align}
and therefore, a simple linear algebra calculation yields
\begin{align}
M_n(z) = 2z\Big[\gamma\, S_n(z^2) +\frac{1}{2}(\gamma-1)\frac{1}{z^2}\Big]\, .
\end{align}
Theorem \ref{thm:var_nonlinear} immediately implies $S_n(z^2)\to s(z^2)$ (almost surely and in $L^1$), 
where
\begin{align}
m(z) = 2z\Big[\gamma\, s(z^2) +\frac{1}{2}(\gamma-1)_+\frac{1}{z^2}\Big]\, .\label{eq:MS}
\end{align}
Equations \eqref{eq:M2} and \eqref{eq:M1} simplify for the case $s=t=0$ (setting $\om_j=m_j(z,s=0,t=0)$) to yield 
\begin{align}
-z\om_1-\om_1\om_2+\frac{c_1^2 \om_1^2\om_2^2}{c_1\om_1\om_2-\psi} & =1\, ,\label{eq:OM1}\\
-z\om_2-\gamma\om_1\om_2+\frac{\gamma c_1^2 \om_1^2\om_2^2}{c_1\om_1\om_2-\psi} & =1\, .\label{eq:OM2}
\end{align}
Taking a linear combination of these two equations, we get
\begin{align}
-zm-2\gamma\om_1\om_2+\frac{2\gamma c_1^2 \om_1^2\om_2^2}{c_1\om_1\om_2-\psi} & =1+\gamma\, .
\end{align}
Comparing this with Eq~\eqref{eq:MS}, we get Eq.~\eqref{eq:coroS1}. Substituting the latter in Eqs.~\eqref{eq:OM1}, \eqref{eq:OM2},
we get Eqs.~\eqref{eq:coroS2}, \eqref{eq:coroS3}.

Finally,  Corollary \ref{coro:Stieltjes} follows by taking $c_1=0$.u

\subsection{Proof of Theorem \ref{cor:purely_nonlinear}}
\label{app:purely_nonlinear}

First notice that Lemma \ref{thm:ResolventKernel} and \ref{lemma:ResolventToVar}
follow by taking $t=0$, $c_1=0$ in Theorem \ref{thm:ResolventKernel-gen} and Lemma  \label{lemma:ResolventToVar-gen}.

The variance result follows simply by taking $c_1\to 0$ in Theorem \ref{thm:var_nonlinear}.
Solving the quadratic equations \eqref{eq:M1}, we get
\begin{align}
  m_2(\xi,s) & = \frac{\gamma-\xi^2-\xi s-1-\sqrt{(\gamma-xi^2-\xi s-1)^2-4\xi(\xi+s)}}{2\xi}\, ,\\
  \partial_x m_2(\xi,x)|_{x=0}& = -\frac{1}{2}+\frac{\gamma-\xi^2+1}{2 \sqrt{(\gamma-xi^2-\xi s-1)^2-4\xi(\xi+s)}}\, ,\\
  m_1(\xi,s) & = (-\xi-s-m_2(\xi,s))^{-1}\, ,\\
\partial_x m_1(\xi,x)|_{x=0}& = \frac{1+\partial_x m_2(\xi,x)|_{x=0}}{(\xi+m_2(\xi,0))^{2}}\, .
\end{align}
The claimed formula for the variance is obtained by using these expressions in Lemma \ref{lemma:ResolventToVar}.

For the bias, recall that
\begin{align}
B_{\bX}(\hbeta_{\lambda}) = \frac{r^2}{p}\Tr\left[\lambda^2(\hSigma_{\bX}+\lambda\id_p)^{-2}\bSigma\right]\, .
\end{align}
Define
\begin{align}
\tB_{\bX}(\hbeta_{\lambda}) = \frac{r^2}{p}\Tr\left[\lambda^2(\hSigma_{\bX}+\lambda\id_p)^{-2}\right]\, .
\end{align}
By Lemma \ref{lemma:Sigma-Sigma0}, and using the fact that $\bSigma_0=\id_p$ when $c_1=0$, we get
\begin{align}
\left|B_{\bX}(\hbeta_{\lambda}) - \tB_{\bX}(\hbeta_{\lambda}) \right| &\le 
\frac{r^2}{p}\|\lambda^2(\hSigma_{\bX}+\lambda\id_p)^{-2}\|_F\|\bSigma-\id_p\|_F\\
&\le 
\frac{C}{p}\|\lambda^2(\hSigma_{\bX}+\lambda\id_p)^{-2}\|_{\sop}\sqrt{p}(\log n)^C\\
& \le \frac{C}{\sqrt{n}}(\log n)^C\, .
\, .
\end{align}
with probability larger than $1-1/n^2$.
Therefore, by Borel-Cantelli it is sufficient to estabilish the claim for $\tB_{\bX}(\hbeta_{\lambda})$. 
By Corollary \ref{coro:Stieltjes}, for $c_1=0$, the empirical spectral distribution of $\hSigma$ converges almost surely (in the weak topology)
to the Marchenko-Pastur law $\mu_{MP}$. Hence
 (for $(\lambda_i)_{i\le p}$ the eigenvalues of $\hSigma$)
\begin{align}
\lim_{n\to\infty}\tB_{\bX}(\hbeta_{\lambda}) &= r^2\, 
\lim_{n\to\infty}\frac{1}{p}\sum_{i=1}^p\frac{\lambda^2}{(\lambda+\lambda_i)^2}\\
& = r^2 \int \frac{\lambda^2}{(\lambda+x)^2}\, \mu_{MP}(\de x) = r^2\lambda^2s'(-\lambda)\, .
\end{align}
Hence the asymptotic bias is the same as in the linear model (for random isotropic features). The claim hence follows by
the results of Section \ref{sec:risk_mis_iso}. Alternatively, we may simply recall  that $\mu_{MP}(\{0\}) = (1-\gamma^{-1})_+$ 
and use dominated convergence.

\bibliographystyle{plainnat}

\end{document}